\documentclass[10pt]{amsart}

\usepackage{mathtools}
\usepackage{amsthm}
\usepackage{mathrsfs} 
\usepackage{amssymb}
\usepackage[all]{xy}
\usepackage[nobottomtitles*]{titlesec}
\usepackage{titletoc}
\usepackage{bbm}
\usepackage{upgreek}
\usepackage{ marvosym}
\usepackage{multicol}
\usepackage{lscape}
\usepackage{ifluatex}
\usepackage[pdfencoding=auto, psdextra]{hyperref}
\usepackage[noabbrev,nameinlink]{cleveref}
\usepackage{enumerate}
\usepackage{soul}
\usepackage{stmaryrd}

\usepackage{verbatim}
\usepackage{tikz-cd}
\makeatletter
\tikzcdset{
    cong/.style={"\cong" {sloped, description, yshift=0pt,#1}, phantom},
    simeq/.style={"\simeq" {sloped, description, yshift=0pt,#1}, phantom},
    snake/.style={
        out= east, in=west,
        to path={
            \pgfextra{
                \pgfextractx{\pgf@xa}{\pgfpointanchor{\tikztostart}{east}}
                \pgfextractx{\pgf@xb}{\pgfpointanchor{\tikztotarget}{west}}
                \pgfextracty{\pgf@ya}{\pgfpointanchor{\tikztostart}{center}}
                \pgfextracty{\pgf@yb}{\pgfpointanchor{\tikztotarget}{center}}
                \edef\tikzstartx{\the\pgf@xa}
                \edef\tikzendx{\the\pgf@xb}
                \edef\midy{\the\dimexpr0.5\dimexpr\pgf@ya\relax +0.5\dimexpr\pgf@yb\relax}
            }
            to[in=0,out=180,looseness=0.5] (\tikzstartx,\midy)
            -| ([xshift=-2ex]\tikztotarget.west)
            -- (\tikztotarget)}
    }
}
\makeatother
\usepackage{sseq}
\usepackage{tikzcdintertext}
\usepackage{etoolbox}
\usepackage{xpatch}

\newcommand{\Mor}{{\sf Mor}}
\newcommand{\hMor}{{\sf \widehat{M}or}}
\newcommand{\Morone}{1\text{-}{\bf mor}}

\crefname{diagram}{diagram}{diagrams}
\creflabelformat{diagram}{#2(#1)#3}
\crefname{sseq}{}{}
\crefname{equation}{}{}
\crefname{equation-b}{Equation}{Equations}
\creflabelformat{sseq}{#2(#1)#3}
\usepackage[normalem]{ulem}
\newcommand{\stkout}[1]{\ifmmode\text{\sout{\ensuremath{#1}}}\else\sout{#1}\fi}

\makeatletter

\pretocmd{\maketitle}{%
    \edef\@title{\unexpanded{\protect\Large}\unexpanded\expandafter{\@title}}%
    \edef\authors{\unexpanded{\protect\normalsize}\unexpanded\expandafter{\authors}}%
}{}{\error}

\def\addlabeltolink#1{\addlabeltolink@#1}
\def\addlabeltolink@#1#2#3#4{#1{#2}{#3}{\thecontentslabel. #4}}

\dottedcontents{section}[3.8em]{}{2.3em}{1pc}
\dottedcontents{subsection}[3.8em]{\tiny}{2em}{1pc}

\titleformat{\section}[block]{\centering\bfseries\Large}{\thetitle. }{0pt}{}

\def\@secnumpunct{. }
\xpatchcmd{\proof}{\topsep6\p@\@plus6\p@\relax}{\topsep0pt\relax}{}{\error}
\makeatother
\usepackage[parfill]{parskip}

\allowdisplaybreaks[1]

\newcommand{\Cat}{\mathcal{C}\mathrm{at}}

\newcommand{\mr}[1]{\mathrm{#1}}
\usepackage[nobottomtitles*]{titlesec}
\usepackage{titletoc}

\newcommand{\br}[1]{\overline{#1}}

\newcommand{\td}[1]{\widetilde{#1}}
\newcommand{\wh}[1]{\widehat{#1}}

\renewcommand{\psi}{\uppsi}

\newcommand\Ocal{\mathcal{O}}

\newcommand{\BU}{\mr{BU}}

\newcommand{\BO}{\mr{BO}}

\newcommand{\Func}{{\bf Fun}}

\newcommand{\sfu}{{\sf u}}

\newcommand{\Rep}{\Rcal\mr{ep}}

\newcommand{\K}{\mr{K}}

\newcommand*{\rom}[1]{\expandafter\@slowromancap\romannumeral #1@} \makeatother

\hypersetup{%
  bookmarksnumbered=true,%
  bookmarks=true,%
  colorlinks=true,%
  pdfnewwindow=true,%
  pdfstartview=FitBH%
}

\hypersetup{
    colorlinks,
    linkcolor={DarkBrown},
    citecolor={DarkBlue},
    urlcolor={blue!80!black}
}

\def\@url#1{{\tt\def~{\lower3.5pt\hbox{\char'176}}\def\_{\char'137}#1}}

\newtheorem{main}{Main Theorem}

\makeatletter
\let\c@lemma\c@theorem
\makeatother

\newtheorem{thm}[equation]{Theorem}
\newtheorem{cor}[equation]{Corollary}
\newtheorem{lem}[equation]{Lemma}

\newtheorem{prop}[equation]{Proposition}

\newtheorem{conj}[equation]{Conjecture}
\newtheorem{claim}[equation]{Claim}
\theoremstyle{definition}
\newtheorem{defn}[equation]{Definition}

\newtheorem{ex}[equation]{Example}

\newtheorem{rmk}[equation]{Remark}

\newtheorem{notn}[equation]{Notation}

\newtheorem*{thm*}{Theorem}
\newtheorem*{cor*}{Corollary}
\newtheorem*{lem*}{Lemma}
\newtheorem*{prop*}{Proposition}
\newtheorem*{not*}{Notation}
\newtheorem*{guess*}{Guess}

\newtheorem*{defn*}{Definition}
\newtheorem*{ex*}{Example}
\newtheorem*{exs*}{Examples}
\newtheorem*{rmk*}{Remark}
\newtheorem*{claim*}{Claim}
\newtheorem*{exer*}{Exercise}
\newtheorem*{conv}{Convention}
\numberwithin{equation}{section}
\numberwithin{figure}{section}

\makeatletter
\let\c@lem=\c@thm
\let\c@cor=\c@thm
\let\c@prop=\c@thm
\let\c@lem=\c@thm
\let\c@ex=\c@thm
\let\c@exs=\c@thm
\let\c@obs=\c@thm
\let\c@rmk=\c@thm
\let\c@perthm=\c@thm
\let\c@conjtel=\c@thm
\let\c@exmps=\c@thm
\let\c@rem=\c@thm
\let\c@question=\c@thm
\let\c@warn=\c@thm
\let\c@claim=\c@thm
\let\c@quest=\c@thm
\let\c@notation=\c@thm
\let\c@note=\c@thm
\let\c@conjtel=\c@thm
\let\c@gue=\c@thm
\let\c@goal=\c@thm
\makeatother

\usepackage[T1]{fontenc}
\makeatletter
\newif\ifkp@upRm
\DeclareSymbolFont{Letters}{OML}{jkp}{m}{n}
\DeclareMathSymbol{\partialup}{\mathord}{Letters}{128}

\DeclareMathOperator{\Hom}{Hom}

\DeclareMathOperator{\im}{im}

\DeclareMathOperator*{\holim}{{\bf holim}}
\DeclareMathOperator*{\hocolim}{{\bf hocolim}}

\DeclareMathOperator{\Sp}{\mathcal{S}p}

\DeclareMathOperator{\Top}{\mathcal{T}op}

\DeclareMathOperator{\Th}{{\sf Th}} 

\DeclareMathOperator{\res}{{\sf res}}

\def\makecommands#1#2#3{
    \bgroup
    \def\tempcmdname##1{#1}
    \def\tempcmdbody##1{#2}
    \def\\##1{\expandafter\xdef\csname\tempcmdname{##1}\endcsname{\unexpanded\expandafter{\tempcmdbody{##1}}}}
    #3
    \egroup
}
\def\upperalphabet{\\A\\B\\C\\D\\E\\F\\G\\H\\I\\J\\K\\L\\M\\N\\O\\P\\Q\\R\\S\\T\\U\\V\\W\\X\\Y\\Z}
\def\loweralphabet{\\a\\b\\c\\d\\e\\f\\g\\h\\i\\j\\k\\l\\m\\n\\o\\p\\q\\r\\s\\t\\u\\v\\w\\x\\y\\z}
\def\lowergreekalphabet{\\\alpha\\\beta\\\gamma\\\delta\\\epsilon\\\zeta\\\eta\\\theta\\\kappa\\\lambda\\\mu\\\nu
    \\\xi\\\pi\\\rho\\\sigma\\\tau\\\upsilon\\\psi\\\chi\\\phi\\\omega}

\makecommands{#1mr}{\mathrm{#1}}{\upperalphabet}
\makecommands{#1cal}{\mathcal{#1}}{\upperalphabet}
\makecommands{#1#1}{\mathbb{#1}}{\upperalphabet}
\makecommands{#1bar}{\overline{#1}}{\upperalphabet}
\makecommands{#1frak}{\mathfrak{#1}}{\upperalphabet \loweralphabet}
\makecommands{#1scr}{\mathscr{#1}}{\upperalphabet \loweralphabet}
\makecommands{#1twee}{\widetilde{#1}}{\upperalphabet\loweralphabet}
\makecommands{sf#1}{{\sf #1}}{\upperalphabet\loweralphabet}
\makeatletter
\makecommands{\expandafter\@gobble\string#1bar}{\overline{#1}}{\lowergreekalphabet}
\makecommands{\expandafter\@gobble\string#1twee}{\widetilde{#1}}{\lowergreekalphabet}
\makeatother

\newcommand{\sma}{\wedge}

\definecolor{limegreen}{rgb}{0.2, 0.8, 0.2}
\definecolor{darkmagenta}{rgb}{0.55, 0.0, 0.55}
\definecolor{lavenderrose}{rgb}{0.91, 0.33, 0.5}
\definecolor{goldenpoppy}{rgb}{0.99, 0.76, 0.0}
\definecolor{seagreen}{rgb}{0.1, 0.4, 0.1}
\definecolor{maroon}{RGB}{128,0,0}
\definecolor{darkviolet}{RGB}{148,0,211}
\definecolor{darkbrown}{rgb}{.55,.3,.2}
\definecolor{snowblue}{rgb}{0.3, 0.4, 0.8}
\definecolor{DarkBlue}{rgb}{.1, 0.35, 0.6} 
\definecolor{DarkBrown}{rgb}{.5, 0.2, 0.2}

\makeatletter

\def\tikzcdequalsignoffset{0.1em}

\def\findedgesourcetarget#1#2{
    \let\sourcecoordinate\pgfutil@empty
    \ifx\tikzcd@startanchor\pgfutil@empty 
        \def\tempa{\pgfpointanchor{#1}{center}}
    \else
        \edef\tempa{\noexpand\pgfpointanchor{#1}{\expandafter\@gobble\tikzcd@startanchor}} 
        \let\sourcecoordinate\tempa
    \fi
    \ifx\tikzcd@endanchor\pgfutil@empty 
        \def\tempb{\pgfpointshapeborder{#2}{\tempa}}
    \else
        \edef\tempb{\noexpand\pgfpointanchor{#2}{\expandafter\@gobble\tikzcd@endanchor}}
    \fi
    \let\targetcoordinate\tempb
    \ifx\sourcecoordinate\pgfutil@empty%
        \def\sourcecoordinate{\pgfpointshapeborder{#1}{\tempb}}%
    \fi
}

\tikzset{/tikz/commutative diagrams/equal/.style=equals,
    /tikz/commutative diagrams/equals/.style = {
    -,
    to path={\pgfextra{
        \findedgesourcetarget{\tikzcd@ar@start}{\tikzcd@ar@target} 
        \ifx\tikzcd@startanchor\pgfutil@empty
            \def\tikzcd@startanchor{.center}
        \fi
        \ifx\tikzcd@endanchor\pgfutil@empty
            \def\tikzcd@endanchor{.center}
        \fi
        \pgfmathanglebetweenpoints{\pgfpointanchor{\tikzcd@ar@start}{\expandafter\@gobble\tikzcd@startanchor}}{\pgfpointanchor{\tikzcd@ar@target}{\expandafter\@gobble\tikzcd@endanchor}}
        \pgftransformrotate{\pgfmathresult}
        \pgfpathmoveto{\pgfpointadd{\sourcecoordinate}{\pgfpoint{0}{\tikzcdequalsignoffset}}}
        \pgfpathlineto{\pgfpointadd{\targetcoordinate}{\pgfpoint{0}{\tikzcdequalsignoffset}}}
        \pgfpathmoveto{\pgfpointadd{\sourcecoordinate}{\pgfpoint{0}{-\tikzcdequalsignoffset}}}
        \pgfpathlineto{\pgfpointadd{\targetcoordinate}{\pgfpoint{0}{-\tikzcdequalsignoffset}}}
        \pgfusepath{draw}
}}}}

\makeatother

\title{Equivariant Weiss Calculus}
\author{Prasit~Bhattacharya}\address{New Mexico State University}\email{prasit@nmsu.edu}
\author{Yang Hu}\address{New Mexico State University}\email{yanghu@nmsu.edu}

\thanks{}

\begin{document}

\begin{abstract} In this paper, we introduce an equivariant analog of Weiss calculus of functors for all finite group $\mathrm{G}$. In our theory,  Taylor approximations and derivatives are index by finite dimensional $\mathrm{G}$-representations, and homogeneous layers are classified by  orthogonal $\mathrm{G}$-spectra. Further, our framework permits a notion of restriction as well as a notion of fixed-point at the level of Weiss functors. We  establish various results  comparing Taylor approximations and derivatives of  fixed-point (resp. restrictions) functors to that of the fixed-point (resp. restrictions) of Taylor approximations and derivatives.  \end{abstract}

\maketitle	
\tableofcontents

\section{Introduction} \label{Sec:intro}
The orthogonal calculus of functors, also known as the Weiss calculus, was introduced in 1995 by Michael Weiss \cite{WeissCalc}  which constructs a  tower of functors \
\begin{equation} \label{Taylor1}
\begin{tikzcd}
\dots \rar  & \Tmr_n \Emr \rar & \dots \rar & \Tmr_1 \Emr \rar \rar & \Tmr_0 \Emr
\end{tikzcd}
\end{equation}
associated to a functor $\Emr: \Jcal^\RR \longrightarrow \Top_*$, where $\Jcal^\RR$ is the category of finite dimensional real vector spaces equipped with positive definite inner products.  A fundamental feature of the Taylor approximation is that the fiber $ \Lmr_n\Emr := \mr{Fiber}(\Tmr_n \Emr \to \Tmr_{n-1} \Emr)$,  which is also called the $n$-th layer of $\Emr$, is a functor 
\begin{equation} \label{eqn:layer}
\begin{tikzcd}
 \Lmr_n\Emr : \Jcal^\RR \rar & \Top_*
 \end{tikzcd}
 \end{equation}
which sends each vector space $\Vmr$ to an infinite loop space. Thus, Weiss calculus facilitates the usage of  tools from stable homotopy theory
to study unstable problems.

 Historically, there are many interesting and important problems in geometry which can be reformulated as a problem in unstable homotopy theory.  In the last three decades, functors extremely important in geometry, such as 
 \begin{equation} \label{eq:exF}
  \text{$\BO(-)$, $\BU(-)$, $\mr{BDiff}(-)$, $\mr{BTop}(-)$},
  \end{equation}
   among others,  have been studied extensively leading to many remarkable applications \cite{WW88, WW89, Arone98, Arone02, Arone09, WW14, RW16, Weiss21, RandalDisc, GR23, Yang23, KRW, CHO}. 
Notice that these functors readily extends to the category of orthogonal representations for any finite group $\Gmr$ (see \Cref{ex:eqWeissF}) which are of crucial in the study of equivariant geometry. Thus, it is desirable to have an equivariant analog of Weiss calculus that can  facilitate the study of such functors through equivariant stable homotopy theory.

 Despite the demand, the development of equivariant Weiss calculus eluded the literature until the  thesis work of Tynan \cite{Ty17}. In 2016, Tynan introduced a version of $\Cmr_2$-equivariant unitary Weiss calculus  using Galois action on $\CC$ which  was further studied and developed by  N. Taggart  in \cite{Taggart22}. They considered a Weiss theory associated to $\Cmr_2$-equivariant functors from $\Jcal^{\CC/\RR}$ (the category of complex vector spaces with Galois action of $\Cmr_2$) to $\Cmr_2\Top$, and  showed that the homogeneous layers are genuine $\Cmr_2$-equivariant infinite loop spaces. In 2024, Emel Yuvuz \cite{Y24} introduced a $\Cmr_2$-equivariant orthogonal calculus  which is drastically different than \cite{Ty17, Taggart22}. A distinct feature of the Weiss theory in \cite{Y24} is that the Taylor approximations and derivatives \cite{Y24} are indexed by finite dimensional $\Cmr_2$-representations.

In  this paper,  we introduce an equivariant analog of  orthogonal Weiss calculus for all finite groups $\Gmr$. In our theory, the  Taylor approximations and derivatives are also index by finite dimensional $\mathrm{G}$-representations, and homogeneous layers are classified by  orthogonal $\mathrm{G}$-spectra\footnote{In this sense,  our  work maybe regarding as an extension of  \cite{Y24} to all finite groups.}. Our theory can be adjusted to  obtain the unitary equivariant Weiss calculus (see \Cref{Weissunitary}),  as well as  an `Atiyah Real analog'  comparable to  \cite{Ty17, Taggart22}  (see \Cref{WeissAR}).   In addition, our theory allows for a notion of  fixed-point and restriction at the level of Weiss functors. We  establish various comparison theorems  comparing Taylor approximations and derivatives  of  fixed-point (resp. restrictions) functors to that of the fixed-point (resp. restrictions) of Taylor approximations and derivatives. 
 
\bigskip
 \begin{notn} \label{notn:set1} Throughout this paper: 
 \begin{itemize}
 \item Let $\Gmr$  be a  finite group, $e \in \Gmr$ be the  identity element, and $\Ocal_{\Gmr}$ be  the orbit category of $\Gmr$. 
  \item For a subgroup $\Hmr$ of $\Gmr$, let $\Nmr(\Hmr) \subset \Gmr$ denote the normalizer group  and  \[ \Wmr(\Hmr) := \Nmr(\Hmr)/\Hmr\] denote the  Weyl group,  of $\Hmr$ in $\Gmr$.   
 \item Let $\Cat$ be the category of small categories.
  \item $\Ucal_\Gmr$ be a complete $\Gmr$ universe, $\Ucal_{\Gmr}^{(n)}$ be the subuniverse of $\Ucal_\Gmr$ generated by $n$-dimensional $\Gmr$-representations, and in general $\Ucal^{(\sfR)}_\Gmr$ be the subuniverse generated by the set of $\Gmr$-representations $\sfR$.
 \item Let $\Rcal\mr{ep}_{\Gmr}$ be the category of finite dimensional orthogonal $\Gmr$-representation  in $\Ucal_\Gmr$ in  which the morphisms space  $\Mor_{\Gmr}(\Wmr, \Vmr)$, is the $\Gmr$-space of linear isometric embeddings where $\Gmr$ acts by  conjugation.
 \item $\Gmr\Top$ be the closed symmetric monoidal category of pointed $\Gmr$-spaces, and let $\Gmr\Top_*$ denote the corresponding pointed category. 
 \item For a $\Wmr \in \Rep_\Gmr$, let  $|\Wmr|$ denote the real  dimension of $\Wmr$.  
 \item $\Sp^\Gmr$ be the category of genuine orthogonal $\Gmr$-spectra,  $\Sp_{(n)}^\Gmr$ be  the orthogonal $\Gmr$-spectra in the $\Gmr$-universe $\Ucal^{(n)}_\Gmr$, and  $\Sp_{(\sfR)}^\Gmr$ be  the orthogonal $\Gmr$-spectra in the $\Gmr$-universe $\Ucal^{(\sfR)}_\Gmr$ (as in \cite{MayMandell}). 
 \item We let 
 \[ 
 \begin{tikzcd}
 \Sigma^{\infty}_{\sfR}:  \Gmr\Top_*  \rar[shift right] & \lar[shift right]\Sp^\Gmr_{(\sfR)} : \Omega^{\infty}_{\sfR}
 \end{tikzcd}
 \]
 denote the usual loop-suspension adjunction restricted the universe $\Ucal^{(\sfR)}_\Gmr$, where we drop the subscript when the universe is complete. 
 
  \item We will use  $\upiota_{\Hmr}(-)$ to denote the $\Hmr$-fixed points functor  of $\Rep_\Gmr$, $\Gmr\Top_*$,  $\Sp^\Gmr$,  $\Sp_{(n)}^\Gmr$,  and $\Sp_{(\sfR)}^\Gmr$ for any subgroup $\Hmr \subset \Gmr$ without distinction. 
 
 \item $(- )^{\Hmr}: \Gmr\Top_* \longrightarrow \Wmr(\Hmr)\Top_*$ denote the $\Hmr$-fixed point functor. We use the same notation to denote the $\Hmr$-fixed point functor of $\Rep_\Gmr$ as well as of all of the above categories of $\Gmr$-spectra. 
 \end{itemize}
 \end{notn}

The main idea in this paper is to replace  $\Jcal^\RR$  by the  presheaf
 $
 \Jfrak_{\Gmr}:  \Ocal_{\Gmr}^{\mr{op}} \longrightarrow \Cat
 $
 which sends $\Gmr/\Hmr$ to   $\Rcal\mr{ep}_{\Hmr}$,  $\Top_*$ by the presheaf 
 $
 \Tfrak_{\Gmr}:  \Ocal_{\Gmr}^{\mr{op}} \longrightarrow \Cat
 $
 which sends $\Gmr/\Hmr$ to $\Gmr\Top_{ \ast}$, and  consider a natural transformation 
 \begin{equation} \label{natE}
  \begin{tikzcd}
  \Emr: \Jfrak_{\Gmr} \rar&  \Tfrak_{\Gmr} 
  \end{tikzcd}
  \end{equation}
  which we call a  {\bf $\Gmr$-equivariant Weiss functor}\footnote{We think of $\Emr$ in \eqref{natE} as  $1$-morphisms in the functor category ${\bf Fun}(\Ocal^{\mr{op}}_{\Gmr}, \Cat)$.}.

 For every $\Gmr$-equivariant Weiss functor, we build an equivariant Taylor approximation (see  \Cref{defn:Taylor})  index by a pair  of $\Gmr$-representations $(\upalpha, \upepsilon)$,  where $\upepsilon$ is a $1$-dimensional which plays the role of a 
  {\bf direction}\footnote{It is reasonable to  extend the analogy between classical Weiss theory  and one-variable calculus  to equivariant Weiss theory by comparing it  with multivariate calculus. } (see \Cref{defn:dir}). 
  \begin{main} \label{main:tower}Let $\upalpha$ be a $\Gmr$-representation and $\upepsilon$ denote a $1$-dimensional $\Gmr$-representation. Then corresponding to every  $1$-morphism  
  \[ 
 \begin{tikzcd}
 \Emr: \Jfrak_{\Gmr} \rar & \Tfrak_{\Gmr}
 \end{tikzcd}
\]
 in ${\bf Fun}( \Ocal_{\Gmr}^{\mr{op}}, \Cat)$, there exists a morphism 
\[ 
\begin{tikzcd}
{\bf T}_{(\upalpha, \upepsilon)} \Emr: \Jfrak_{\Gmr} \rar & \Tfrak_{\Gmr},
 \end{tikzcd}
\]
called the $\upalpha$-th Taylor approximation of $\Emr$ in the direction of $\upepsilon$, 
together with 2-morphisms
\begin{enumerate}[$(i)$]
\item $
\begin{tikzcd}
\upeta_{\alpha, \upepsilon}^{\Emr}:  \Emr \rar & {\bf T}_{(\upalpha, \upepsilon)}\Emr,
\end{tikzcd}
$
\item 
$ 
\begin{tikzcd}
\upeta^{\Emr}_{\upalpha, \upalpha', \upepsilon}: {\bf T}_{(\alpha, \upepsilon)} \Emr \rar & {\bf T}_{(\alpha', \upepsilon)}\Emr
\end{tikzcd}
$
for a sub-repesentation $\upalpha' \subset \upalpha$,  
\end{enumerate}
such that 
\[  \upeta^{\Emr}_{\upalpha', \upepsilon} = \upeta^{\Emr}_{\upalpha, \upalpha', \upepsilon} \circ \upeta^{\Emr}_{\upalpha, \upepsilon}\]
\[ \upeta^{\Emr}_{\upalpha'', \upalpha', \upepsilon} \circ  \upeta^{\Emr}_{\upalpha, \upalpha', \upepsilon}  = \upeta^{\Emr}_{\upalpha, \upalpha'',\upepsilon} \]
whenever $\upalpha'' \subset \upalpha' \subset \upalpha$ as $\Gmr$-representations. 
 \end{main}
 \begin{notn} For an equivariant Weiss functor $\Emr$ (as in \eqref{natE}), we let 
 \[ 
\begin{tikzcd}
\Emr_{\Gmr/\Hmr}: \Rep_{\Hmr} \rar & \Hmr\Top_*, 
\end{tikzcd}
\]
denote the functor at  $\Gmr/\Hmr$. 
 \end{notn}
 \begin{ex} \label{ex:eqWeissF} Some of standard examples of equivariant Weiss functors include 
 \begin{itemize}
 \item ${\bf S}_\Gmr: \Jfrak_\Gmr \longrightarrow \Tfrak_\Gmr$, where  $({\bf S}_\Gmr)_{\Gmr/\Hmr}$ sends an $\Hmr$-representation $\Wmr$ to its one-point compactification $\Smr^{\Wmr}$.
 \item  ${\bf B}_\Gmr{\bf O}: \Jfrak_\Gmr \longrightarrow \Tfrak_\Gmr $, where $({\bf B}_\Gmr{\bf O})_{\Gmr/\Hmr}$ sends an $\Hmr$-representation $\Wmr$ to $\Bmr_\Hmr\Omr(\Wmr)$ (see \Cref{bpBGO}).
  \item   ${\bf B}_\Gmr{\bf TOP}: \Jfrak_\Gmr \longrightarrow \Tfrak_\Gmr $, which sends an  $\Hmr$-representation $\Wmr$ to the classifying space of origin preserving homeomorphism (nonequivariant) of $\Wmr$ on which $\Hmr$ acts by conjugation.
 \item   ${\bf B}_\Gmr{\bf Diff}: \Jfrak_\Gmr \longrightarrow \Tfrak_\Gmr $, which sends an  $\Hmr$-representation $\Wmr$ to the  classifying space of origin preserving diffeomorphisms (nonequivariant) of $\Wmr$ on which $\Hmr$ acts by conjugation.
 \item ${\bf Emb}^{\Gmr}_{\Mmr, \Nmr} : \Jfrak_\Gmr \longrightarrow \Tfrak_\Gmr$ for $\Gmr$-manifolds $\Mmr$ and $\Nmr$, which sends an $\Hmr$-representations $\Wmr$ to the classifying  space of embeddings (nonequivariant) of $\Mmr$ into $\Nmr \times \Vmr$ (with a disjoint base point).  
 \end{itemize}
 \end{ex}
 \begin{rmk} \label{bpBGO} For a $k$-dimensional orthogonal $\Hmr$-representation $\Wmr$, let $\Omr(\Wmr)$ denote the orthogonal group $\Omr(k)$ with the conjugation action of $\Hmr$ induced from $\Wmr$.  
  The bar construction on the set of functions from $\Hmr$ to $\Omr(\Wmr)$ is an $\Hmr$-space 
  \[ \Bmr_{\Hmr}\Omr(\Wmr):= \Bmr{\bf Func}(\Hmr, \Omr(\Wmr)), \]
whose homotopy type does not depend on $\Wmr$ as long as $|\Wmr| = k$. However, the natural base point of $\Bmr_{\Hmr}\Omr(\Wmr)$ depends on $\Wmr$.  According to \cite[Theorem 10]{LM} (also see \cite{FolingClassify} for details), the path components of $\Hmr$-fixed points of $\Bmr_{\Hmr}\Omr(\Wmr)$
  \[ 
 \Bmr_{\Hmr}\Omr(\Wmr)^{\Hmr}  \cong \underset{ \left\lbrace \substack{ [\Wmr'] \in {\sf Obj}(\mr{Rep}_\Hmr)/\cong \\ \dim_\RR(\Wmr') =k } \right\rbrace } \coprod \Bmr \Omr(k)
 \]
 are in one-to-one correspondence with the isomorphic classes of $k$-dimensional orthogonal $\Hmr$-representations, and the  base point of $\Bmr_{\Hmr}\Omr(\Wmr)$ naturally lies  in the component index by  $[\Wmr]$. 
 \end{rmk}
 \begin{defn} \label{defn:pwequiv} We say a $2$-morphism $\upeta: \Emr \longrightarrow \Fmr$ between two $1$-morphisms
 \[ 
 \begin{tikzcd}
 \Emr, \Fmr: \Jfrak_{\Gmr} \rar & \Tfrak_{\Gmr}
 \end{tikzcd}
\]
 in ${\bf Fun}( \Ocal_{\Gmr}^{\mr{op}}, \Cat)$,
 a {\bf pointwise equivalence} if  
\[ 
\begin{tikzcd}
\upeta_{\Wmr}: \Emr_{\Gmr/\Hmr}(\Wmr) \rar & \Fmr_{\Gmr/\Hmr}(\Wmr)
\end{tikzcd}
\]
is a weak equivalence in $\Hmr\Top_*$ for all $\Wmr \in \Rep_\Hmr$ and subgroup $\Hmr$ of $\Gmr$. 
\end{defn}
In the nonequivariant case, the $n$-th Taylor approximation $\Tmr_n\Emr$ are examples of $n$-polynomial functors (see \cite[Definition 5.1]{WeissCalc}). We generalize this notion in  \Cref{defn:polynomial} and show that  ${\bf T}_{(\alpha, \upepsilon)}\Emr$ is an example of  {\bf $\alpha$-polynomial functor in the direction of $\upepsilon$}. Further:
\begin{main} \label{main:poly} If $ \begin{tikzcd}
 \Emr: \Jfrak_{\Gmr} \rar & \Tfrak_{\Gmr}, 
 \end{tikzcd}
$ is $\alpha$-polynomial  in the direction of $\upepsilon$ then for any $\Hmr \subset \Gmr$
\[ 
\begin{tikzcd}
\upeta_{\upalpha \oplus \upalpha', \upepsilon}^\Emr:  \Emr \rar & {\bf T}_{(\upalpha \oplus \upalpha', \upepsilon )}\Emr
\end{tikzcd}
\] 
is a pointwise equivalence for all $\upalpha' \in \Rep_\Gmr$ which are isomorphic to a finite sum of one dimensional $\Gmr$-representations. 
\end{main}
In the nonequivariant case, Weiss constructed the category $\Jcal_n^\RR$, for each natural number $n$, whose objects are that of $\Jcal^\RR$, i.e. finite dimensional real vector spaces with inner products, but morphism spaces are Thom spaces of certain tautological bundles so that we have a sequence of faithful functors 
\[ 
\begin{tikzcd}
\Jcal^\RR  \rar[hook] & \cdots \rar[hook] & \Jcal_n^\RR \rar[hook] & \Jcal_{n+1}^\RR \rar[hook] & \cdots .
\end{tikzcd}
\]
Then the $n$-th derivative of a functor $\mr{E} : \Jcal \longrightarrow \Top_*$ is defined  as the functor 
\[ 
\begin{tikzcd}
\Emr^{(n)}: \Jcal_n^\RR \rar & \Top_*
\end{tikzcd}
\]
obtained by right Kan extending $\Emr$ to $\Jcal_n^\RR$ (see \cite[$\mathsection$2]{WeissCalc}).  Associated to the $n$-th derivative functor $\Emr^{(n)}$, Weiss defined a spectrum $\uptheta\Emr^{(n)}$, called the {\bf $n$-th derivative of $\Emr$ at $\infty$}, that  admits a n\"aive action of $\mr{O}(n)$. One of the fundamental results of Weiss calculus claims
\[ 
\Lmr_n\Emr (\Vmr) \simeq \Omega^{\infty}\left( \Smr^{n \Vmr} \sma \uptheta\Emr^{(n)}\right)_{\sfh \Omr(n)}
\]
for all $\Vmr \in \Jcal^\RR$, i.e.,  calculation of Weiss layers is a stable homotopy theoretic problem and can be understood in terms of the derivatives at $\infty$.   

 In this paper we construct, for each $\upalpha \in \Rcal\mr{ep}_{\Gmr}$,  a functor 
\[
\begin{tikzcd}
\Jfrak^{\alpha}_\Gmr: \Ocal^{\mr{op}}_{\Gmr} \rar & \Cat
\end{tikzcd}
\]
  such that there is natural transformation 
 \[
 \begin{tikzcd}
  \iota_{\upalpha, \upalpha'}: \Jfrak^{\alpha}_\Gmr \rar & \Jfrak^{\upalpha'}_\Gmr 
  \end{tikzcd}
  \]
  whenever $\alpha$ is a sub $\Gmr$-representation of $\alpha'$, satisfying 
  \[ \iota_{\upalpha'', \upalpha'} \circ  \iota_{\upalpha, \upalpha'}  = \iota_{\upalpha, \upalpha''}\]
  for any triplet $\upalpha \subset \upalpha' \subset \upalpha''$. The category $\Jfrak^{\upalpha}_{\Gmr}(\Gmr/\Hmr)$ should be thought of as a `thickening of $ \Jfrak_\Gmr(\Gmr/\Hmr):= \Rcal\mr{ep}_{\Hmr}$ using $\upalpha$'. Then we use right Kan extensions to define a natural transformation 
  \[ 
  \begin{tikzcd}
  \Emr^{(\upalpha)}: \Jfrak^{\upalpha}_\Gmr \rar& \Tfrak_{\Gmr},
  \end{tikzcd}
  \]
 which, following \cite{WeissCalc}, will be called  the {\bf $\upalpha$-th derivative of $\Emr$} (\Cref{defn:derivative}).

 Associated to $\Emr^{(\upalpha)}$ we construct an  $\upalpha$-{\bf system}  of $\Gmr$-spectra  ${\bf \Theta}\Emr^{(\upalpha)}$ (as in \Cref{defn:alphasystem})  which  we  call {the \bf $\upalpha$-th derivative of $\Emr$ at $\infty$} (\Cref{defn:der-at-infty}). The system  ${\bf \Theta}\Emr^{(\upalpha)}$ is equipped with a n\"aive action of $\Omr(\upalpha)$.   Further, we  define the  {\bf $\upalpha$-th layer} of $\Emr$ (see \Cref{defn:layer}) as a functor 
 \[ 
 \begin{tikzcd}
 {\bf L}_{\upalpha} \Emr : \Jfrak_\Gmr \rar[] & \Tfrak_\Gmr
 \end{tikzcd}
 \]
and show that it is  {\bf $\upalpha$-homogeneous} in the sense of \Cref{defn:vhomo}. Further, we show that the layers are pointwise equivariant infinite loop spaces determined by $\hat{\bf \Theta}\Emr^{(\upalpha)}:= {\bf \Theta} {\bf L}_{\upalpha}\Emr^{(\upalpha)}$ whenever  $\upalpha$   is a  finite  sum of one dimensional $\Gmr$-representations: 
 \begin{main} \label{main3} Suppose $\upalpha$ is a $\Gmr$-representation isomorphic to a finite sum of  one dimensional $\Gmr$-representations and $\hat{\bf \Theta}\Emr^{(\upalpha)}$  is uniformly bounded below in the sense of \Cref{defn:bbelow}. Then for any finite subspace $\Umr$ of $\Ucal_\Hmr$ 
 \[ 
({\bf L}_{\upalpha}\Emr)_{\Gmr/\Hmr}(\Umr) \simeq  \Omega^{\infty}_{\upiota_\Hmr(1)} \left( (\Smr^{\upiota_\Hmr(\upalpha) \otimes \Umr_{\upalpha}} \sma \hat{\bf \Theta}\Emr^{(\upalpha)}(\Umr_{\upalpha}^{\perp}))_{\sfh \Omr(\upalpha)} \right),
 \]
 where $\Umr_{\upalpha} = \Umr \cap \Ucal^{\upiota_\Hmr(1)}_\Hmr$ and $\Umr_{\upalpha}^{\perp}$ refers to its orthogonal complement in $\Umr$. 
 \end{main}
  \begin{rmk}
  All $\Gmr$-spectra in the system ${\bf \Theta}\Emr^{(\upalpha)}$ belongs to $\Sp^{\Gmr}_{(1)}$ and admit a  natural map to a genuine $\Gmr$-spectrum $\wh{\Theta\Emr}^{(\upalpha)}$ up to a suspension as described in  \eqref{map-to-genuine}. In some sense, $\wh{\Theta\Emr}^{(\upalpha)}$ is the cocompletion of the system ${\bf \Theta}\Emr^{(\upalpha)}$. 
When   $\Ucal^{(1)}_\Gmr$ is the complete universe\footnote{This is the case when $\Gmr = \Cmr_{2}^{\times n}$,  as well as for all finite Abelian groups in the unitary analog of our Weiss theory.} then the system is determined by just one genuine $\Gmr$-spectrum, namely $\wh{\Theta\Emr}^{(\upalpha)}$. 
\end{rmk}
\begin{rmk} Whether $\hat{\bf \Theta}\Emr^{(\upalpha)}$ and ${\bf \Theta}\Emr^{(\upalpha)}$ are equivalent (or not) is the subject of  \Cref{conj:symsys}. 
 \end{rmk}
 \begin{ex} \label{ex:1derequiv}When $ \Ucal_{\Gmr}^{(1)}$ is the complete universe $\Ucal_{\Gmr}$ we notice using the same arguments as in \cite{WeissCalc} that 
   \[ {\bf \Theta}{\bf B}_\Gmr{\bf O}^{(\upepsilon)} \cong \SS_\Gmr  \]
  the genuine $\Gmr$-equivariant sphere spectrum  for all $1$-dimensional $\Gmr$-representation $\upepsilon$. When $\Ucal_\Gmr^{(1)}$ is not the complete universe then 
  \[ {\bf \Theta}{\bf B}_\Gmr{\bf O}^{(\upepsilon)} = \bigcup_{\Hmr \subset \Gmr}\left\lbrace \Sigma^{ \upiota_\Hmr(\upepsilon)\otimes \Umr}\SS_{\upiota_\Hmr(1)} : \Umr \underset{\text{finite}}{\subset} \Ucal_{\Hmr} \right\rbrace  \]
 is the system, where  $\SS_{\upiota_\Hmr(1)}$ is the sphere spectrum in the $\Ucal^{\upiota_\Hmr(1)}_{\Gmr}$ universe. 
 \end{ex}
For any subgroup $\Kmr \subset \Gmr$, there is an obvious functor 
\[ 
\begin{tikzcd}
{\sf q}_{\Kmr}:\Ocal^{\mr{op}}_{\Kmr } \rar & \Ocal^{\mr{op}}_{\Gmr }
\end{tikzcd}
\]
 sending $\Kmr/\Hmr$ to $\Gmr/\Hmr$. It is easy to see that $\Jfrak_\Gmr \circ \sfq_{\Kmr} = \Jfrak_{\Kmr}$, $\Tfrak_\Gmr \circ \sfq_{\Kmr} = \Tfrak_{\Kmr}$, and the natural transformation $\Emr$ of  \eqref{natE} restricts to a natural transformation
  \begin{equation} \label{natresE}
  \begin{tikzcd}
  \upiota_{\Kmr}\Emr: \Jfrak_{\Kmr} \rar[]&  \Tfrak_{\Kmr} 
  \end{tikzcd}
  \end{equation}
given by  $(\upiota_{\Kmr}\Emr_{\Kmr/\Hmr})(\Vmr) :=  \Emr_{\Gmr/\Hmr}(\Vmr)$, where $\Hmr \subset \Kmr$ and $\Vmr \in \Rcal\mr{ep}_{\Hmr}$.
\begin{ex} We notice that the  restriction of ${\bf B}_\Gmr{\bf O}$ to a subgroup $\Kmr \subset \Gmr$ is simply ${\bf B}_\Kmr{\bf O}$. 
\end{ex}

 We notice equivariant Weiss calculus behaves well with restriction functors (see \Cref{lem:restower} and \Cref{lem:resderinf}):

\begin{main} \label{main:res} Suppose   $\upalpha$ and $\upepsilon$ are $\Gmr$-representations such that  $|\upepsilon| = 1$. Then, for any $1$-morphism  $
 \begin{tikzcd}
 \Emr: \Jfrak_{\Gmr} \rar & \Tfrak_{\Gmr}
 \end{tikzcd}
$ 
\begin{enumerate}[(1)]
\item $\upiota_\Kmr ({\bf T}_{(\upalpha, \upepsilon)} \Emr) = {\bf T}_{(\upiota_{\Kmr}\upalpha, \upiota_{\Kmr}\upepsilon)}\upiota_{\Kmr}( \Emr)$, 
\item $\upiota_\Kmr(\Emr^{(\upalpha)}) = \upiota_\Kmr\Emr^{(\upiota_\K\upalpha)} $, 
\item there is a natural map \[ 
\sfi^\Kmr_*\upiota_\Kmr {\bf \Theta}\Emr^{(\upalpha)}(\Umr)  \longrightarrow  {\bf \Theta} \upiota_\Kmr (\Emr^{(\upalpha)})(\Umr)
,\] where $\sfi^\Kmr_*$ is the change of universe functor of \eqref{C-of-U-1},  which is an isomorphism when  $\upiota_\Kmr \Ucal^{(1)}_\Gmr \cong \Ucal^{(1)}_\Kmr $, 
\end{enumerate}
 for any subgroup $\Kmr \subset \Gmr$. 
\end{main}

A $\Gmr$-equivariant Weiss functor $\Emr$  admits a $\Kmr$-fixed point functor
\begin{equation} \label{natfixE}
  \begin{tikzcd}
  \Emr^\Kmr: \Jfrak_{\Wmr(\Kmr)} \rar&  \Tfrak_{\Wmr(\Kmr)} 
  \end{tikzcd}
\end{equation}
for every subgroup $\Kmr \subset \Gmr$. Given $\Hmr/\Kmr \in \Wmr(\Kmr)$ and an $\Hmr/\Kmr$-representation $\Vmr$, let  $\uppi^*_{\Hmr}\Vmr$ denote the 
$\Hmr$-representation obtained by pulling back $\Vmr$ along the quotient map $\uppi_{\Hmr}: \Hmr \twoheadrightarrow \Hmr/\Kmr$. Then define $\Emr^\Kmr$ using the formula 
\[ \Emr^\Kmr_{\Hmr/\Kmr} (\Vmr) := \left(\upiota_{\Nmr(\Hmr)}\Emr_{\Nmr(\Hmr)/\Hmr} (\uppi^*\Vmr) \right)^{\Hmr},\]
whenever $\Hmr$ is the subgroup of the normalizer $\Nmr(\Hmr)$ of $\Hmr$ in $\Gmr$, and $\Vmr$ is an $\Hmr/\Kmr$-representation. 

The equivariant Weiss calculus is not as well behaved with respect to fixed-point functors as it is with restriction functors.  Combining \Cref{lem:towerfix}, \Cref{lem:derfix}, and \Cref{cor:derinffix}, we obtain the following result.  
\begin{main} \label{main:fix} Suppose  $\Kmr$ is a subgroup of $\Gmr$, and  $\upalpha, \upepsilon$ are $\Wmr(\Hmr)$-representations such that  $\dim_\RR(\upepsilon)= 1$.  Then,  for every $1$-morphism  
 $ 
 \begin{tikzcd}
 \Emr: \Jfrak_{\Gmr} \rar & \Tfrak_{\Gmr}
 \end{tikzcd}
$ there exists 
\begin{enumerate}[(1)]
\item  a $2$-morphism  $
{\bf T}_{(\upalpha, \upepsilon)}\Emr^{\Kmr}  \longrightarrow\left({\bf T}_{(\uppi^*\upalpha, \uppi^*\upepsilon)} \upiota_{\Nmr(\Hmr)}\Emr \right)^{\Kmr} $
 in  ${\bf Fun}( \Ocal_{\Wmr(\Hmr)}^{\mr{op}}, \Cat)$,
\item  a  $2$-morphism 
$
\begin{tikzcd}
 (\Emr^{\Kmr})^{(\upalpha)}  \rar &  (\Emr^{(\uppi^*\upalpha)})^\Kmr
 \end{tikzcd}
 $
 in  ${\bf Fun}( \Ocal_{\Wmr(\Hmr)}^{\mr{op}}, \Cat)$,
\item and a zig-zag of  maps in $\Sp^{\Wmr(\Kmr)}_{\upiota_{\br{\Hmr}}(1:\upalpha^\Kmr)}$
\[ 
\begin{tikzcd}
&& ({\sfi}^{\Nmr(\Kmr)}_*{\bf \Theta}\Emr^{(\upalpha)}(\uppi^*\Umr))^{\Kmr} \dar[]  \\
{\bf \Theta} ((\Emr^{\Kmr})^{(\upalpha^\Kmr)})(\Umr) \ar[rr] && \left( ({\bf \Theta}(\upiota_{\Nmr(\Kmr)}\Emr^{(\upalpha)})(\uppi^*\Umr)\right)^\Kmr, 
\end{tikzcd}
\]
\end{enumerate}
 where $\uppi: \Nmr(\Kmr) \twoheadrightarrow \Wmr(\Hmr)$ is the  quotient map and    $\sfi^{\Nmr(\Kmr)}_*$ 
 is the change of universe map \eqref{C-of-U-1}. 
\end{main} 
\begin{ex} Here we discuss the fixed point functors of some of the $\Gmr$-equivariant Weiss functors discussed in \Cref{ex:eqWeissF}. 
\begin{itemize}
\item The fixed-point functor $ ({\bf S}_\Gmr)^{\Hmr}$ is simply  ${\bf S}_{\Wmr(\Hmr)}$. In particular when $\Hmr = \Gmr$ we get the nonequivariant functor  that facilitates the comparison between the Weiss Calculus and the   Goodwillie calculus (see \cite[pg 1004]{GDK}). 
\item The fixed-point functor ${\bf B}_\Gmr{\bf O}^\Hmr$ is not  ${\bf B}_{\Wmr(\Hmr)}{\bf O}$ due to \cite[Theorem 10]{LM} (also see \cite[Theorem 2.13]{BZ}) but contains ${\bf B}_{\Wmr(\Hmr)}{\bf O}$ as a sub-functor. For instance when $\Hmr= \Gmr$, ${\bf B}_\Gmr{\bf O}^\Gmr$ is the nonequivariant Weiss functor sending a vector space $\RR^n$ to disjoint union of $\BO(n)$, one for each isomorphism class of $n$-dimensional $\Gmr$-representation (see \Cref{bpBGO}). A similar conclusion holds for ${\bf B}_\Gmr{\bf TOP}$ and ${\bf B}_\Gmr{\bf Diff}$.
\item The $\Gmr$-fixed point functor $({\bf Emb}^{\Gmr}_{\Mmr, \Nmr})^{\Gmr}$ sends $\RR^n$ to the space of $\Gmr$-equivariant embedding of $\Mmr$ into $\Nmr \times \RR^n$ with a disjoint basepoint. If  $\Gmr$  acts trivially on $\Mmr$ and $\Nmr$ then $({\bf Emb}^{\Gmr}_{\Mmr, \Nmr})^{\Gmr}$ is the nonequivariant embedding functor ${\bf Emb}_{\Mmr, \Nmr}$. 
\end{itemize}
\end{ex}

\begin{ex} Although ${\bf B}_\Gmr{\bf O}^\Gmr$ is not quite the nonequivariant Weiss functor ${\bf B}{\bf O}$ (see \Cref{bpBGO}), the difference lies only in the zeroth Taylor approximation as 
\[ 
{\bf T}_{{\bf 0}}({\bf B}_\Gmr{\bf O}^\Gmr)(\Vmr) = \coprod_{{\sf Irr}(\Gmr)} \BO, 
\] 
where ${\sf Irr}(\Gmr)$ is the set of irreducible $\Gmr$-representations. After considering basepoint, we see 
\[ {\bf Fiber}\left({\bf B}_\Gmr{\bf O}^\Gmr(\Vmr) \to {\bf B}_\Gmr{\bf O}^\Gmr(\Vmr \oplus \RR)\right) \simeq  {\bf Fiber}\left({\bf BO}(\Vmr) \to  {\bf BO}(\Vmr \oplus \RR) \right)\]
is the sphere $\Smr^{\Vmr}$. This  combined with the arguments of Weiss \cite{WeissCalc}, shows that the first derivative of ${\bf B}_\Gmr{\bf O}^\Gmr$ at $\infty$ is the nonequivariant sphere spectrum $\SS$. This is different from the $\Gmr$-fixed points of the  first derivative of ${\bf B}_\Gmr{\bf O}$ at $\infty$ (see \Cref{ex:1derequiv}) even when $\Gmr= \Cmr_2$ due to tom Dieck splitting results \cite{tDsplitting}. 
\end{ex}
\begin{rmk}[Unitary equivariant Weiss calculus] \label{Weissunitary} When we  work with complex $\Gmr$-representations  and set  morphisms as $\CC$-linear isometries, we get a presheaf 
\[ 
\begin{tikzcd}
\Jfrak_{\Gmr}^{ \CC}: \Ocal^{\mr{op}}_\Gmr \rar & \Cat,
\end{tikzcd}
 \]
 which is the complex analog of $\Jfrak_{\Gmr}$.  Then, our theory builds Taylor approximations and derivatives for any $1$-morphism $\Emr: \Jfrak_{\Gmr}^{\CC} \longrightarrow \Tfrak_{\Gmr}$ and the analogous results hold after replacing orthogonal groups with unitary groups. 
\end{rmk}
\begin{rmk}[Atiyah Real Weiss calculus] \label{WeissAR} When $\Gmr = \Cmr_2$,  one may utilize the Galois action of $\Cmr_2$ on $\CC$ to define 
\[ 
\begin{tikzcd}
\Jfrak_{\CC/\RR}:\Ocal^{\mr{op}}_{\Cmr_2} \rar & \Cat
\end{tikzcd}
\]
as follows. We let  $\Jfrak_{\CC/\RR}(\Cmr_2/\Cmr_2)$ be  the  $\Cmr_2$-category with objects of the form $\Vmr \otimes_{\RR} \CC$, where $\Vmr$ is a  finite dimensional real vector space, and morphisms are $\CC$-linear isometries, and $\Jfrak_{\CC/\RR}(\Cmr_2/e)$ is 
simply the category of finite dimensional complex vector spaces  with $\CC$-linear isometries as morphisms, where $\Cmr_2$ acts trivially. Note that the fixed-point functor  $\Jfrak_{\CC/\RR}^{\Cmr_2}$ sends $e/e$ to $\Jcal^\RR$. 
\end{rmk}
\begin{ex} \label{ex:AR} For a  finite dimensional $\RR$-vector space $\Wmr$, let $\BU_\RR(\Wmr)$ denote the complex linear distance preserving automorphisms of $\Wmr \otimes \CC$. Now, consider the $1$-morphism 
\[
\begin{tikzcd}
{\bf BU}_\RR: \Jfrak_{\CC/\RR} \rar & \Tfrak_{\Cmr_2} 
\end{tikzcd} 
\]
 that sends $\Wmr$ to $\Bmr\Umr_\RR(\Wmr)$ if $\Wmr \in \Jfrak_{\Cmr_2/\Cmr_2}$, and to $\BU(\Wmr)$ if $\Wmr \in \Jfrak_{\Cmr_2/e}$. Then its restriction to $e$ is simply  $\BU(-): \Jcal^\CC \longrightarrow \Top_*$, and its $\Cmr_2$-fixed point functor is $\BO(-): \Jcal^\RR \longrightarrow \Top_*$. 
\end{ex}
\subsection*{Future directions. }  \

In \Cref{main:poly} and \Cref{main3}, we see restrictions to finite sums of $1$-dimensional representations. A part of these  restrictions are due to technical reasons as our key results such as  \Cref{thm:SC1} and \Cref{lem:polyplus1} do not seem to generalize if $\upepsilon$ has dimension bigger than $1$. These restrictions are also imposed due to the philosophy that the role of `direction'  in Weiss calculus is played by $1$-dimensional subspaces. It will be  interesting to see if it is possible to overcome these technical and philosophical barriers  so that the $\upalpha$-homogeneous layer is a genuine $\Gmr$-equivariant infinite loop space  classified using a genuine $\Gmr$-spectrum.

In this paper, we restrict our attention to finite groups to make sure that the regular representation is finite dimensional. However, it will be of great interest to see if the equivariant Weiss calculus presented in this paper can be extended to compact Lie groups, and more generally, in the framework of global  homotopy theory \cite{Global}. In the nonequivariant case, one can compare Weiss calculus with  Goodwillie calculus (see \cite{GDK}).  Thus, one may also ask if our equivariant Weiss calculus can be related to equivariant Goodwillie calculus of \cite{Dotto}.

A significant portion of the literature on Weiss Calculus is dedicated to identifying derivatives at $\infty$ of important Weiss functors (as in \eqref{eq:exF}) \cite{WW88, WW89,  Arone02, Arone09, RandalDisc, KRW}.  We are cautiously optimistic that the work of Arone \cite{Arone02} can be generalized to identify the derivatives of equivariant Weiss functors such as ${\bf B_\Gmr U}, {\bf B_{\Gmr} O}$, and ${\bf BU}_\RR$. Such results will lead to enumeration of unstable equivariant vector bundles  extending the recent results of \cite{Yang23, CHO} to equivariant homotopy theory.  

 An important result in the nonequivariant Weiss theory is the calculation of the first derivative of $\Bmr\mr{Top}(-)$. The classical work of Waldhausen \cite{Wald} shows that the first derivative of $\Bmr\mr{Top}(-)$ is the  Waldhausen $\Amr$-theory of a point. In 2019, Malkiewich and Merling \cite{MM} constructed a genuine equivariant  Wadhausen $\Amr$-theory functor 
\[ 
\begin{tikzcd}
{\bf A}_\Gmr(-): \Gmr\Top \rar & \Sp^{\Gmr},
\end{tikzcd}
 \]
 and therefore,  it is conceivable that there exists an incomplete version of equivariant $\Amr$-theory functor, say ${\bf A}_\Gmr'$, whose codomain is $\Sp^{\Gmr}_{(1)}$ (instead of $\Sp^{\Gmr}$). If so,  we expect: 
\begin{conj} ${\bf \Theta} {\bf B}_{\Gmr}{\bf Top}^{(1)}({\bf 0})$  is weakly  equivalent to ${\bf A}_{\Gmr}'(\ast)$. 
\end{conj}

\subsection*{Acknowledgements.} The authors are grateful to  S\o{}ren Galatias, Gijs Heuts, Ishan Levy, Connor Malin, Mike Mandell,  Niall Taggart, and Foling Zou for helpful discussions, and indebted to {\it Algebraic structures in Topology conference} at San Jaun (2024)  and {\it International Workshop on Algebraic Topology}  at  Shanghai (2024) which facilitated some of these interactions. 

The authors would  like to acknowledge that they have greatly benefitted from the thesis work of Emel Yavuz \cite{Y24}. It contains several breakthrough ideas for $\Gmr=\Cmr_2$ that have been adopted for the general case in this paper. 

Last but not least, the authors are grateful to Michael Weiss for introducing this beautiful subject of Weiss Calculus.

This research is supported by NSF grant DMS-2305016. 
\subsection*{Organization of the paper.} In \Cref{sec:TaylowTower},  we first introduce Taylor approximations  of equivariant Weiss functors  and prove \Cref{main:tower}. Then we compare the Taylor approximations across restrictions and fixed-points.

In \Cref{sec:Weissder}, we introduce derivatives of equivariant Weiss functors and discuss the effect of restriction and fixed-point functors, thereby completing the proof of \Cref{main:res} and \Cref{main:fix}. 

In \Cref{sec:SC}, we discuss equivariant Stiefel combinatorics which is the technical core of this paper. In \Cref{sec:poly}, we introduce the notion of  polynomial functors and  prove \Cref{main:poly}. 

In \Cref{sec:homo}, we define layers of equivariant Weiss functors, introduce the notion of homogeneity,  and prove equivariant analogs of two important results, namely,  the classification of stable symmetric objects (\Cref{thm:stablesymmetric}) and the classification of homogeneous functors (\Cref{thm:classification}), which leads us to the proof of \Cref{main3}. 

\section{Equivariant  Taylor approximations} \label{sec:TaylowTower}
A morphism $\Emr: \Jfrak_{\Gmr} \longrightarrow \Tfrak_\Gmr$ 
in ${\bf Fun}( \Ocal_{\Gmr}^{\mr{op}}, \Cat)$ consists of, for each subgroup $\Hmr \subset \Gmr$, a functor
\[
\begin{tikzcd}
\Emr_{\Gmr/\Hmr}: \Rcal\mr{ep}_\Hmr \rar & \Hmr\Top, 
\end{tikzcd} 
\]
 which are compatible with restrictions, i.e., the diagram 
\[ 
\begin{tikzcd}
\Rcal\mr{ep}_\Hmr \rar["\Emr_{\Gmr/\Hmr}"] \dar["\upiota_{\Kmr}"] & \Hmr\Top \dar["\upiota_{\Kmr}"]\\
\Rcal\mr{ep}_\Kmr  \rar["\Emr_{\Gmr/\Hmr}"] & \Kmr\Top
\end{tikzcd}
\]
commutes whenever $\Kmr$ is a subgroup of  $ \Hmr$  up to a conjugation (within $\Gmr$).

\begin{notn} \label{notn:V} For a subgroup $\Hmr \subset \Gmr$ and  $\Wmr \in \Rcal\mr{ep}_\Hmr$, we let
\begin{itemize}
\item  $\Vscr(\Wmr)$ denote the $\Hmr$-category whose objects are  sub-vector spaces  of  $\Wmr$ and morphisms are  linear isometries. 
\item  $\Vscr^{\circ}(\Wmr)$ denote the full subcategory of nonzero vector spaces in  $\Vcal(\Wmr)$. 
\end{itemize}
Note that the objects of $\Vscr(\Wmr)$ fixed by $\Hmr$ are precisely the sub $\Hmr$-representations and the undercategory $\Wmr/_{\Rcal\mr{ep}_{\Hmr}}$ is a full subcategory of $\Vscr(\Wmr)$.
\end{notn}

Let $\widehat{\Emr}_{\Wmr}$ denote the  right Kan extension of  $\Emr_{\Gmr/\Hmr}$  to the undercategory $\Wmr/_{\Rcal\mr{ep}_{\Hmr}}$ 
\[ 
\begin{tikzcd}
\Wmr/_{\Rcal\mr{ep}_{\Hmr}} \dar[hook'] \ar[rr, "\Emr_{\Gmr/\Hmr}"] && \Hmr\Top   \\
\Vscr(\Wmr) \ar[urr, dashed, "\widehat{\Emr}_{\Wmr}"']
\end{tikzcd}
\]
 to $\Vscr(\Wmr)$. When $\Wmr$ is a sub $\Hmr$-representation of $\Wmr'$, $\Vscr(\Wmr)$ embeds fully and faithfully in  $\Vscr(\Wmr')$, and therefore, 
\begin{equation}
\wh{\Emr}_{\Wmr'}(\Umr) = \wh{\Emr}_{\Wmr}(\Umr)
\end{equation}
for all $\Umr \in \Vscr(\Wmr)$. Thus, $\wh{\Emr}_{\Wmr}(\Umr)$ is independent of the choice of $\Wmr \supset \Umr$, and therefore, one may define
\[ \wh{\Emr}_{\Gmr/\Hmr}(\Umr) := \wh{\Emr}_{\Wmr}(\Umr)   \]
which is independent of the choice of  $\Wmr \supset \Umr$.  
\begin{rmk} Alternatively, one may consider the $\Hmr$-category $\Vscr_{\Hmr}$ of finite dimensional sub-vector spaces of the complete $\Hmr$-universe $\Ucal_{\Hmr}$ (in which morphisms are linear isometries), and define $\wh{\Emr}_{\Gmr/\Hmr}$ as the right Kan extension
\begin{equation} \label{Ehat}
\begin{tikzcd}
\Rcal\mr{ep}_{\Hmr} \dar[hook', "{\sf e}_\Hmr"'] \ar[rr, "\Emr_{\Gmr/\Hmr}"] && \Hmr\Top   \\
\Vscr_{\Hmr} \ar[urr, dashed, "\widehat{\Emr}_{\Gmr/\Hmr}"']
\end{tikzcd}
\end{equation}
for all subgroup $\Hmr$ of $\Gmr$. 
\end{rmk}
\begin{rmk} Let $\wh{\Jfrak}_{\Gmr}: \Ocal^{\mr{op}}_{\Gmr} \longrightarrow \Cat$ denote the functor that sends $\Hmr$ to $\Vscr_{\Hmr}$. Then 
$\{ \wh{\Emr}_{\Gmr/\Hmr} : \Hmr \subset \Gmr\}$  assembles to form a $1$-morphism 
\[ 
\wh{\Emr}: \wh{\Jfrak}_{\Gmr} \longrightarrow \Tfrak_\Gmr
\]
in ${\bf Fun}( \Ocal_{\Gmr}^{\mr{op}}, \Cat)$. 
\end{rmk}

\begin{defn} \label{defn:tau} Suppose  $\upalpha \in \Rep_\Gmr$   and $ \Emr: \Jfrak_{\Gmr} \longrightarrow \Tfrak_\Gmr$ is an object in ${\bf Fun}( \Ocal_{\Gmr}^{\mr{op}}, \Cat)$. Then define  
\[ 
\begin{tikzcd}
\uptau_{\upalpha} \Emr: \Jfrak_{\Gmr} \rar & \Tfrak_\Gmr
\end{tikzcd}
\]
using the formula 
\begin{equation} \label{eqn:taudefn}
\begin{tikzcd}
(\uptau_{\upalpha}\Emr)_{\Gmr/\Hmr}(\Wmr):= \underset{\Umr \in \Vscr^{\circ}_{\Hmr}(\upiota_{\Hmr}(\upalpha))} \holim \wh{\Emr}_{\Gmr/\Hmr}(\Wmr \oplus \Umr),
\end{tikzcd}
\end{equation}
where $\Wmr \in \Rep_{\Hmr}$ and the action of $\Hmr$ is determined using  \Cref{rmk:actionholim}. 
\end{defn}

\begin{rmk} \label{rmk:actionholim} Suppose $\Ccal$ is a small $\Gmr$-category then the  homotopy limit of 
\[
\begin{tikzcd}
 \Fmr: \Ccal \rar &  \Gmr \Top
 \end{tikzcd}
  \] is define as the totalization of the cosimplical object that sends ${\bf n} \in \Delta$ to 
\begin{equation*}
\begin{tikzcd}
\underset{c_0 \to \dots \to c_n}{\prod}\Fmr(c_n) 
\end{tikzcd}
\end{equation*}
on which the action of ${\sf g} \in \Gmr$  is given by the formula
\[ 
\begin{tikzcd}
{\sf g}^{-1} \cdot \left( \underset{c_0 \to \dots \to c_n}{\prod}\Fmr(c_n) \right) = \underset{\sfg \cdot c_0 \to \dots \to \sfg \cdot c_n}{\prod}\sfg^{-1}  \cdot \Fmr(\sfg \cdot c_n). 
\end{tikzcd}
\]
In particular, we note that the action on the homotopy limit of $\Fmr$ depends on the action of $\Gmr$ on $\Ccal$. 
\end{rmk}

 It follows from \eqref{eqn:taudefn} that there exist natural maps 
\begin{equation} \label{etapointwise}
\begin{tikzcd} 
\eta^{\Emr}_{\upalpha, \Wmr}: \Emr_{\Gmr/\Hmr}(\Wmr) \rar & (\uptau_{\upalpha }\Emr)_{\Gmr/\Hmr}(\Wmr)
\end{tikzcd}
 \end{equation}
which assemble to form a  $2$-morphism 
\begin{equation} \label{eta}
\begin{tikzcd}
\eta^{\Emr}_{\upalpha}: \Emr \rar[] &  \uptau_{\upalpha}\Emr 
\end{tikzcd}
\end{equation}
in ${\bf Fun}( \Ocal_{\Gmr}^{\mr{op}}, \Cat)$.

We define $ \uptau_{\upalpha }^{(n)}\Emr $ inductively using the formula 
\[  \uptau_{\upalpha }^{(n)}\Emr := \uptau_{\upalpha }( \uptau_{\upalpha }^{(n-1)}\Emr)  \]
 and apply  $\uptau_{\upalpha}^{(n)}(-)$ to \eqref{eta} to make the following definition. 
\begin{defn} \label{defn:Taylor} Suppose $\upalpha$ and $\upepsilon$ are $\Gmr$-representations  such that  $|\upepsilon| = 1$.  Then define the {\bf $\upalpha$-th Taylor approximation} of $\Emr: \Jfrak_\Gmr \longrightarrow \Tfrak_\Gmr$ {\bf in the direction of $\upepsilon$} as the $1$-morphism 
\[ 
\begin{tikzcd}
{\bf T}_{(\upalpha, \upepsilon)} \Emr : \Jfrak_\Gmr \rar & \Tfrak_\Gmr
\end{tikzcd}
\]
given by  the formula
\[ 
({\bf T}_{(\upalpha, \upepsilon)} \Emr)_{\Gmr/\Hmr}(\Wmr) := \hocolim_{n \in \NN}    \uptau_{\upalpha \oplus \upepsilon}^{(n)}\Emr_{\Gmr/\Hmr}(\Wmr),
\]
where $\Wmr \in \Rep_\Hmr$. 
\end{defn}

\bigskip
\begin{proof}[{\bf Proof of \Cref{main:tower}}] It is clear from the construction of  ${\bf T}_{(\upalpha, \upepsilon)}$ in  \Cref{defn:Taylor} that there exist $2$-morphisms
\[ \begin{tikzcd}
\upeta_{\alpha, \upepsilon}^{\Emr}:  \Emr \rar & {\bf T}_{(\upalpha, \upepsilon)}\Emr
\end{tikzcd} \]
induced by  \Cref{eta}. When $\upalpha'$ is a sub $\Gmr$-representation of $\upalpha$, then  $\Vscr^{\circ}(\upiota_\Hmr\upalpha')$ is a subcategory of  $\Vscr^{\circ}(\upiota_\Hmr\upalpha)$, consequently there is a natural map 
\[ 
\begin{tikzcd}
\eta_{\alpha \oplus \epsilon, \alpha' \oplus \epsilon}: (\uptau_{\upalpha \oplus \upepsilon}\Emr)_{\Gmr/\Hmr}(\Wmr) \rar & (\uptau_{\upalpha \oplus \upepsilon}\Emr)_{\Gmr/\Hmr}(\Wmr)
\end{tikzcd}
\]
compatible with the restriction map. This leads to a $2$-morphism
\[ 
\begin{tikzcd}
\upeta^{\Emr}_{\upalpha, \upalpha', \upepsilon}: {\bf T}_{(\alpha, \upepsilon)} \Emr \rar & {\bf T}_{(\alpha', \upepsilon)}\Emr, 
\end{tikzcd}
\]
and the equations
\[  \upeta^{\Emr}_{\upalpha', \upepsilon} = \upeta^{\Emr}_{\upalpha, \upalpha', \upepsilon} \circ \upeta^{\Emr}_{\upalpha, \upepsilon}\]
\[ \upeta^{\Emr}_{\upalpha'', \upalpha', \upepsilon} \circ  \upeta^{\Emr}_{\upalpha, \upalpha', \upepsilon}  = \upeta^{\Emr}_{\upalpha, \upalpha'',\upepsilon} \]
can be readily verified. 
\end{proof}

\bigskip

\subsection{Restrictions of equivariant Taylor approximations} \label{subsec:restower}\

For a subgroup $\Kmr \subset \Gmr$,  there is an obvious functor 
\begin{equation} \label{resfunK}
\begin{tikzcd}
{\upiota}_{\Kmr}:\Ocal^{\mr{op}}_{\Kmr } \rar & \Ocal^{\mr{op}}_{\Gmr }
\end{tikzcd}
\end{equation}
 sending $\Kmr/\Hmr$ to $\Gmr/\Hmr$. It is easy to see that 
 \begin{itemize}
 \item $\Jfrak_\Gmr \circ \upiota_{\Kmr} = \Jfrak_{\Kmr}$,
 \item  $\Tfrak_\Gmr \circ \upiota_{\Kmr} = \Tfrak_{\Kmr}$, and
 \item  a $1$-morphism $\Emr:  \begin{tikzcd}
\Jfrak_{\Gmr}  \rar[]&  \Tfrak_{\Gmr} 
  \end{tikzcd}
$
  restricts to a $1$-morphism
$
  \begin{tikzcd}
  \upiota_{\Kmr}\Emr: \Jfrak_{\Kmr}  \rar[]&  \Tfrak_{\Kmr} 
  \end{tikzcd}
$
given by  
\begin{equation} \label{resformula}
(\upiota_{\Kmr}\Emr_{\Kmr/\Hmr})(\Wmr) :=  \Emr_{\Gmr/\Hmr}(\Wmr)
\end{equation}
 for $\Wmr \in \Rcal\mr{ep}_{\Hmr}$ and  $\Hmr \subset \Kmr$.
\end{itemize}
\begin{lem} \label{lem:restower} Suppose $\upalpha$ and $\upepsilon$ are $\Gmr$-representations  such that  $|\upepsilon| =1$. Then the  diagram 
\[ 
\begin{tikzcd} 
 \Morone(\Jfrak_\Gmr, \Tfrak_{\Gmr})\ar[rr, "{\bf T}_{(\upalpha, \upepsilon)}"] \dar[" \upiota_{\Kmr}"'] && \Morone(\Jfrak_\Gmr, \Tfrak_{\Gmr})\dar[" \upiota_{\Kmr}"] \\
\Morone(\Jfrak_\Kmr, \Tfrak_{\Kmr}) \ar[rr, "{\bf T}_{(\upiota_{\Kmr}\upalpha, \upiota_{\Kmr}\upepsilon)}"'] && \Morone(\Jfrak_\Kmr, \Tfrak_{\Kmr})
\end{tikzcd}
\]
commutes, i.e., $ \upiota_{\Kmr}{\bf T}_{(\upalpha, \upepsilon)} \Emr = {\bf T}_{(\upiota_{\Kmr}\upalpha, \upiota_{\Kmr}\upepsilon)} \upiota_{\Kmr}\Emr $ 
for any $1$-morphism $\Emr: \Jfrak_{\Gmr} \longrightarrow \Tfrak_{\Gmr}$.
\end{lem}
\begin{proof} Since $\wh{\Jfrak}_{\Kmr}(\Kmr/\Hmr) = \Vscr_{\Hmr} = \wh{\Jfrak}_{\Gmr}(\Gmr/\Hmr)$, it follows from  \eqref{Ehat} and \eqref{resformula} that 
\[ \upiota_{\Kmr} \wh{\Emr}_{\Kmr/\Hmr}(\Wmr) = \wh{\Emr}_{\Gmr/\Hmr}(\Wmr) \]
for all $\Wmr \in \Vscr_{\Hmr}$. Consequently, 
\begin{eqnarray*}
(\uptau_{(\upiota_\Hmr \upalpha, \upiota_\Hmr \upepsilon)} \upiota_\Kmr \Emr)_{\Kmr/\Hmr}(\Wmr) &=&  \underset{\Umr \in \Vscr^{\circ}_{\Hmr}(\upiota_{\Kmr}(\upalpha \oplus \upepsilon))} \holim (\upiota_{\Kmr}\wh{\Emr})_{\Kmr/\Hmr}(\Wmr \oplus \Umr) \\
&=& \underset{\Umr \in \Vscr^{\circ}_{\Hmr}(\upiota_{\Hmr}(\upalpha \oplus \upepsilon))} \holim \wh{\Emr}_{\Gmr/\Hmr}(\Wmr \oplus \Umr) \\
&=& \uptau_{(\upalpha, \upepsilon)} \Emr_{\Gmr/\Hmr} (\Wmr) \\
&=& \upiota_{\Kmr}(\uptau_{(\upalpha, \upepsilon)} \Emr)_{\Gmr/\Hmr} (\Wmr)
\end{eqnarray*}
and  
$ \eta^{\upiota_{\Kmr}\Emr}_{(\upiota_\Hmr\upalpha, \upiota_\Hmr\upepsilon)}(\Wmr) = \upiota_{\Kmr}(\eta^{\Emr}_{(\upalpha, \upepsilon)}(\Wmr)) $. Since restrictions commute with homotopy colimits, the result follows. 
\end{proof}
\subsection{Fixed-points of equivariant  Taylor approximations} \label{subsec:fixtower} \

For a subgroup $\Kmr \subset \Gmr$, the $\Kmr$-fixed points are constructed by first restricting to the normalizer subgroup $\Nmr(\Kmr) \subset \Gmr$, and then taking the fixed points. Therefore, without loss of generality we assume $\Kmr$ to  be a normal subgroup of $\Gmr$, so that  the quotient $\overline{\Gmr} := \Gmr/\Kmr$ is the Weyl group $\Wmr(\Kmr)$. 

Let 
$
\uppi: \Gmr \twoheadrightarrow \br{\Gmr} 
$
denote the quotient map. Then there is a functor between the orbit categories 
\[ 
\begin{tikzcd}
\upphi: \Ocal^{\mr{op}}_{\br{\Gmr}} \rar & \Ocal^{\mr{op}}_{\Gmr} 
\end{tikzcd}
\]
sending $\br{\Gmr}/\br{\Hmr} \mapsto \Gmr/ \Hmr $, where $\Hmr = \uppi^{-1}(\br{\Hmr})$.

 Given an $\Hmr$-space $\Xmr$ for $\Hmr \supset \Kmr$,  its  $ \Kmr$-fixed points $\Xmr^{\Kmr}$ is an $\br{\Hmr}$-space, and for  an $\overline{\Hmr}$-space  $\Ymr$, $\uppi^{\ast}\Ymr$ is an $\Hmr$-space with trivial action of $\Kmr$. Thus,  in ${\bf Fun}( \Ocal_{\br{\Gmr}}^{\mr{op}}, \Cat)$, there exist $1$-morphisms 
\[ \begin{tikzcd}
\uppi^*: \Jfrak_{\br{\Gmr}} \rar[shift right] &\lar[shift right] \upphi^*\Jfrak_{\Gmr} : (-)^{\Kmr}
\end{tikzcd} \text{ and }
\begin{tikzcd}
\Emr^{\Kmr}:\uppi^*: \Tfrak_{\br{\Gmr}} \rar[shift right] &\lar[shift right] \upphi^*\Tfrak_{\Gmr} : (-)^{\Kmr}
\end{tikzcd}
 \] 
such that $(\uppi^{*}(-))^{\Kmr}$ is the identity  in both cases. 
\begin{defn} \label{defn:Efix} Define the {\bf $\Kmr$-fixed point of the $\Gmr$-equivariant Weiss functor} $\Emr: \Jfrak_{\Gmr} \longrightarrow \Tfrak_{\Gmr}$  as the composite 
\[ 
\begin{tikzcd}
\Emr^{\Kmr}: \Jfrak_{\br{\Gmr}} \rar["\uppi^*"] &  \upphi^*\Jfrak_{\Gmr} \rar["\upphi^*\Emr"]  & \upphi^*\Tfrak_{\Gmr} \rar["(-)^\Kmr"] & \Tfrak_{\br{\Gmr}}
\end{tikzcd}
\]
in ${\bf Fun}( \Ocal_{\br{\Gmr}}^{\mr{op}}, \Cat)$. 
\end{defn}
\begin{rmk} An $\overline{\Hmr}$-representation $\Wmr$, for  $\br{\Hmr} \subset \br{\Gmr}$,  can be viewed as an $ \Hmr$-representation, where 
$\Hmr = \uppi^{-1}(\br{\Hmr})$. It is nothing but the pullback of $\Wmr$ along the quotient map $\Hmr \twoheadrightarrow \br{\Hmr}$ induced by $\uppi$. Therefore, we will denote it by $\uppi^*(\Wmr)$.  Then,  by \Cref{defn:Efix} of $\Kmr$-fixed point functor of 
 $\Emr$  is given by 
\[ \Emr^{\Kmr}_{\br{\Gmr}/\br{\Hmr}}(\Wmr) = \Emr_{\Gmr/\Hmr}\left(\uppi^{*}(\Wmr)\right)^\Kmr \]
 for any $\Wmr \in \Rep_\Hmr$. 
\end{rmk}
\begin{lem} \label{lem:towerfix} Suppose $\upalpha$ and $\upepsilon$ are $\br{\Gmr}$-representations  such that  $\dim_\RR(\upepsilon) = 1$, then there exists  a $2$-morphism 
\begin{equation} \label{map:uplambda}
\begin{tikzcd}
\uplambda_{\upalpha, \upepsilon}^\Emr:{\bf T}_{(\upalpha, \upepsilon)}\Emr^{\Kmr}  \rar & \left({\bf T}_{(\uppi^*\upalpha, \uppi^*\upepsilon)} \Emr \right)^{\Kmr}
\end{tikzcd}
\end{equation}
for any $1$-morphism $\Emr: \Jfrak_{\Gmr} \longrightarrow \Tfrak_{\Gmr}$.
\end{lem}
\begin{proof} Since $(\wh{\Emr^\Kmr})_{\br{\Gmr}/\br{\Hmr}}$ is the right Kan extension of  $\Emr_{\br{\Gmr}/\br{\Hmr} }^\Kmr$ along ${\sf e}_{\br{\Hmr}}:\Rep_{\br{\Hmr}}  \hookrightarrow  \Vscr_{\br{\Hmr}} $ (see  \eqref{Ehat}), and the fact that  
\[ 
\begin{tikzcd}
\Rep_{\br{\Hmr}} \dar["\sfe_{\br{\Hmr}}"', hook]  \rar["\uppi^*"] & \Rep_{\Hmr}  \rar["\Emr_{\Gmr/\Hmr}"] \dar["\sfe_{\Hmr}"', hook]  & \Hmr\Top \rar["(-)^{\Kmr}"] & \br{\Hmr}\Top \\
\Vscr_{\br{\Hmr}} \rar["\uppi^{*}"'] &\Vscr_{\Hmr} \ar[ur, "\wh{\Emr}_{\Gmr/\Hmr}"']
\end{tikzcd}
\]
\begin{eqnarray*}
 \wh{\Emr}_{\Gmr/\Hmr}^{\Kmr} ( {\sf e}_{\br{\Hmr}}(-)) &= & \left(  \wh{\Emr}_{\Gmr/\Hmr}( \uppi^*( {\sf e}_{\br{\Hmr}}(-))) \right)^\Kmr  \\
 &=& \left(  \wh{\Emr}_{\Gmr/\Hmr}(  {\sf e}_{\Hmr}(\uppi^*(-))) \right)^\Kmr \\
 &=& \left(  \Emr_{\Gmr/\Hmr}(  \uppi^*(-)) \right)^\Kmr \\
 &=& \Emr_{\br{\Gmr}/\br{\Hmr} }^\Kmr(-), 
\end{eqnarray*}
there exists a $2$-morphism 
\[ 
\begin{tikzcd}
\lambda: \wh{\Emr^{\Kmr}} \rar & \wh{\Emr}^{\Kmr}.
\end{tikzcd}
\]
This leads to us to the map $(\uptau_{(\upalpha, \upepsilon)} \lambda)_{\br{\Gmr}/\br{\Hmr}}(\Wmr)  $ which is the composite 
\[
\begin{tikzcd}
\left( \uptau_{(\upalpha, \upepsilon)}\Emr^\Kmr\right)_{\br{\Gmr}/\br{\Hmr}}(\Wmr)  \hspace{-2pt} :=  \underset{U \in \Vscr^{\circ}_{\br{\Hmr}}(\upiota_{\br{\Hmr}}(\upalpha \oplus \upepsilon)) }{\holim} (\wh{\Emr^\Kmr})_{\br{\Gmr}/\br{\Hmr}}(\Wmr \oplus \Umr) \dar \hspace{100pt}\\
\underset{\Umr \in \Vscr^{\circ}_{\br{\Hmr}}(\upiota_{\br{\Hmr}}(\upalpha \oplus \upepsilon)) }{\holim} \wh{\Emr}_{\Gmr/\Hmr}^{\Kmr} (\Wmr \oplus \Umr) \dar[equal]  \\
\underset{\Umr \in \Vscr^{\circ}_{\br{\Hmr}}(\upiota_{\br{\Hmr}}(\upalpha \oplus \upepsilon)) }{\holim} \left(\wh{\Emr}_{\Gmr/\Hmr} (\uppi^*\Wmr \oplus \uppi^*\Umr)\right)^{\Kmr} \dar[cong]  \\
\hspace{100pt}  \left(\underset{\Umr \in \Vscr^{\circ}_{\Hmr}(\upiota_\Hmr(\uppi^*\upalpha \oplus \uppi^*\upepsilon)) }{\holim} \wh{\Emr}_{\Gmr/\Hmr}(\uppi^*\Wmr \oplus \Umr) \right)^\Kmr =:  \left( \uptau_{(\uppi^*\upalpha, \uppi^*\upepsilon)}\Emr\right)_{\br{\Gmr}/\br{\Hmr}}^{\Kmr}(\Wmr), 
\end{tikzcd} 
 \] 
where the last isomorphism is a consequence of  the fact that \[ \uppi^*: \Vscr^{\circ}_{\br{\Hmr}}(\upiota_{\br{\Hmr}}(\upalpha \oplus \upepsilon) )\longrightarrow \Vscr^{\circ}_{\Hmr}(\upiota_\Hmr(\uppi^*\upalpha \oplus \uppi^*\upepsilon)) \] 
is an isomorphism of $\br{\Hmr}$-categories\footnote{Since  $\Kmr$ acts trivially on $\Vscr^{\circ}_{\Hmr}(\upiota_\Hmr(\uppi^*\upalpha \oplus \uppi^*\upepsilon)$,  it admits an action of $\br{\Hmr}$.}. 
Further, we have a commutative diagram 
\[ 
\begin{tikzcd}
\Emr^\Kmr \rar["\eta^{\Emr^\Kmr}_{(\upalpha, \upepsilon)}"] \ar[dr, "(\eta^{\Emr}_{(\upalpha, \upepsilon)})^{\Kmr}"'] & \uptau_{\upalpha \oplus \upepsilon}\Emr^\Kmr \dar["\uptau_{\upalpha \oplus \upepsilon} \lambda"] \\
& \left( \uptau_{\uppi^*\upalpha \oplus \uppi^*\upepsilon}\Emr\right)_{\br{\Gmr}/\br{\Hmr}}^{\Kmr}
\end{tikzcd}\] 
of $2$-morphisms, which leads to  \eqref{map:uplambda} following \Cref{defn:tau}. 
\end{proof}
\section{ Equivariant Weiss derivatives} \label{sec:Weissder}

Notice that any two finite dimensional $\Hmr$-representations $\Umr$ and $\Umr'$,  the space of  morphisms $\Mor_{\Hmr}(\Umr, \Umr')$ is  the  Stiefel manifold of   $|\Umr|$-frames in $\RR^{|\Umr'|}$  on which $\Hmr$-acts by conjugation. Given $\upalpha \in \Rep_\Gmr$,, and $\Umr, \Umr'  \in \Rep_\Hmr$,  we consider the equivariant bundle 
\begin{equation} \label{bundlealpha}
 \upgamma_{\Hmr, \upalpha}(\Umr, \Umr') := 
\begin{tikzcd}
\{ 
({\sf f}, {\sf v}):  {\sf f} \in \Mor_{\Hmr}(\Umr, \Umr') \text{ and } {\sf v} \in \upiota_\Hmr (\upalpha) \otimes {\sf img}({\sf f})^{\perp}
\} \dar \\ 
\Mor_{\Hmr}(\Umr, \Umr')
\end{tikzcd}  
\end{equation}
to construct a functor 
\[ 
\begin{tikzcd}
\Jfrak^{\upalpha}_\Gmr: \Ocal^{\mr{op}}_{\Gmr} \rar & \Cat,
\end{tikzcd}
\]
as follows.

The $\Hmr$-category $\Jfrak^{\alpha}_{\Gmr}(\Gmr/\Hmr)$ has the same objects as $\Rep_\Hmr$ but the space of morphisms is the  Thom space
 \[ \Mor_{\Hmr, \upalpha}(\Umr, \Umr') :=  \Th(\upgamma_{\Hmr, \upalpha}(\Umr, \Umr')) \]
 for  any two  $\Hmr$-representation $\Umr$ and $\Umr'$. The composition of morphisms is induced by the map of  $\Hmr$-equivariant vector bundles
\begin{equation} \label{eqn:compJ}
 \upgamma_{\Hmr, \upalpha}(\Umr', \Umr'') \times \upgamma_{\Hmr, \upalpha}(\Umr, \Umr') \to \upgamma_{\Hmr, \upalpha}(\Umr, \Umr'')
\end{equation}
which sends $\left((\sff, \sfv), (\sfg, \sfw) \right)  \mapsto (\sff \circ \sfg, \sfv + \sff_*(\sfw))$. One may think of $\Jfrak^{\upalpha}_\Gmr$ as an `enlargement of $\Jfrak_\Gmr$ by $\upalpha$'. Using the zero sections of Thom spaces  we get a $1$-morphism 
\begin{equation} \label{eqn:zerosec}
\begin{tikzcd}
\upzeta^{\upalpha} : \Jfrak_\Gmr \rar & \Jfrak^{\upalpha}_{\Gmr}
\end{tikzcd}
\end{equation}
for all $\upalpha \in \Rep_\Gmr$ and the inclusion of bundles $\alpha \subset \alpha'$ leads to  
\begin{equation} \label{eqn:relzerosec}
\begin{tikzcd}
\upzeta^{\upalpha,  \upalpha'} : \Jfrak_\Gmr^{\upalpha} \rar & \Jfrak^{\upalpha'}_{\Gmr},
\end{tikzcd}
\end{equation}
 and together they satisfy 
\[  \upzeta^{\upalpha, \upalpha'}\circ  \upzeta^{\upalpha} = \upzeta^{\upalpha'}\]
\begin{equation} \label{compzerosec}
\upzeta^{\upalpha', \upalpha''} \circ \upzeta^{\upalpha, \upalpha'} = \upzeta^{\upalpha'', \upalpha}
\end{equation} 
whenever  $\upalpha \subset \upalpha' \subset \upalpha''$. 

\begin{rmk} When $\upalpha = {\bf 0}$, then $\Jfrak^{\bf 0}_{\Gmr}$ is essentially $\Jfrak_{\Gmr}$ except that each morphism space has an additional disjoint base point, i.e.,  
\[ 
\Mor_{ \Hmr, {\bf 0}}(\Umr, \Umr') = \Mor_{\Hmr}(\Umr, \Umr')_+
\]
for  any pair of  $\Umr, \Umr' \in \Rep_\Hmr$. 
\end{rmk}

\begin{defn} \label{defn:derivative} For an $\upalpha \in \Rep_\Gmr$,  the {\bf $\upalpha$-th derivative} of  $\Emr: \Jfrak_\Gmr \longrightarrow \Tfrak_\Gmr$ is the  $1$-morphism 
\[ 
\begin{tikzcd}
\Emr^{(\alpha)}: \Jfrak^{\upalpha}_\Gmr \rar & \Tfrak_\Gmr
\end{tikzcd}
\]
where $\Emr^{(\upalpha)}_{\Gmr/\Hmr}$ is the Kan extension of $\Emr$
\[ 
\begin{tikzcd}
\Rep_{\Hmr} \ar[rr, "\Emr_{\Gmr/\Hmr}"] \dar["\upzeta^{\alpha}_{\Gmr/\Hmr}"']  && \Hmr\Top \\
\Jfrak^{\upalpha}_\Gmr \ar[rru, dashed, "\Emr^{(\upalpha)}_{\Gmr/\Hmr}"']
\end{tikzcd}
\]
along $\upzeta^{\alpha}_{\Gmr/\Hmr}$. 
\end{defn} 

\subsection{Derivatives at infinity.} \  

Let $\Hmr$ be a sub group of $\Gmr$. Then for any  $1$-morphism $\Fmr: \Jfrak^{\upalpha}_\Gmr \longrightarrow \Tfrak_{\Gmr}$, we get a continuous pairing 
\[
\begin{tikzcd}
\Mor_{\Hmr, \upalpha}(\Umr, \Umr \oplus \Vmr) \sma \Fmr_{\Gmr/\Hmr}(\Umr) \rar & \Fmr_{\Gmr/\Hmr}(\Umr \oplus \Vmr), 
\end{tikzcd}
\]
where $\Umr$ and $\Vmr$ are  orthogonal $\Hmr$-representations. If we restrict this pairing 
 to the canonical inclusion map $\iota: \Umr \hookrightarrow \Umr \oplus \Vmr $, we get a system of maps
 \begin{equation} \label{map:UV}
\begin{tikzcd} 
\upupsilon_{\Vmr, \Umr}^\Fmr:\Smr^{\upalpha \otimes \Vmr} \sma \Fmr_{\Gmr/\Hmr}(\Umr) \rar & \Fmr_{\Gmr/\Hmr}(\Umr \oplus \Vmr)
\end{tikzcd}
\end{equation}
which is equivalent to a genuine $\Gmr$-equivariant orthogonal prespectrum  $\Theta\Fmr$ (whose $(\upalpha \otimes \Umr)$-th space is $\Fmr_{\Gmr/\Gmr}(\Umr)$).  
\begin{notn}  Let  $\wh{\Theta \Fmr}$ denote the genuine $\Gmr$-spectrum  obtained by 
spectrifying $\Theta\Fmr$ (see \cite{MayMandell}). 
\end{notn}
 Since $1$-dimensional $\Gmr$-representations play the  role of directions, we introduce the following notation.  
\begin{notn} Let $\Ucal^{(1:\upalpha)}_\Gmr$ denote the sub-universe of the complete $\Gmr$-universe generated by 
\[ (1: \upalpha) :=\{ \upalpha \otimes \upepsilon_1 \otimes \dots \otimes \upepsilon_k: |\upepsilon_i|  =1 \} \]
and let $\Sp^\Gmr_{(1:\upalpha)}$ denote category of orthogonal $\Gmr$-spectra in the universe $\Ucal^{(1:\upalpha)}_\Gmr$. More generally, we write $\Ucal^{(1:\upalpha)_\sfR}_{(\Gmr)}$ for the sub-universe of  $\Ucal_{\Gmr}^{(\sfR)}$ generated by 
\[ (1: \upalpha)_{\Rmr} :=\{ \upalpha \otimes \upepsilon_1 \otimes \dots \otimes \upepsilon_k: |\upepsilon_i|  =1 \text{ and } \upepsilon_i \subset \Ucal^{(\sfR)}_\Gmr \} \]
for any $\upalpha \subset \Ucal_{\Gmr}^{(\sfR)}$.  
\end{notn}

Given  \eqref{map:UV}, we can construct for each $\Umr \in \Ucal_{\Hmr}$,  a spectrum \[ {\bf \Theta} \Fmr(\Umr) \in \Sp^{\Gmr}_{\upiota_\Hmr(1:\upalpha)}\] by spectrifying the prespectrum whose $ (\upalpha \otimes \Wmr)$-th space,  for a finite $\Wmr \in \Ucal_{\Hmr}^{\upiota_\Hmr(1:\upalpha)}$ is $\Fmr_{\Gmr/\Hmr}(\Umr \oplus \Wmr)$. Then the collection
\begin{equation} \label{eqn:system}
 {\bf \Theta}(\Fmr):= \bigcup_{\Hmr \subset \Gmr } \lbrace {\bf \Theta} \Fmr(\Umr)  \in \Sp^{\upiota_\Hmr(1:\upalpha)}: \Umr \underset{\mr{finite}}{\subset} \Ucal_\Hmr   \rbrace
 \end{equation}
 form an $\upalpha$-system of spectra in the sense of the following definition.
\begin{defn} \label{defn:alphasystem} Fix $\upalpha \in \Ucal^{(\sfR)}_\Gmr$. Then an {\bf  $\upalpha$-system of spectra} in the universe  $\Ucal^{(\sfR)}_{\Gmr}$ is a collection   \[ \upchi := \bigcup_{\Hmr \subset \Gmr} \{ \upchi({\Umr}) \in  \Sp^{\Gmr}_{\upiota_\Hmr(1: \upalpha)_\sfR}: \Umr \underset{\mr{finite}}\subset \upiota_\Hmr(\Ucal_{\Hmr}^{(\sfR)})   \} \] 
with maps
\[ 
\begin{tikzcd}
f_{\Umr, \Vmr}:\Smr^{\upiota_\Hmr(\upalpha) \otimes \Umr^{\perp}} \sma  \upchi(\Umr) \ \rar & \upchi(\Vmr)
\end{tikzcd}
\]
whenever $\Umr \subset  \Vmr \in \upiota_\Hmr(\Ucal_{\Hmr}^{(\sfR)}) $ such that
\begin{itemize}
\item
 $f_{ \Vmr, \Wmr} \circ f_{\Umr, \Vmr} = f_{ \Umr, \Wmr}$, 
 \item $f_{\Umr, \Vmr}$ is an equivalence whenever $\Umr^{\perp} \subset \Ucal_\Gmr^{\upiota_\Hmr(1:\upalpha)_{\sfR}}$,  
 \item $\upiota_\Kmr\upchi(\Umr) = \upchi(\upiota_\Kmr\Umr) $ and $\upiota_\Kmr f_{\Umr, \Vmr} = f_{\upiota_\Kmr \Umr, \upiota_\Kmr \Vmr}$ whenever  $\Kmr \subset \Hmr \subset \Gmr$. 
 \end{itemize}
\end{defn}

The inclusion map $\iota: \Ucal^{\upiota_\Hmr(1:\upalpha)}_\Hmr \hookrightarrow \Ucal_\Hmr$ induces an adjuction 
\begin{equation} \label{eq:CofU}
\begin{tikzcd}
\iota_* : \Sp^{\Hmr}_{(1:\upalpha)} \rar[shift left] & \lar[shift left] \Sp^{\Hmr} : \iota^*,
\end{tikzcd} 
\end{equation}
and by construction, there is a natural map 
\begin{equation} \label{map-to-genuine}
\begin{tikzcd}
\wh{\bf r}_{\Vmr}:{\bf \Theta}(\Fmr)(\Umr) \rar &\iota^*(\Sigma^{\upiota_\Hmr(\upalpha) \otimes \Umr}  \upiota_\Hmr(\wh{\Theta \Fmr})) 
\end{tikzcd}
\end{equation}
for any $\Umr \in \Rep_\Hmr$,   which may not be an equivalence  in general.  
\begin{conv}  We will follow \cite{MayMandell} and adopt the convention that $\Gmr$-spectra indexed by the zero universe  are simply $\Gmr$-spaces.
\end{conv}

\begin{rmk}
On one extreme when  $\upalpha = {\bf 0}$ in \Cref{defn:alphasystem}, an $\upalpha$-system is simply a prespectrum in $\Ucal^{(\sfR)}_\Gmr$. On the other extreme when  $\Ucal^{(1:\upalpha)_\sfR}_\Gmr = \Ucal^{(\sfR)}_\Gmr$,  the $\upalpha$-system ${\bf \Theta}$ determined by  $\Theta({\bf 0}) \in \Sp^{\Gmr}_\sfR$ as 
\[ 
{\bf \Theta}(\Umr) \simeq \Sigma^{\upiota_\Hmr(\upalpha) \otimes \Umr} \upiota_\Hmr{\bf \Theta}({\bf 0})
\]
for any  $\Umr \underset{\mr{finite}} \subset \Ucal_\Gmr^{(\sfR)}$. 
\end{rmk}

\begin{defn} \label{defn:der-at-infty} Given $\upalpha \in \Rep_\Gmr$ and a $1$-morphism $\Emr: \Jfrak_\Gmr \longrightarrow \Tfrak_\Gmr$,  we call  the $\upalpha$-system
 $ {\bf \Theta}(\Emr^{(\upalpha)})$
  the {\bf $\upalpha$-th derivative of $\Emr$ at $\infty$}.
\end{defn}

For a $\Gmr$-representation $\upalpha$, one may view  $\Omr(\upalpha)$ as a functor
\[ 
\begin{tikzcd}
\mr{O}(\upalpha): \Ocal^{\mr{op}}_{\Gmr} \rar & \Cat,
\end{tikzcd}
\]
where $\Omr(\upalpha)(\Gmr/\Hmr)$ is the one object $\Hmr$-category with $\Omr(\upiota_{\Hmr}\upalpha)$ as the space of morphisms.  
For all $\Umr, \Umr' \in \Rep_{\Hmr}$,  $\Omr(\upiota_{\Hmr}(\upalpha))$ acts on $\upgamma_{\Hmr, \upiota_{\Hmr}(\upalpha)}(\Umr, \Umr')$ such that the action map
\[ 
\begin{tikzcd}
\Omr(\upiota_{\Hmr}(\upalpha)) \times \upgamma_{\Hmr, \upiota_\Hmr(\upalpha)}(\Umr, \Umr') \rar & \upgamma_{\Hmr, \upiota_\Hmr(\upalpha)}(\Umr, \Umr'), 
\end{tikzcd}
\] 
is an $\Hmr$-equivariant map. This results in a $1$-morphism
\[ 
\begin{tikzcd}
\Omr(\upalpha) \times \Jfrak^{\upalpha}_\Gmr \rar & \Jfrak^{\upalpha}_\Gmr, 
\end{tikzcd}
\]
which maybe regarded as the action of $\Omr(\upalpha)$ on $\Jfrak^{\upalpha}_\Gmr$. Thus, for every $\sfg \in \Omr(\upalpha)$, there is an invertible $1$-morphism $\sfg : \Jfrak^{\upalpha}_\Gmr \longrightarrow \Jfrak^{\upalpha}_\Gmr$.

\begin{rmk} \label{rmk:DerStabSym} Since the $\alpha$-th derivative of 
$ \Emr^{(\upalpha)}: \Jfrak_{\Gmr}^{\upalpha} \longrightarrow \Tfrak_{\Gmr}$  is defined as a right Kan extension and $\Jfrak_{\Gmr}$ has a trivial action of $\Omr(\upalpha)$, it follows that the composition 
\[ 
\begin{tikzcd}
 \Jfrak_{\Gmr}^{\upalpha}  \rar["\sfg^{-1}"] & \Jfrak_{\Gmr}^{\upalpha}  \rar["\Emr^{(\upalpha)}"] & \Tfrak_{\Gmr} 
\end{tikzcd}
\]
 equals  $\Emr^{(\upalpha)}$ for all $\sfg \in \Omr(\upalpha)$. Consequently, each spectrum in a system  ${\bf \Theta}( \Emr^{(\upalpha)})$ (as well as $\wh{\Theta\Emr}^{(\upalpha)}$) admits a n\"aive action of $\Omr(\upalpha)$ and the structure maps of the system is compatible with this action (compare \cite[$\mathsection$3]{WeissCalc}). 
\end{rmk}
\subsection{Restrictions and Weiss derivatives}  \label{subsec:resder} \ 

Let $\Kmr$ be a subgroup of $\Gmr$. It is easy to see that the restriction of $\Jfrak_\Gmr^{\alpha}$ to $\Kmr$, i.e., the composite functor 
\[ 
\begin{tikzcd}
\upiota_\Kmr \Jfrak_\Gmr^{\upalpha}: \Ocal^{\mr{op}}_\Kmr \rar["\upiota_\Kmr"] & \Ocal^{\mr{op}}_\Gmr \rar["\Jfrak_\Gmr"] & \Cat, 
\end{tikzcd}
\]
where $\upiota_\Kmr$ as in \eqref{resfunK}, is the functor $\Jfrak_{\Kmr}^{\upiota_{\Kmr}(\upalpha)}$. Consequently,   \[ \upiota_\Kmr (\Emr^{(\upalpha)}) = (\upiota_\Kmr\Emr)^{(\upiota_\Kmr(\upalpha))} \] for any  $\Emr: \Jfrak_\Gmr \longrightarrow \Tfrak_\Gmr$. 

It is important to note that restriction of $\Ucal_\Gmr^{(1:\upalpha)}$ to the subgroup $\Kmr$ is  subuniverse of  $\Ucal_{\Kmr}^{(1:\upiota_\Kmr(\upalpha))}$
\[ 
\begin{tikzcd}
\sfi^\Kmr: \upiota_\Kmr(\Ucal_\Gmr^{(1:\upalpha)}) = \Ucal_{\Kmr}^{\upiota_\Kmr(1:\upalpha)} \rar[hook] & \Ucal_{\Kmr}^{(1:\upiota_\Kmr(\upalpha))}, 
\end{tikzcd}
\]
 which  may not be an equality in genera. Such a situation may arise if there exists a $1$-dimensional $\Kmr$-representation that is not a restriction  of a $\Gmr$-representation\footnote{The sign representation of $\Cmr_2$ is not a restriction of a $\Cmr_4$-representation.}.  Let 
 \begin{equation} \label{C-of-U-1}
\begin{tikzcd}
\sfi_*^{\Kmr} : \Sp^{\Hmr}_{\upiota_\Hmr(1:\upalpha)} \rar[shift left] & \lar[shift left] \Sp^{\Hmr}_{\upiota_\Hmr(1:\upiota_\Kmr(\upalpha))} : \sfi^{\Kmr*},
\end{tikzcd} 
\end{equation}
denote the adjunction induced by $\sfi^\Kmr$ for any subgroup $\Hmr \subset \Kmr$. 

Given a $1$-morphism $\Fmr: \Jfrak^{\upalpha}_\Gmr \longrightarrow \Tfrak_\Gmr$ and  an $\Hmr$-representation $\Umr$ for $\Hmr \subset \Kmr \subset \Gmr$,    the restriction  $\upiota_\Hmr{\bf \Theta}\Fmr(\Umr)$  may not equal $\upiota_\Hmr{\bf \Theta} \upiota_\Kmr\Fmr(\Umr)$ as they are spectra in different universes. Instead, there is a map 
\begin{equation} \label{SpRes} 
\begin{tikzcd}
\sfr_\Kmr(\Umr):\sfi^\Kmr_* {\bf \Theta}\Fmr(\Umr)  \rar &  {\bf \Theta} \upiota_\Kmr (\Fmr)(\Umr),
\end{tikzcd}
\end{equation}
in $\Sp^\Hmr_{\upiota_\Hmr(1: \upiota_\Kmr(\upalpha))}$ which is an equivalence when $\upiota_\Kmr(\Ucal_\Gmr^{(1: \upalpha)}) = \Ucal_{\Kmr}^{(1:\upiota_\Kmr(\upalpha))}$. This leads us to the following result. 
\begin{lem} \label{lem:resderinf} Suppose  $\Emr: \Jfrak_\Gmr \longrightarrow \Tfrak_\Gmr$ is a $1$-morphism in ${\bf Fun}(\Ocal^{\mr{op}}_{\Gmr}, \Cat)$ and $\upalpha \in \Rep_\Gmr$, then there exists a $\Omr(\upiota_\Hmr(\upalpha))$-equivariant map 
\[ 
\begin{tikzcd}
\sfr_\Kmr(\Umr):\sfi^\Kmr_* {\bf \Theta}\Emr^{(\upalpha)}(\Umr)  \rar &  {\bf \Theta} \upiota_\Kmr (\Emr^{(\upalpha)})(\Umr),
\end{tikzcd}
\]
in $\Sp^{\Hmr}_{\upiota_\Hmr(1: \upiota_\Kmr(\upalpha))}$ for all subgroup $\Kmr$ of $\Gmr$, which is an equivalence if $\upiota_\Kmr(\Ucal_\Gmr^{(1)}) = \Ucal_{\Kmr}^{(1)}$. 
\end{lem}

\subsection{Fixed points and Weiss derivatives} \ 

We first consider the case when  $\Kmr$ is a normal subgroup of $\Gmr$. Let  $\br{\Gmr}$ denote the quotient group $\Gmr/\Kmr$ and $\uppi: \Gmr \twoheadrightarrow \overline{\Gmr}$ denote  the quotient map.  
Given $\Emr: \Jfrak_{\Gmr} \longrightarrow \Tfrak_{\Gmr} $
and $\upalpha \in \Rep_{\br{\Gmr}}$, the $\upalpha$-th derivative of the fixed-point functor $\Emr^\Kmr$ can be compared to the $\Kmr$-fixed point functor of  $\Emr^{(\uppi^*\upalpha)}$, where the later is constructed as follows.
 
For  $\br{\Hmr} \subset \br{\Gmr}$, we let $\Hmr$ be the subgroup  $\uppi^{-1} (\br{\Hmr}) \subset \Gmr$. Let $\Umr, \Umr'$ be $\br{\Hmr}$-representations. Notice that the action of   $\Kmr$ on  $\upgamma_{\Hmr, \uppi^*\upalpha}(\uppi^*\Umr, \uppi^*\Umr')$ is trivial. Therefore, it admits an action of $\br{\Hmr}$, and as an $\br{\Hmr}$-equivariant vector bundle, it is isomorphic to $\upgamma_{\br{\Hmr}, \upalpha}(\Umr, \Umr')$. Consequently, there are $1$-morphisms 
\[ \begin{tikzcd}
\uppi^*: \Jfrak_{\br{\Gmr}}^{\upalpha^\Kmr} \rar[shift right] &\lar[shift right] \upphi^*\Jfrak_{\Gmr}^{\upalpha} : (-)^{\Kmr}
\end{tikzcd} 
 \] 
such that $(\uppi^{*}(-))^{\Kmr}$ is the identity. 
\begin{defn} \label{defn:Efix} For a $1$-morphism $\Fmr: \Jfrak_{\Gmr}^{\upalpha} \longrightarrow \Tfrak_{\Gmr}$ in ${\bf Fun}( \Ocal_{\Gmr}^{\mr{op}}, \Cat)$, define its $\Kmr$-fixed points $\Fmr^\Kmr$ as the composite $1$-morphism 
\[ 
\begin{tikzcd}
\Jfrak_{\br{\Gmr}}^{\upalpha^\Kmr} \rar["\uppi^*"] &  \upphi^*\Jfrak_{\Gmr}^{\upalpha} \rar["\upphi^*\Fmr"]  & \upphi^*\Tfrak_{\Gmr} \rar["(-)^\Kmr"] & \Tfrak_{\br{\Gmr}}
\end{tikzcd}
\]
in ${\bf Fun}( \Ocal_{\br{\Gmr}}^{\mr{op}}, \Cat)$. 
\end{defn}
\begin{lem} \label{lem:derfix} For any $\upalpha \in \Rep_{\br{\Gmr}}$ and  $\Emr: \Jfrak_\Gmr \longrightarrow \Tfrak_\Gmr$, there is a $2$-morphism 
\[ 
 (\Emr^{\Kmr})^{(\upalpha)}  \longrightarrow (\Emr^{(\uppi^*\upalpha )})^\Kmr
 \] 
in  ${\bf Fun}(\Ocal^{\mr{op}}_{\br{\Gmr}}, \Cat)$. 
\end{lem}
\begin{proof} Since $(\Emr^\Kmr)^\upalpha$ is the right Kan extension of $\Emr^\Kmr$ along $\upzeta^{\upalpha}:\Jfrak_{\br{\Gmr}} \hookrightarrow \Jfrak_{\br{\Gmr}}^{\upalpha}$, and the fact that we have a commutative diagram 
\[ 
\begin{tikzcd}
\Jfrak_{\br{\Gmr}} \dar["\upzeta^{\upalpha}"'] \rar["\uppi^*"] &  \upphi^*\Jfrak_{\Gmr} \dar \rar["\upphi^*\Emr"]   & \upphi^*\Tfrak_{\Gmr} \rar["(-)^\Kmr"] & \Tfrak_{\br{\Gmr}} \\
\Jfrak_{\br{\Gmr}}^{\upalpha} \rar["\uppi^*"] & \upphi^*\Jfrak_{\Gmr}^{\uppi^*\upalpha}  \ar[ru, "\upphi^*\Emr^{(\uppi^*\upalpha)}"']
\end{tikzcd}
\]
implies the existence of a $1$-morphism $  (\Emr^{\Kmr})^{(\upalpha)} \longrightarrow (\Emr^{(\uppi^*\upalpha)})^\Kmr$. 
\end{proof}

\begin{claim} \label{clm:derfixinf} Suppose $\Kmr$ is a normal subgroup of $\Gmr$, then for any $1$-morphism $\Fmr: \Jfrak_\Gmr^{\upalpha} \longrightarrow \Tfrak_\Gmr$, there is a natural map 
\[ 
\begin{tikzcd}
\sff_\Kmr(\Umr): {\bf \Theta}\Fmr^\Kmr(\Umr)\rar & ({\bf \Theta}\Fmr(\uppi^*\Umr))^\Kmr
\end{tikzcd}
\]

in $\Sp^{\br{\Gmr}}_{(1:\upalpha)}$.
\end{claim} 
\begin{proof} Let $\br{\Hmr} \subset \br{\Kmr}$. Then for any  $\Umr \in \Rep_{\br{\Hmr}}$ for  $(\uppi^*\Umr)^\Kmr = \Umr$ and  the $\Kmr$-fixed points of the universe $\Ucal_\Gmr^{(1:\upalpha)}$  is isomorphic to $\Ucal_{\br{\Gmr}}^{(1:\upalpha^\Kmr)}$. Further, we have a natural map \[ 
\begin{tikzcd}
\upkappa_\Umr:\Fmr^\Kmr (\Umr^\Kmr)  \rar &  \Fmr (\Umr)^\Kmr,
\end{tikzcd}
 \]
and a commutative diagram 
\[ 
\begin{tikzcd}
\Smr^{\upalpha^{\Kmr} \otimes \Vmr} \sma  \Fmr^\Kmr_{\br{\Gmr}/\br{\Hmr}} (\Umr) \ar["\upupsilon^{\Fmr^\Kmr}_{\Umr, \Vmr}",rrr] \dar["\upkappa_\Umr"'] &&& \Fmr^\Kmr (\Umr \oplus \Vmr) \dar["\upkappa_{\Umr \oplus \Vmr}"] \\
\Smr^{\upalpha^{\Kmr} \otimes \Vmr} \sma  \Fmr (\uppi^*\Umr)^\Kmr \ar[rrr, "(\upupsilon_{\uppi^*\Umr, \uppi^*\Vmr}^\Fmr)^\Kmr"'] &&& \Fmr (\uppi^*\Umr \oplus \uppi^*\Vmr)^\Kmr
\end{tikzcd}
\]
for any $\Umr, \Vmr \in \Rep_{\br{\Hmr}}$, where $\upupsilon^{(-)}_{(-), (-)}$  is the natural map in \eqref{map:UV}. This  
 leads us to the result. 
\end{proof}

\begin{cor} \label{cor:derinffix} For any $1$-morphism $\Emr: \Jfrak_\Gmr \longrightarrow \Tfrak_\Gmr$ and $\upalpha \in \Rep_\Gmr$, there is a natural map 
\[ 
\begin{tikzcd}
\sff_\Kmr(\Umr): {\bf \Theta}((\Emr^\Kmr)^{(\upalpha^\Kmr)})(\Umr)\rar & ({\bf \Theta}\Emr^{(\upalpha)}(\uppi^*\Umr))^\Kmr
\end{tikzcd}
\]
in $\Sp^{\br{\Gmr}}$ for all $\Umr \in \Rep_{\br{\Hmr}}$ for all subgroup $\br{\Hmr} \subset \br{\Gmr}$.
\end{cor}

When $\Kmr$ is not a normal subgroup of $\Gmr$ then we first restrict to the normalizer $\Nmr(\Kmr)$ and then apply the discussion above. As a consequence of  \Cref{cor:derinffix} and \Cref{lem:resderinf}, we have:
\begin{lem} \label{lem:derinffix} Suppose $\Kmr$ is a subgroup of $\Gmr$, $\upalpha \in \Rep_{\Gmr}$ and $\uppi:\Nmr(\Kmr) \twoheadrightarrow \Wmr(\Kmr)$ denote the quotient map. Let $\br{\Hmr} \subset \Wmr(\Kmr)$ and $\Hmr = \uppi^{-1}(\br{\Hmr})$.   Then for any $1$-morphism $\Emr: \Jfrak_\Gmr \longrightarrow \Tfrak_\Gmr$, there is a natural zig-zag 
\[ 
\begin{tikzcd}
&& ({\sfi}^{\Nmr(\Kmr)}_*{\bf \Theta}\Emr^{(\upalpha)}(\uppi^*\Umr))^{\Kmr} \dar["\sfr_{\Nmr(\Kmr)}(\uppi^*\Umr)^\Kmr"]  \\
{\bf \Theta} ((\Emr^{\Kmr})^{(\upalpha^\Kmr)})(\Umr) \ar[rr, "\sff_{\Kmr}(\Umr)"'] && \left( ({\bf \Theta}(\upiota_{\Nmr(\Kmr)}\Emr^{(\upalpha)})(\uppi^*\Umr)\right)^\Kmr
\end{tikzcd}
\]
in $\Sp^{\Wmr(\Kmr)}_{\upiota_{\br{\Hmr}}(1:\upalpha^\Kmr)}$ for any $\Umr \in \Rep_{\br{\Hmr}}$, where $\sfi^{\Nmr(\Kmr)}_*$ is as defined in \eqref{C-of-U-1}.
\end{lem}

\section{Equivariant Stiefel combinatorics} \label{sec:SC}

\bigskip
 If we set $\Umr' = \Umr \oplus \upiota_\Hmr(\upepsilon)$ in \eqref{eqn:compJ} for some $\upepsilon \in \Rep_\Gmr$ with $|\upepsilon| = 1$, and restrict the induced map on Thom spaces to the inclusion map 
$
\begin{tikzcd}
\iota : \Umr \rar[hook] & \Umr \oplus \upiota_\Hmr(\upepsilon)
\end{tikzcd}, $
 we get a map
\[ 
\begin{tikzcd}
\ofrak_{\iota}:  \Smr^{\upiota_{\Hmr}(\upalpha \otimes \upepsilon)} \sma \Mor_{\Hmr, \upalpha}(\Umr\oplus \upiota_{\Hmr} (\upepsilon), \Umr'')  \rar & \Mor_{\Hmr, \upalpha}(\Umr , \Umr'')
\end{tikzcd}
\]
whose cofiber is determined by  the following theorem. 
\begin{thm}[{\bf Equivariant Stiefel combinatorics I}] \label{thm:SC1} Suppose  $\upalpha$ and $\upepsilon$ are orthogonal $\Gmr$-representation such that $|\upepsilon| =1$. Then we have a cofiber sequence 
\[ 
\begin{tikzcd}
\Mor_{\Hmr, \upalpha}(\Umr\oplus \upiota_{\Hmr} (\upepsilon), \Umr'') \sma \Smr^{\upiota_{\Hmr}(\upalpha \otimes \upepsilon)} \rar["\ofrak_{\iota}"] & \Mor_{\Hmr, \upalpha}(\Umr , \Umr'') \rar & \Mor_{\Hmr, \upalpha \oplus \upepsilon}(\Umr, \Umr'')
\end{tikzcd}
\]
which is natural in the variables $\Umr, \Umr'' \in \Rep_\Hmr$.
\end{thm}
\begin{proof} The proof is essentially an equivariant analog of the argument in the nonequivariant case laid out in the proof of  \cite[Proposition 1.2]{WeissCalc}. 

First note that the reduced cone of $\ofrak_\iota$ is a quotient of 
\[  [0, \infty] \times \Mor_{\Hmr, \upalpha}(\Umr\oplus \upiota_\Hmr(\upepsilon), \Umr') \times \Smr^{\upiota_{\Hmr}(\upalpha \otimes \upepsilon)}, \]
and any element of this quotient other than the base point can be represented by a tuple $(t, ({\sf f}, {\sf w}), {\sf v})$,  where 
\begin{itemize}
\item[-] $t \in [0, \infty]$, 
\item[-]  ${\sf f}$ is a linear isometric embedding $\Umr\oplus \upiota_\Hmr(\upepsilon)\rightarrow \Umr'$, 
\item[-] ${\sf w}$ is a vector in $\upiota_\Hmr (\upalpha) \otimes {\sf img}({\sf f})^{\perp}$, and 
\item[-] ${\sf v}$ is a vector in $\upiota_\Hmr( \upalpha \otimes \upepsilon)$.
\end{itemize}
To such a tuple $(t, ({\sf f}, {\sf w}), {\sf v})$, we associate  
\[  (\sff  \circ \iota, \sfu )\in \Mor_{\Hmr, \upiota_\Hmr(\upalpha \oplus \upepsilon)}(\Umr, \Umr'), \]
where $\sfu \in \upiota_\Hmr (\upalpha \oplus \upepsilon)\otimes \Umr' $ is the element  
\[ {\sf w} + ({\sf f}|_{\upiota_\Hmr(\upepsilon)})_*({\sf v}) +  \varphi(\sff|_{\upiota_\Hmr(\upepsilon)})  \]
in which 
\begin{itemize}
\item[-] $({\sf f}|_{\upepsilon})_*$ is the composition 
\[ \upiota_{\Hmr}(\upalpha) \otimes \upepsilon \longrightarrow \upiota_{\Hmr}(\upalpha) \otimes \Umr' \longrightarrow \upiota_\Hmr(\upalpha \oplus \upepsilon) \otimes \Umr',\] 
\item[-]  $\varphi$ denotes the isomorphism $ \hom(\upiota_\Hmr(\upepsilon), \Umr')  \overset{\cong}\longrightarrow  \upiota_\Hmr(\upepsilon)\otimes \Umr' $.
\end{itemize}
It is now easy to check that this gives an $\Hmr$-equivariant homeomorphism between the reduced cone of $\ofrak_\iota$ and $\Mor_{\Hmr, \upiota_\Hmr(\upalpha \oplus \upepsilon)}(\Umr, \Umr')$, as desired. 
\end{proof}

To state the next  result, we introduce the following notation. 
\begin{notn}
We write $\hMor_{\Hmr}(\Umr, \Umr')$ for the morphisms in $\wh{\Jfrak}_{\Gmr/\Hmr}$, and $\hMor_{\Hmr, \upalpha}(\Umr, \Umr')$ for those in $\wh{\Jfrak}^{\upalpha}_{\Gmr/\Hmr}$.
\end{notn}
\begin{thm}[{\bf Equivariant Stiefel Combinatorics II}] \label{thm:SC2}  Given a  $\Gmr$-representation $\upalpha$, the (unreduced) cofiber of
\[ 
\begin{tikzcd}
\underset {\Vmr \in \Vscr^{\circ}(\upiota_\Hmr (\upalpha))}{\hocolim} \hMor_{\Hmr}(\Umr \oplus \Vmr, \Umr') \rar["\wh{\ofrak}_{\upalpha}"] & \Mor_{\Hmr}(\Umr, \Umr') 
\end{tikzcd}
\]
admits a homeomorphism  to $ \Mor_{\Hmr, \upalpha}(\Umr, \Umr')$ as an $\Hmr$-spaces which is natural in the variables $\Umr, \Umr' \in \Rep_{\Hmr}$. 
\end{thm}
\begin{proof} We first recall the underlying nonequivariant homeomorphism (in the notations used  in this paper) as laid out in \cite{WeissCalc} and then argue that the given homeomorphism is indeed an $\Hmr$-equivariant map. 
 
 In the proof of  \cite[Theorem 4.1]{WeissCalc}, the underlying homoemorphism between $ \Mor_{\Hmr, \upalpha}(\Umr, \Umr')$  and the unreduced mapping cone of $\wh{\ofrak}_{\upalpha}$ is induced by a map on the respective complements of $\Mor_{\Hmr}(\Umr, \Umr')_+$
 \[ 
 \begin{tikzcd}
 \Phi:  \Mor_{\Hmr, \upalpha}(\Umr, \Umr') \setminus \Mor_{\Hmr}(\Umr, \Umr') \rar & (0, \infty) \times \underset {\Vmr \in \Vscr^{\circ}(\upiota_\Hmr (\upalpha))}{\hocolim} \hMor_{\Hmr}(\Umr \oplus \Vmr, \Umr'), 
 \end{tikzcd}
 \]
   which is constructed as follows.

 Suppose we are given $(\sff, \sfw) \in  \Mor_{\Hmr, \upalpha}(\Umr, \Umr') \setminus \Mor_{\Hmr}(\Umr, \Umr') $, where  ${\sf f}: \Umr \hookrightarrow \Umr'$ is an isometric embedding  and ${\sf w} \in \upiota_\Hmr( \upalpha) \otimes \mr{coker}({\sf f})$. Then we may veiw 
 $\sfw$ as a map 
\begin{equation} \label{wmap}
\begin{tikzcd}
 {\sf w}:  \upiota_\Hmr(\upalpha) \rar & \mr{coker}({\sf f}),
 \end{tikzcd}
 \end{equation}
 and consider ${\sf w}^* {\sf w} : \res_{\Hmr}^{\Gmr} \upalpha \longrightarrow \res_{\Hmr}^{\Gmr} \upalpha $, which is self-adjoint operator, and therefore has eigenvalues  \[0 < \lambda_0 < \lambda_1 < \dots < \lambda_k. \]  Thus, 
 we have a decomposition (of underlying nonequivariant real vector spaces)
 \[ 
 \upiota_\Hmr(\upalpha) = \mr{ker}({\sf w}^*{\sf w} ) \oplus \Emr_0 \oplus \dots \oplus \Emr_{k},
 \]
where $\Emr_i$ is the eigenspace corresponding to $\lambda_i$. This data can be used to produce a point in 
\[
(0, \infty) \times  \underset{ \Vmr \in \Vscr_{\Hmr}^{\circ}(\upiota_\Hmr(\upalpha))}\hocolim \hMor_{\Hmr}(\Umr \oplus \Vmr, \Umr' )
\]
where we use the usual  model the homotopy colimit above, i.e. as the quotient of
\[
\coprod_{k\geq 0} \coprod_{ \phi: {\bf k} \to \Ccal} \Mor_{\Hmr}(\Umr \oplus \phi(0), \Umr' ) \times \Delta^k
\]
where $\Ccal$ denotes opposite of the $\Hmr$-category of $ \Vscr_{\Hmr}^{\circ}(\upiota_\Hmr(\upalpha))$. Explicitly, we let
\begin{itemize}
\item[-] $\phi$ be the  functor ${\bf k} \rightarrow \Ccal$ sending $r\leq k$ to $\Emr_0 \oplus \dots \oplus \Emr_{k-r}$, 
\item[-] ${\sf g} \in \Mor_{\Hmr}(\Umr \oplus \phi(0), \Umr' )$ be such that its restriction to $\Wmr$ is ${\sf f}$, and that its restriction to $\Emr_i$ is $\frac{1}{\sqrt{\lambda_i}}{\sf w}$,
\item[-] $\sfp\in \Delta^k$ be the point with $\frac{1}{\lambda_k}(\lambda_0, \lambda_1 - \lambda_0, \cdots, \lambda_k - \lambda_{k-1})$ as its barycentric coordinates, 
\item[-]  $t\in (0, \infty)$ be the number $\lambda_k$, 
\end{itemize}
and define $\Phi({\sf f}, {\sf w})  = (t, \phi, {\sf g}, \sfp)$. The underlying nonequivariant map $\Phi$ is shown to be a homeomorphism in \cite[Theorem 4.1]{WeissCalc}. We will complete the proof by showing that $\Phi$ is also an $\Hmr$-equivariant map. 

The action of $\sfh \in \Hmr$ on the $\Mor_{\Hmr, \upalpha}(\Umr, \Umr') \setminus \Mor_{\Hmr}(\Umr, \Umr')$ is given by 
\[ {\sf h}^{-1} \cdot ({\sf f}, {\sf w}) := ({\sf h}^{-1}{\sf f}{\sf h}, {\sf h}^{-1}{\sf w}{\sf h} ), \]
where ${\sf w}$ is viewed as the map in  \eqref{wmap}. The eigenvalues of \[ (\sfh^{-1} {\sf w} \sfh)^* (\sfh^{-1} {\sf w} \sfh) =\sfh^{-1} ({\sf w}^*{\sf w}) \sfh \] are the same as that of $ {\sf w}^*{\sf w}$, and the eigenspace of $\lambda_i$ is the subspace $\sfh^{-1} \Emr_i \sfh$ of 
$ \upiota_\Hmr(\upalpha)$. Thus 
\begin{eqnarray*}
\Phi(\sfh \cdot (\sff, \sfw)) &=&  \Phi({\sf h}^{-1}{\sf f}{\sf h}, {\sf h}^{-1}{\sf w}{\sf h})  \\
&=& (t, \sfh^{-1}\phi , \sfh^{-1}\sfg \sfh, \sfp) \\
&=& \sfh \cdot \Phi(\sff, \sfw), 
\end{eqnarray*}
i.e., $\Phi$ is an $\Hmr$-equivariant map.
 \end{proof}
\subsection{Consequences of Stiefel combinatorics} \ 

We  now discuss some of the important consequences of the Stiefel combinatorics which will be important in subsequent  sections. 

\begin{notn} Let $\Smr \upgamma_{\Hmr, \upalpha}(\Umr, \Umr')$ denote the unit sphere bundle of $\upgamma_{\Hmr, \upalpha}(\Umr, \Umr')$ so that $\Mor_{\Hmr, \upalpha}(\Umr, \Umr') := \Th(\upgamma_{\Hmr, \upalpha})$ is the unreduced cofiber of the projection map 
\begin{equation} \label{map:proj}
\begin{tikzcd}
{\sf p}_{\upalpha}: \Smr \upgamma_{\Hmr, \upalpha}(\Umr, \Umr') \rar & \Mor_{\Hmr}(\Umr, \Umr')  
\end{tikzcd}
\end{equation}
for any pair of  $\Hmr$-representations $\Umr$ and $\Umr'$. 
\end{notn}
\begin{prop} \label{prop:sc3}
There is a natural $\Hmr$-equiva homeomorphism
\[
 \underset{\Vmr \in \Vscr^{\circ}(\upiota_\Hmr (\upalpha))}{\hocolim} \hMor_{\Hmr}(\Umr \oplus \Vmr, \Umr') \cong \Smr \upgamma_{\Hmr, \upalpha}(\Umr, \Umr')
 \]
 for any pair $\Umr, \Umr' \in \Rep_\Hmr$. 
\end{prop}

\begin{proof}
This follows from the proof of \Cref{thm:SC2}, and the argument is identical to that of \cite[Proposition 4.2]{WeissCalc}. 
\end{proof}

Given a $\Gmr$-representation $\Umr$, one can vary $\Vmr'$ in $\Mor_{\Hmr, \alpha}(\upiota_\Hmr(\Umr), \Vmr')$  to obtain a $1$-morphism  
 \[ 
 \begin{tikzcd}
 {\bf Mor}_{\upalpha}(\Umr, - ) : \Jfrak_{\Gmr} \rar & \Tfrak_{\Gmr}
\end{tikzcd}
\]
in $\Func(\Ocal^{\mr{op}}_{\Gmr}, \Cat)$. Likewise, by varying $\Wmr'$ in $\Smr \upgamma_{\Hmr, \upalpha}(\upiota_{\Hmr}(\Umr), \Wmr')_+$, we obtain the  $1$-morphism
\[ 
\begin{tikzcd}
{\bf S}  \upgamma_{\upalpha}(\Umr, - ) : \Jfrak_{\Gmr} \rar & \Tfrak_{\Gmr}
\end{tikzcd}
\]
such that there is a $2$-morphism  
\[ 
\begin{tikzcd}
{\sf p}_{\upalpha}: {\bf S}  \upgamma_{\upalpha}(\Umr, - ) \rar & {\bf Mor}_{{\bf 0}}(\Umr, - )
\end{tikzcd}
\]
 induced by the projection map of \eqref{map:proj}.
 \begin{notn}
 Let $\Ecal_{\Gmr, \upalpha}$ denote the category of $1$-morphisms from $\Jfrak^{\upalpha}_\Gmr \longrightarrow \Tfrak_\Gmr$
in $\Func(\Ocal^{\mr{op}}_\Gmr, \Cat )$, i.e., objects of  $\Ecal_{\Gmr, \upalpha}$ are natural transformations $\Jfrak^{\upalpha}_\Gmr \Rightarrow \Tfrak_\Gmr$,  and  morphisms are natural transformations between such natural transformations (i.e., 3-morphisms). Alternatively, one can regard $\Ecal_{\Gmr, \upalpha}$ as a $\Cat$-valued presheaf
\[ 
\begin{tikzcd}
\Ecal_{\Gmr, \upalpha}: \Ocal^{\mr{op}}_\Gmr \rar & \Cat 
\end{tikzcd}
\]
sending $\Gmr/\Hmr$ to  ${\bf Fun}(\Jfrak_{\Gmr/\Hmr}^{\upalpha}, \Tfrak_{\Gmr/\Hmr})$.
\end{notn}
\begin{notn} \label{notn:nat}
 Given two objects $\Amr$ and $\Amr'$ of $\Ecal_{\Gmr, \upalpha}$, we let 
 $
\Ncal_{\upalpha}(\Amr, \Amr')
$  denote the set of $2$-morphisms  from $\Emr$ to $\Fmr$. 
\end{notn}
 
Yoneda's lemma implies that
\begin{align*}
 \Ncal_{\upalpha}(\Mor_{\upalpha}(\Umr, -), \Amr)_{\Gmr/\Hmr} &\cong \Ncal_{\upalpha}\left(\Mor_{\upalpha}(\Umr, -)_{\Gmr/\Hmr}, \Amr_{\Gmr/\Hmr}\right) \\
  &\cong \Ncal_{\upalpha}\left(\Mor_{\Hmr, \upalpha}(\upiota_{\Hmr}(\Umr), -), \Amr_{\Gmr/\Hmr}\right) \\
  &\cong \Amr_{\Gmr/\Hmr}\left(\upiota_{\Hmr}(\Umr)\right)
\end{align*}

for all subgroups  $\Hmr$ of $ \Gmr$ and $\Amr \in \Ecal_{\Gmr, \alpha}$. We will use the equation 
\begin{equation} \label{eqn:Yoneda}
\Ncal_{\upalpha}(\Mor_{\upalpha}(\Umr, -), \Amr) \cong  \Amr(\Umr)
\end{equation}
to denote the above Yoneda phenomena. 

Using \eqref{eqn:zerosec},  we get a restriction functor 
\[ 
\begin{tikzcd}
{\bf R}^{\upalpha \oplus \upalpha'}_{\upalpha}: \Ecal_{\Gmr, \upalpha \oplus \upalpha'}  \rar & \Ecal_{\Gmr, \upalpha}
\end{tikzcd}
\]
whose right adjoint 
\[ 
\begin{tikzcd}
{\bf I}^{\upalpha \oplus \upalpha'}_{\upalpha}: \Ecal_{\Gmr, \upalpha }  \rar & \Ecal_{\Gmr, \upalpha \oplus \upalpha'}
\end{tikzcd}
\]
is constructed using right Kan extensions and  will be called the induction functor  following \cite[$\mathsection$2]{WeissCalc}. 
\begin{rmk} \label{rmk:IRproperties} By construction, ${\bf I}^{\upalpha}_{\bf 0} (\Emr) = \Emr^{(\upalpha)}$ for any $\Emr \in \Ecal_{\Gmr, {\bf 0}}$, and  using \eqref{compzerosec}, it is straightforward to check that 
\begin{enumerate}
\item  ${\bf R}^{\upalpha \oplus \upalpha'}_{\upalpha} \circ {\bf R}^{\upalpha \oplus \upalpha' \oplus \upalpha''}_{\upalpha \oplus \upalpha'} = {\bf R}^{\upalpha \oplus \upalpha' \oplus \upalpha''}_{\upalpha}$, 
\item $ {\bf I}^{\upalpha \oplus \upalpha' \oplus \upalpha''}_{\upalpha \oplus \upalpha'} \circ {\bf I}^{\upalpha \oplus \upalpha'}_{\upalpha}  = {\bf I}^{\upalpha \oplus \upalpha' \oplus \upalpha''}_{\upalpha}$.
\end{enumerate}

\end{rmk}

\begin{cor} \label{cor:sc1cor}
Given $\Emr \in \Ecal_{\Gmr, {\bf 0}}$, there is a natural fiber sequence 
\[ 
\begin{tikzcd}
\Emr^{( \upalpha \oplus \upepsilon)}_{\Gmr/\Hmr}(\Umr) \rar  & \Emr_{\Gmr/\Hmr}^{(\upalpha)}(\Umr) \rar & \Omega^{\upiota_{\Hmr}(\upalpha \otimes \upepsilon)}\Emr_{\Gmr/\Hmr}^{(\upalpha)}(\Umr \oplus \upepsilon)
\end{tikzcd}
\]
for every $\Hmr \leq \Gmr$ and every $\Umr \in \Rep_\Hmr$. 
\end{cor}
\begin{proof}  
Applying $\Ncal_{\upalpha}(- , \Emr^{(\upalpha)})$ to the cofiber sequence of Stiefel combinatorics I,  i.e.  \Cref{thm:SC1},  we get  
\[ 
\begin{tikzcd}
\Mor_{\Hmr, \upalpha}(\Umr \oplus \upiota_\Hmr(\upepsilon), - ) \sma \Smr^{\upiota_\Hmr(\upalpha \otimes \upepsilon)} \dar \\
\Mor_{\Hmr, \upalpha}(\Umr, -) \dar \\
{\bf R}^{\upalpha \oplus \upepsilon}_{\upalpha}\Mor_{\Hmr, \upalpha \oplus \upepsilon}(\Umr, -)
\end{tikzcd} \overset{\Ncal_{\upalpha}(-, \Emr_{\Gmr/\Hmr}^{(\upalpha)})}\rightsquigarrow
\begin{tikzcd}
\Omega^{\upiota_\Hmr(\upalpha \otimes \upepsilon)} \Emr^{(\upalpha)}_{\Gmr/\Hmr}(\Umr \oplus \upiota_{\Hmr}(\upepsilon)) \\
\Emr^{(\upalpha)}_{\Gmr/\Hmr}(\Umr) \uar \\
\Emr^{(\upalpha \oplus \upepsilon)}_{\Gmr/\Hmr}(\Umr), \uar 
\end{tikzcd}
\]
where we use the identification \eqref{eqn:Yoneda} and 
\begin{eqnarray*}
\Ncal_{\upalpha}({\bf R}^{\upalpha \oplus \upepsilon}_{\upalpha} \Mor_{\upalpha \oplus \upepsilon}(\Umr, -), \Emr^{(\upalpha)}) &\cong& \Ncal_{\upalpha \oplus \upepsilon}( \Mor_{\upalpha \oplus \upepsilon}(\Umr, -), {\bf I}^{\upalpha \oplus \upepsilon}_{\upalpha}\Emr^{(\upalpha)}) \\
&\cong& \Emr^{(\upalpha \oplus \upepsilon)}(\Umr)
\end{eqnarray*}
following \Cref{rmk:IRproperties}. 
\end{proof}

\begin{cor} \label{cor:sc2cor}
For $\Emr \in \Ecal_{\Gmr, {\bf 0}}$, there is a natural fiber sequence 
\[ 
\begin{tikzcd}
\Emr^{(\upalpha)}_{\Gmr/\Hmr} (\Umr) \rar & \Emr_{\Gmr/\Hmr}(\Umr) \rar & \uptau_{\upalpha}\wh{\Emr}_{\Gmr/\Hmr}(\Umr)
\end{tikzcd}
\]
for every $\Hmr \leq \Gmr$ and every $\Hmr$-representation $\Umr$. 
\end{cor} 

\begin{proof} The result follows from applying  $\Ncal_{{\bf 0}}(- , \Emr^{(\upalpha)})$ to the cofiber sequence of Stiefel combinatorics II,  i.e.  \Cref{thm:SC2}.
\end{proof}

\section{Equivariant polynomial functors} \label{sec:poly}

The concept of polynomial functors in classical Weiss calculus also admits an equivariant generalization. We first introduce the notion of strongly polynomial following \cite[Definition 4.1.7]{Y24}. 
\begin{defn} \label{defn:Spolynomial} Suppose $\upalpha \in \Rep_\Gmr$. 
We say  $\mr{E}: \Jfrak_\Gmr \longrightarrow \Tfrak_\Gmr$ is {\bf strongly $\upalpha$-polynomial} if the  natural map 
\[
\begin{tikzcd}
\Emr_{\Gmr/\Hmr}(\Umr)  \rar & (\uptau_{\upalpha }\Emr)_{\Gmr/\Hmr}(\Umr)
 \end{tikzcd} 
 \]
 is a weak equivalence   for every $\Umr \in \Rep_\Hmr$ and $\Hmr \subset \Gmr$. 
\end{defn}
Analogous to traditional calculus, the $(n+k)$-th derivative of the $n$-th Taylor approximation of a Weiss functor is trivial (see \cite[Proposition 5.3]{WeissCalc}). The analogous result in  equivariant Weiss calculus is as follows.
\begin{thm} \label{thm:derpoly} Suppose  $\upalpha, \upalpha' \in \Rep_\Gmr$. 
Then any $1$-morphism $\Emr: \Jfrak_{\Gmr} \longrightarrow \Tfrak_{\Gmr}$ which is  strongly $\upalpha$-polynomial then 
\begin{equation} \label{null}
\Emr^{(\upalpha \oplus \upalpha')}_{\Gmr/\Hmr}(\Umr) \simeq \ast
\end{equation}
for all $\Umr \in \Rep_\Hmr$ and  $\Hmr \subset \Gmr$.  
\end{thm}
\begin{proof} 

By \Cref{rmk:IRproperties}
\begin{eqnarray*}
\Emr^{(\upalpha \oplus \upalpha')}_{\Gmr/\Hmr}(\Umr) &\cong& {\bf I}^{\upalpha \oplus \upalpha'}_{\bf 0} \Emr_{\Gmr/\Hmr}(\Umr) \\
&\cong & {\bf I}^{\upalpha \oplus \upalpha'}_{\upalpha } {\bf I}^{\upalpha }_{\bf 0} \Emr_{\Gmr/\Hmr}(\Umr) \\
&\cong & {\bf I}^{\upalpha \oplus \upalpha'}_{\upalpha }  \Emr^{(\upalpha)}_{\Gmr/\Hmr}(\Umr) \\
&\simeq& {\bf I}^{\upalpha \oplus \upalpha'}_{\upalpha }( \ast ) \\
& \simeq& \ast 
\end{eqnarray*}
as desired. 
\end{proof}
In \cite{Y24}, the term `strongly' is used because restriction of a strongly  $\upalpha$-polynomial functor $\Emr$ to the trivial group 
is  a polynomial  of degree  (less than or equal to) $|\upalpha|-1$ in the sense of \cite[Definition 5.1]{WeissCalc}. This motivates the following definition. 
\begin{defn} \label{defn:polynomial} Suppose $\upalpha, \upepsilon \in \Rep_\Gmr$  such that  $|\upepsilon| = 1$. 
We say  $\mr{E}: \Jfrak_\Gmr \longrightarrow \Tfrak_\Gmr$ is {\bf $\upalpha$-polynomial in the direction of $\upepsilon$} if $\Emr$ is strongly $\upalpha \oplus \upepsilon$-polynomial according to \Cref{defn:Spolynomial}. 
\end{defn}

In standard calculus, the $(n+k)$-th Taylor approximation of a degree $n$ polynomial function is the function itself for  all $k \in \NN$.
The analogous result in equivariant Weiss calculus takes the following form. 
\begin{thm} \label{thm:taylorapprox}If $\Emr \in \Ecal_{\Gmr, {\bf 0}}$ is an $\upalpha$-polynomial in the direction of $\upepsilon$. Then it is $(\upalpha \oplus \upalpha')$-polynomial  in the direction of $\upepsilon$ when $\upalpha' \in \Rep_\Gmr$ is a direct sum of $1$-dimensional representations. 
\end{thm}
\begin{proof} The proof will follow immediately from \Cref{lem:polyplus1}. 
\end{proof}
Following \cite{WeissCalc}, we make the following definition. 
\begin{defn} Given a $1$-morphism $\Emr: \Jfrak_{\Gmr} \longrightarrow \Tfrak_{\Gmr}$, we say a natural transformation 
$\eta: \Smr \to \Tmr$ between two functors $\Smr, \Tmr: \Rep_\Hmr \longrightarrow \Hmr\Top$ is an {\bf $\Emr$-substitution} if the induced map 
\[ 
\begin{tikzcd}
\eta^*: {\bf nat}(\Tmr, \Emr_{\Gmr/\Hmr}) \rar & {\bf nat}(\Smr, \Emr_{\Gmr/\Hmr}) 
\end{tikzcd}
\]
is an equivalence in $\Hmr\Top$, where ${\bf nat}(-, -)$ denotes the space of  natural transformations. 
\end{defn} 

\begin{lem} \label{lem:polyplus1} Suppose $\Emr \in \Ecal_{\Gmr, {\bf 0}}$ such that  $\Emr \longrightarrow \uptau_{\upalpha} \Emr$ is a pointwise equivalence then 
\[
\begin{tikzcd}
 \Emr \ar[r] & \uptau_{\upalpha \oplus \upepsilon} \Emr
 \end{tikzcd}
  \]
is also a pointwise equivalence for all $1$-dimensional representations $\upepsilon \in \Rep_\Gmr$. 
\end{lem}
\begin{proof}  Using \Cref{prop:sc3} and comparing it with \Cref{defn:tau}, we notice 
\[ (\uptau_{\upalpha}\Emr)_{\Gmr/\Hmr}(\Umr) = {\bf nat}(\Smr \upgamma_{\Hmr, \upalpha}(\Umr, -), \Emr_{\Gmr/\Hmr} ), \]
 and by assumption $\sfp_{\upalpha}$ in the diagram 
 \begin{equation} \label{diag:bun}
\begin{tikzcd}
\Smr \upgamma_{\Hmr, \upalpha}(\Umr, -) \rar["\sff"] \ar[dr, "\sfp_{\upalpha}"']  & \Smr \upgamma_{\Hmr, \upalpha \oplus \upepsilon}(\Umr, -) \dar["\sfp_{\upalpha \oplus \upepsilon}"] \\
& \Mor_{\Hmr}(\Umr, -)
\end{tikzcd}
\end{equation}
  is an $\Emr$-substitution.  Our goal is to show that $\sfp_{\upalpha \oplus \upalpha'}$  is an $\Emr$-substitute. This will follow if we show  ${\sf f}$ is an $\Emr$-substitution.  Let \[ \Smr \upgamma_{\Hmr, \upalpha }(\Umr, -) \boxtimes \Smr \upgamma_{\Hmr,  \upepsilon}(\Umr, -): \Rep_{\Hmr} \longrightarrow \Hmr\Top  \] denote the functor which sends an $\Hmr$-representation $\Wmr'$ to the fiberwise cartesian product of $\Smr \upgamma_{\Hmr, \upalpha }(\Umr, \Umr')$ and  $\Smr \upgamma_{\Hmr,  \upepsilon}(\Umr, \Umr') $.  Since  
\begin{equation} \label{suppdiag}
\begin{tikzcd}
\Smr \upgamma_{\Hmr, \upalpha }(\Umr, -) \boxtimes \Smr \upgamma_{\Hmr,  \upepsilon}(\Umr, -) \rar["p_2"] 
\dar["p_1"']  & \Smr \upgamma_{\Hmr,  \upepsilon}(\Umr, -) \dar  \\
\Smr \upgamma_{\Hmr, \upalpha }(\Umr, -) \rar["\sff"'] & \Smr \upgamma_{\Hmr, \upalpha \oplus \upepsilon}(\Umr, -)
\end{tikzcd}
\end{equation}
is a homotopy pushout,  it suffices to show that $p_2$ is an $\Emr$-substitution. 

Now observe,  there is an $\Hmr$-equivariant homeomorphism
\[ \Smr \upgamma_{\Hmr,  \upepsilon}(\Umr, \Umr') \cong \Mor_{\Hmr}(\Umr \oplus \upiota_\Hmr(\upepsilon), \Umr') \]
 sending  $(f: \Umr \to \Umr', v)$, where $v \in \upiota_\Hmr(\upepsilon) \otimes \im (f)^\perp  $, to the function 
 \[ 
 \begin{tikzcd}
 \tilde{f}: \Umr \oplus \upiota_\Hmr(\upepsilon) \rar & \Umr'
 \end{tikzcd}
  \] which restricts to $f$ on $\Umr$, and to the map $\tilde{v}$ on  $\upiota_\Hmr(\upepsilon)$ which corresponds to  $v$ under the isomorphism 
 $ \upiota_\Hmr(\upepsilon) \otimes \Umr' \cong  \Hom(\upiota_\Hmr(\upepsilon),\Umr' )$. Furthermore,  there is a natural $\Hmr$-equivariant homeomorphism
 \[ \Smr \upgamma_{\Hmr, \upalpha }(\Umr, \Umr') \boxtimes \Smr \upgamma_{\Hmr,  \upepsilon}(\Umr, \Umr') \cong {\sf r}^*\Smr\upgamma_{\Hmr, \upalpha}(\Umr, \Umr'), \] 
 where ${\sf r}: \Mor_{\Hmr}(\Umr \oplus \upiota_\Hmr(\upepsilon), \Wmr') \longrightarrow  \Mor_{\Hmr}(\Umr , \Umr') $ is the map induced the inclusion $\Umr \hookrightarrow \Umr \oplus \upiota_\Hmr(\upepsilon)$, and the map $p_2$ in \eqref{suppdiag} can be identified with  the projection of ${\sf r}^*\Smr\upgamma_{\Hmr, \upalpha}(\Umr, \Umr')$ to its base space.  
 
 Since \[ {\sf r}^*\upgamma_{\Hmr, \upepsilon}(\Umr, \Umr') \cong {\bf triv}_{\upiota_\Hmr(\upepsilon \otimes \upalpha)  }\oplus  \upgamma_{\Hmr, \upalpha}(\Umr \oplus \upiota_\Hmr(\upepsilon), \Umr'), \]
 where ${\bf triv}_{\upiota_\Hmr(\upepsilon \otimes \upalpha)  }$ is the trivial bundle with fiber $\upiota_\Hmr(\upepsilon \otimes \upalpha)$, and the unit sphere bundle of a Whitney sum is fiberwise join of the unit sphere bundles of the summands, we have a homotopy pushout diagram 
 \[ 
 \begin{tikzcd}
 \Smr(\upiota_\Hmr(\upepsilon \otimes \upalpha)) \times \Smr\upgamma_{\Hmr, \upalpha}(\Umr\oplus \upiota_\Hmr(\upepsilon), -)\ar[dd,"\Smr(\upiota_\Hmr(\upepsilon \otimes \upalpha)) \times \sfp_{\upalpha}"'] \rar & \Smr\upgamma_{\Hmr, \upalpha}(\Umr\oplus \upiota_\Hmr(\upepsilon), -) \ar[hook, dd, "i"]   \\ \\
  \Smr(\upiota_\Hmr(\upepsilon \otimes \upalpha)) \times \Mor_{\Hmr}(\Umr, -)  \rar &  {\sf r}^*\Smr\upgamma_{\Hmr, \upalpha}(\Umr,-)
 \end{tikzcd}
 \]
 in which the left vertical map is an $\Emr$-substitution by assumption. Consequently, the inclusion $i$ as well as the composition 
 \[ 
 \begin{tikzcd}
 \Smr\upgamma_{\Hmr, \upalpha}(\Umr \oplus \upiota_\Hmr(\upepsilon), -) \rar[hook, "i"] &  {\sf r}^*\upgamma_{\Hmr, \upepsilon}(\Umr, \Umr') \rar["p_2"] & \Mor_{\upalpha}(\Wmr \oplus \upiota_\Hmr(\upepsilon)) 
 \end{tikzcd}
 \]
are $\Emr$-substitutions. Therefore, $p_2$ is an  $\Emr$-substitution. 
\end{proof}

Our next goal is to show that ${\bf T}_{(\upalpha, \upepsilon)}\Emr$ is $\upalpha$-polynomial in the direction of $\upepsilon$ for all $\Emr \in \Ecal_{\Gmr, {\bf 0}}$ and prove \Cref{main:poly}. 
\subsection{Equivariant Taylor approximations are polynomial}   \

 In \Cref{thm:tdTproperties}, we show an equivariant analog of \cite[Proposition 6.3]{WeissCalc}\footnote{The original proof in \cite{WeissCalc} had a gap which was  completed in the erratum \cite{WeissErr}. }, which implies that the  $\upalpha$-th reduced  Taylor approximation, which we will introduced next, is  $\upalpha$-polynomial. 
 \begin{notn} \label{notn:Talpha} For any $\upalpha \in \Rep_\Gmr$, define the {\bf $\upalpha$-th reduced Taylor approximation} functor $\td{\bf T}_{\upalpha}: \Ecal_{\Gmr, {\bf 0}}\longrightarrow \Ecal_{\Gmr,{\bf 0}}$ by setting 
 \begin{equation} \label{eqdef:Ttd}
  \td{\bf T}_{\upalpha} \Emr := \hocolim \uptau_{\upalpha}^{(n)}\Emr 
  \end{equation}
 for $\Emr \in \Ecal_{\Gmr, {\bf 0}}$. By construction we have a natural  $1$-morphism $\upeta^{\Emr}_{\upalpha}:\Emr \longrightarrow \td{\bf T}_{\upalpha} \Emr$ in $ \Func(\Ocal^{\mr{op}}_\Gmr, \Cat)$ for every $\Emr \in \Ecal_{\Gmr, {\bf 0}}$, i.e.,  
 a natural transformation
\[ 
\begin{tikzcd}
\upeta_{\upalpha}: {\bf 1}_{\Ecal_{\Gmr, {\bf 0}}} \rar & \td{\bf T}_{\upalpha} 
\end{tikzcd}
\]
 for every $\upalpha \in \Rep_\Gmr$.  
\end{notn}  
\begin{rmk} Note that $\td{\bf T}_{\upalpha \oplus \upepsilon}\Emr$ is nothing but ${\bf T}_{(\upalpha, \upepsilon)}\Emr$, the $\upalpha$-th Taylor approximation in the direction of $\upepsilon$, as in \Cref{defn:Taylor}.
\end{rmk}
\begin{thm} \label{thm:tdTproperties}  Suppose $\upalpha \in \Rep_\Gmr$ and  $\Emr \in \Ecal_{0}$. 
\begin{enumerate}
\item If $\Emr \longrightarrow \uptau_{\upalpha} \Emr$ is a pointwise equivalence then so is 
 $\Emr \longrightarrow \td{\bf T}_{\upalpha}\Emr$.
\item $\eta^{\td{\bf T}_{\upalpha}\Emr}_{\upalpha}: \td{\bf T}_{\upalpha}\Emr \longrightarrow  \uptau_{\upalpha}\td{\bf T}_{\upalpha}\Emr$   is a pointwise equivalence.
\item $\upeta_{\upalpha}^{\td{\bf T}_{\upalpha}\Emr }:   \td{\bf T}_{\upalpha}  \Emr \longrightarrow \td{\bf T}_{\upalpha} \td{\bf T}_{\upalpha}\Emr $ is a poitwise equivalence. 
\end{enumerate} 
\end{thm}
Consequently:
\begin{cor} \label{cor:Taylorpoly}For any $\Emr \in \Ecal_{\Gmr, {\bf 0}}$, ${\bf T}_{(\upalpha, \upepsilon)}\Emr$ is an $\upalpha$-polynomial in the $\upepsilon$ direction.  
\end{cor}
\begin{rmk}
Following \cite{WeissErr}, we  note that 
\[ 
(\uptau_{\upalpha}^{(n)})_{\Gmr/\Hmr}(\Wmr) \cong \underset{(\Umr_1, \dots, \Umr_n) \in \Vscr^{\circ}_{\Hmr}(\upiota_{\Hmr}(\upalpha))^{\times n}} \holim \wh{\Emr}_{\Gmr/\Hmr}(\Wmr \oplus \Umr_1 \oplus \dots \oplus \Umr_n), 
\]
and the homotopy colimits of 
\[
\begin{tikzcd}
\Emr \ar[rr,"\eta_{\upalpha}^\Emr"] && \uptau_{\upalpha}\Emr \ar[rr,"\uptau_{\upalpha}(\eta_{\upalpha}^\Emr)"]   && \uptau_{\upalpha}^{(2)}\Emr
\ar[rr,"\uptau^{(2)}_{\upalpha}(\eta_{\upalpha}^\Emr)"]     && \uptau_{\upalpha}( \uptau_{\upalpha}^{(2)}\Emr)
\ar[rr,"\uptau_{\upalpha}^{(3)}(\eta_{\upalpha}^\Emr)"]   &&  \dots 
\end{tikzcd}
\]
and 
\[
\begin{tikzcd}
\Emr \ar[rr,"\eta_{\upalpha}^\Emr"] && \uptau_{\upalpha}\Emr \ar[rr,"\eta_{\upalpha}^{\uptau_{\upalpha}\Emr}"]   && \uptau_{\upalpha}(\uptau_{\upalpha}\Emr)
\ar[rr,"\eta_{\upalpha}^{\uptau_{\upalpha}^{(2)}\Emr}"]   &&  \uptau_{\upalpha}(\uptau^{(2)}_{\upalpha}\Emr)
\ar[rr,"\eta_{\upalpha}^{\uptau_{\upalpha}^{(3)}\Emr}"]  && \dots 
\end{tikzcd}
\]
are isomorphic pointwise, therefore, either sequence can be used in  \eqref{eqdef:Ttd} where we define  $\td{\bf T}_{\upalpha}\Emr$. 
\end{rmk}
\begin{notn}[Equivariant connectivity] For a $\Gmr$-space $\Xmr$, its equivariant connectivity is a function 
\[ 
\begin{tikzcd}
\nu : \{ \Hmr: \Hmr \text{ is a subgroup of } \Gmr  \} \rar & \NN
\end{tikzcd}
 \]
 which assigns a subgroup $\Hmr$  the connectivity of $\Xmr^{\Hmr}$,  the $\Hmr$-fixed points of $\Xmr$.  
\end{notn}

\begin{proof}[{\bf Proof of \Cref{thm:tdTproperties}}] The result {\it (1)} follows from the fact that homotopy limits and homotopy colimits preserve weak equivalences.  

Notice that {\it (2)} is equivalent to showing  that the vertical arrows in the diagram 
\[ 
\begin{tikzcd}
\Emr \ar[rr,"\eta_{\upalpha}^\Emr"] \dar["\eta_{\upalpha}^\Emr"'] && \uptau_{\upalpha}\Emr \ar[rr,"\eta_{\upalpha}^{\uptau_{\upalpha}\Emr}"]  \dar["\eta_{\upalpha}^{\uptau_{\upalpha}\Emr}"']   && \uptau_{\upalpha}(\uptau_{\upalpha}\Emr)
\ar[rr,"\eta_{\upalpha}^{\uptau_{\upalpha}^{(2)}\Emr}"]   \dar["\eta_{\upalpha}^{\uptau_{\upalpha}^{(2)}\Emr}"']  &&  \uptau_{\upalpha}(\uptau^{(2)}_{\upalpha}\Emr)
\ar[rr,"\eta_{\upalpha}^{\uptau_{\upalpha}^{(3)}\Emr}"]  \dar["\eta_{\upalpha}^{\uptau_{\upalpha}^{(3)}\Emr}"']  && \dots \\
 \uptau_{\upalpha}\Emr \ar[rr,"\uptau_{\upalpha}(\eta_{\upalpha}^\Emr)"']   && \uptau_{\upalpha}^{(2)}\Emr
\ar[rr,"\uptau^{(2)}_{\upalpha}(\eta_{\upalpha}^\Emr)"']     && \uptau_{\upalpha}( \uptau_{\upalpha}^{(2)}\Emr)
\ar[rr,"\uptau_{\upalpha}^{(3)}(\eta_{\upalpha}^\Emr)"']   && \uptau_{\upalpha}( \uptau_{\upalpha}^{(3)}\Emr)
\ar[rr,"\uptau_{\upalpha}^{(4)}(\eta_{\upalpha}^\Emr)"']   &&  \dots 
\end{tikzcd}
 \]
induces a weak equivalence on the homotopy colimits of horizontal lines, i.e.,  
\begin{eqnarray*}
(\td{\bf T}_{\upalpha}\Emr)_{\Gmr/\Hmr}(\Umr) &:=&  \underset{n \in \NN} {\bf hocolim} \  (\uptau^{(n+1)}_{\upalpha} \Emr)_{\Gmr/\Hmr}(\Wmr) \\ 
&=&  \underset{n \in \NN} {\bf hocolim} \  \underset{\Vmr \in \Vscr^{\circ}_{\Hmr}(\upiota_{\Hmr}(\upalpha))}{\bf holim}  (\uptau_{\upalpha}^{(n)}\Emr)_{\Gmr/\Hmr}(\Umr \oplus \Vmr)   \\
   &\overset{\simeq} \longrightarrow &    \underset{\Vmr \in \Vscr^{\circ}_{\Hmr}(\upiota_{\Hmr}(\upalpha))}{\bf holim} \ \underset{n \in \NN} {\bf hocolim} \  ( \uptau_{\upalpha}^{(n)}\Emr)_{\Gmr/\Hmr}(\Umr \oplus \Vmr) \\
   &=& (\uptau_{\upalpha}  \td{\bf T}_{\upalpha}\Emr)_{\Gmr/\Hmr}(\Vmr)
\end{eqnarray*}
for all $\Umr \in \Rep_\Hmr$.

The result follows from \Cref{lem:tauconnect} using an argument identical to that in \cite[Theorem 6.3 (1)]{WeissErr} (also see \cite[Theorem 4.2.9]{Y24}) once we establish that the morphism
\[ 
\begin{tikzcd}
{\sf p}_{\upiota_\Hmr(\upalpha)}:  {\bf S}\upgamma_{\upiota_\Hmr(\upalpha)}(\Umr, -)_+ \rar &  \Mor_{\bf 0}(\Umr, -)
 \end{tikzcd}
  \] 
  in $\Ecal_{\Hmr, {\bf 0}}$
satisfies the hypothesis of \Cref{lem:tauconnect}. To this end, we observe from \Cref{rmk:Kfix} that ${\sf p}_{\upiota_\Hmr(\upalpha)}$ at $\Vmr$ is 
$\nu$-connected, where 
\[ \nu(\Kmr) = \dim_\RR \upalpha^{\Kmr} (\dim_\RR \Vmr^{\Kmr} - \dim_\RR\Umr^\Kmr)  + \dim_\RR (\upalpha^{\perp})^{\Kmr}( \dim_\RR (\Vmr^{\perp})^{\Kmr}  -  \dim_\RR(\Umr^\Kmr)^{\perp})
\]
for  any $\Kmr \subset \Hmr$. 

The result {\it (3)}  follows immediately from {\it (1)} and {\it (2)}. 
\end{proof}
\begin{rmk} \label{rmk:Kfix} Note that for $\Umr, \Vmr \in \Rep_\Hmr$, there is a natural splitting
\[ 
\Mor_{\Hmr}(\Umr, \Vmr)^{\Kmr} = \Mor_{\Hmr}(\Umr^\Kmr, \Vmr^\Kmr) \times \Mor_\Hmr((\Umr^{\Kmr})^{\perp}, (\Vmr^{\Kmr})^\perp)
\]
for every subgroup $\Kmr \subset \Hmr$. This is because  $\Kmr$-fixed points of $\Mor_{\Hmr}(\Umr, \Vmr)$ consists of $\Kmr$-equivariant linear maps whichpreserve the subspace of  $\Kmr$-fixed points as well as its orthogonal complement.  Similarly, one can argue that 
\[ 
\upgamma_{\Hmr, \upiota_{\Hmr}(\upalpha)}(\Umr, \Vmr)^{\Kmr} = \upgamma_{\Hmr, \upiota_{\Hmr}(\upalpha)}(\Umr^\Kmr, \Vmr^\Kmr) \times \upgamma_{\Hmr, \upiota_{\Hmr}(\upalpha)}((\Umr^{\Kmr})^{\perp}, (\Vmr^{\Kmr})^\perp)
\]
for all $\Vmr, \Wmr \in \Rep_\Hmr$. 
\end{rmk}
\begin{lem} \label{lem:tauconnect} For a given  $\Gmr$-representation $\upalpha$,  if  $p: \Fmr \longrightarrow \Dmr$ in $\Ecal_{\Gmr, {\bf 0}}$ such that 
\[
\begin{tikzcd}
{\sf p}_{\Gmr/\Hmr}: \Fmr_{\Gmr/\Hmr}(\Umr) \rar & \Dmr_{\Gmr/\Hmr}(\Umr) 
\end{tikzcd}
 \]
 is $\nu$-connected, where 
 \[ 
\begin{tikzcd}
\nu : \{ \Kmr: \Kmr \text{ is a subgroup of } \Hmr  \} \rar & \NN
\end{tikzcd}
 \]
denote the function that sends 
\[ 
\Kmr \mapsto  \underset{\Qmr \subset \Kmr}{\sf min}\left\lbrace |\upalpha^{\Qmr}| |\Umr^{\Qmr}|  + |(\upalpha^{\Qmr})^{\perp} ||(\Umr^{\Qmr})^{\perp}|  - {\sf b}_{\Qmr}\right\rbrace
\]
for some constant ${\sf b}_\Qmr$, then 
 \[
\begin{tikzcd}
(\uptau_{\upalpha}{\sf p})_{\Gmr/\Hmr}: \uptau_{\upalpha}\Fmr_{\Gmr/\Hmr}(\Umr) \rar & \uptau_{\upalpha}\Dmr_{\Gmr/\Hmr}(\Umr) 
\end{tikzcd}
 \]
 is $(\nu + 1)$-connected. 
\end{lem}
\begin{proof} This is essentially the equivariant extension of  \cite[e.3]{WeissErr}. Following the argument \cite[e.3]{WeissErr}, we observe that $\Hmr$-equivariant connectivity of \[ (\uptau_{\upalpha}{\sf p})_{\Gmr/\Hmr}: (\uptau_{\upalpha}\Fmr)_{\Gmr/\Hmr}(\Umr) \longrightarrow (\uptau_{\upalpha}\Dmr)_{\Gmr/\Hmr}(\Umr)\]
at $\Kmr \subset \Hmr$ is given by minimum of 
\begin{eqnarray*}
&& \dim_\RR( \upalpha^{\Qmr} )\dim_\RR \left( \Umr^{\Qmr} \oplus \Lmr({\bf k})^\Qmr \right)    \\
&+&  \dim_\RR (\upalpha^{\Qmr})^{\perp} \dim_\RR \left( (\Umr^{\Qmr})^{\perp} \oplus (\Lmr({\bf k})^{\Qmr})^{\perp} \right)  \\
&-& (\dim \Cmr(\uplambda, \upiota_\Hmr(\upalpha))^\Qmr + k),
 \end{eqnarray*}
 where 
 \begin{itemize}
 \item  $\Qmr$ varies over subgroups of $\Kmr$, 
 \item $k \in \NN$, 
 \item ${\bf k}$ is the ordinal  $0 < 1 < \dots < k$, 
 \item $\uplambda: {\bf k} \longrightarrow {\bf n}$ is a functor  where $n = \dim_\RR \upalpha$, 
 \item $\Cmr(\uplambda, \upiota_\Hmr(\upalpha)):= \{ \Lmr: {\bf k} \to \Vscr^{\circ}_{\upiota_\Hmr(\upalpha)}: \dim_\RR \Lmr(i) = \uplambda(i) \text{ for all $i$} \} $, 
 \item $\Lmr \in \Cmr(\uplambda)$, 
 \end{itemize}
 given our assumption. Using \Cref{rmk:Kfix}, we observe that 
 \[ 
 \Cmr(\uplambda, \upiota_\Hmr(\upalpha))^{\Qmr} =  \Cmr(\uplambda, \upiota_\Hmr(\upalpha)^{\Qmr}) \times  \Cmr(\uplambda, (\upiota_\Hmr(\upalpha)^\Qmr)^{\perp})
 \]
and consequently, an argument identical to the one in \cite{WeissErr} (estimating  $\dim_\RR\Cmr(\uplambda)$ in \cite[e.3]{WeissErr}) shows that 
\[ \dim_\RR \Cmr(\uplambda, \upiota_\Hmr(\upalpha))^{\Qmr} < \dim_\RR \upalpha^{\Qmr} \dim_\RR \Lmr(k)^{\Qmr} + \dim_\RR(\upalpha^\Qmr)^{\perp} \dim_\RR(\Lmr(k)^\Qmr)^{\perp} - k ,\]
and hence the result. 
\end{proof}

\bigskip
\begin{proof}[{\bf Proof of \Cref{main:poly}}] If $ \begin{tikzcd}
 \Emr: \Jfrak_{\Gmr} \rar & \Tfrak_{\Gmr}, 
 \end{tikzcd}
$ is $\alpha$-polynomial  in the direction of $\upepsilon$ then $\Emr \longrightarrow \uptau_{\upalpha \oplus \upepsilon} \Emr$ is a pointwise equivalence by \Cref{defn:polynomial}. Thus, by \Cref{thm:taylorapprox}, $\Emr$ is $\alpha \oplus \alpha'$-polynomial  in the direction of $\upepsilon$ , i.e., 
\[ \Emr \longrightarrow \uptau_{\upalpha \oplus \upalpha' \oplus \upepsilon} \Emr \]
is a pointwise equivalence, if $\upalpha'$  is isomorphic to a finite sum of one dimensional $\Gmr$-representation. Then the result follows from \Cref{thm:tdTproperties}.
\end{proof}

\bigskip
The following result is crucial in defining homogeneous layers in the next section. 
\begin{lem} \label{lem:commuteT} For any pair $\upalpha_1, \upalpha_2 \in \Rep_{\Gmr}$ 
\[ 
\td{\bf T}_{\upalpha_1} \td{\bf T}_{\upalpha_2} \Emr \cong \td{\bf T}_{\upalpha_2} \td{\bf T}_{\upalpha_1} \Emr
\]
for all  $\Emr \in \Ecal_{\Gmr, {\bf 0}}$. 
\end{lem}  
\begin{proof} The result follows immmediately from the fact that, for all $\Wmr \in \Rep_{\Hmr}$ and $\Hmr \subset \Gmr$
\[  
 (\uptau_{\upalpha \oplus \upepsilon_1} \uptau_{\upalpha \oplus \upepsilon_2} \Emr)_{\Gmr/\Hmr}(\Wmr) := 
\underset{\Umr_1 \in \mathscr{V}^{\circ}_\Hmr(\upiota_\Hmr(\upalpha_1))   }\holim  \  \underset{\Umr_2 \in \mathscr{V}^{\circ}_\Hmr(\upiota_\Hmr(\upalpha_2))  }\holim \wh{\Emr}_{\Gmr/\Hmr}(\Wmr \oplus \Umr_1 \oplus \Umr_2)  \]
is homoemorphic to  $(\uptau_{\upalpha \oplus \upepsilon_2} \uptau_{\upalpha \oplus \upepsilon_1} \Emr)_{\Gmr/\Hmr}(\Wmr) $,  
 as homotopy limits commute with homotopy limits. 
\end{proof}
\begin{cor} \label{cor:commuteT} Given $\upalpha, \upepsilon_1, \upepsilon_2 \in \Rep_{\Gmr}$ such that $\dim_{\RR}(\upepsilon_1) = 1 =  \dim_{\RR}(\upepsilon_2)$, then  
\[ 
{\bf T}_{(\upalpha, \upepsilon_1)} {\bf T}_{(\upalpha, \upepsilon_2)} \Emr \cong {\bf T}_{(\upalpha, \upepsilon_2)} {\bf T}_{(\upalpha, \upepsilon_1)} \Emr
\]
for all  $\Emr \in \Ecal_{\Gmr, {\bf 0}}$. 
\end{cor}

\section{Equivariant layers and homogeneous 
functors} \label{sec:homo}

In order to describe the layers of equivariant Weiss functors, we first need to generalize the notion of direction beyond one dimensional representations.  

\begin{defn} \label{defn:dir} A {\bf direction vector} for $\Gmr$ is an unordered tuple
\[
\vec{v} = \{ \upepsilon_1, \upepsilon_2, \cdots, \upepsilon_n \}
\]
 of one-dimensional representation of $\Gmr$ which are pairwise non-isomorphic. 
We call that $\vec{v}$  the {\bf full direction vector} if every isomorphism class of $1$-dimensional $\Gmr$-representation appears in $\vec{v}$.
\end{defn}

\begin{notn}
Note that the collection of direction vectors admits a partial order under inclusion, hence form a category in which the full direction vector is the terminal object. It is also convenient to include an initial object -- the empty tuple, which we denote by $\vec{0}$.
\end{notn}
\begin{notn} For $\Emr\in \Ecal_{\Gmr,0}$ and a direction vector $\vec{v} = (\upepsilon_1, \upepsilon_2, \cdots, \upepsilon_n)$, let 
${\bf T}_{(\upalpha, \vec{v})}$ denote 
\[
{\bf T}_{(\upalpha, \upepsilon_1)}{\bf T}_{(\upalpha, \upepsilon_2)}\cdots {\bf T}_{(\upalpha, \upepsilon_n)}\Emr
\]
which is well defined upto an isomorphism because of \Cref{cor:commuteT}. When $\vec{v}$ is the  $\vec{0}$ vector, we declare ${\bf T}_{(\upalpha, \vec{v})} := \td{\bf T}_{\upalpha}\Emr$. 
\end{notn}
\begin{notn} \label{notn:universe} Let $\Ucal^{(n)}_{\Gmr}$ denote sub-universe of a complete $\Gmr$-universe spanned by $\Gmr$-representations of dimension less than or equal to $n$. 
\end{notn}
\begin{rmk}
When  $\alpha = {\bf 0}$ and let $\vec{v} = (\upepsilon_1, \upepsilon_2, \cdots, \upepsilon_n)$ be a direction vector for $\Gmr$. Then
\[
({\bf T}_{({\bf 0}, \vec{v})}\Emr)_{\Gmr/\Gmr}(\Wmr) = \Emr_{\Gmr/\Gmr}(\Wmr \oplus \upepsilon_1^{\oplus \infty} \oplus \cdots \oplus \upepsilon_n^{\oplus \infty})
\]
  is not a constant functor 
unless $\Ucal_{\Gmr}^{(1)}$ is the complete universe. 
\end{rmk}

\subsection{Layers of equivariant Weiss functors} \ 

\begin{defn} \label{defn:layer}For $\Gmr$-representation $\upalpha$, we define the {\bf $\upalpha$-th  layer} of  $\Emr \in \Ecal_{\Gmr, {\bf 0}}$ as the functor 
\[ 
\begin{tikzcd}
{\bf L}_{\upalpha}\Emr: \Jfrak_{\Gmr} \rar & \Tfrak_\Gmr
\end{tikzcd}
\]
given by the formula
\[ 
({\bf L}_{\upalpha}\Emr)_{\Gmr/\Hmr}(\Wmr) := {\bf hofiber}\left( ({\bf T}_{(\upalpha, \vec{v})}\Emr)_{\Gmr/\Hmr}(\Wmr) \to (\td{\bf T}_{\upalpha}\Emr)_{\Gmr/\Hmr}(\Wmr) \right),
\]
where $\vec{v}$ is the full direction vector. 
\end{defn}
Suppose $\upepsilon_1$ and $\upepsilon_2$  are nonisomorphic  $1$-dimensional $\Gmr$-representation such that their restriction to a subgroup  $\Hmr \subset \Gmr$ are isomorphic  $\upiota_\Hmr\upepsilon_1 \cong \upiota_\Hmr\upepsilon_2 \cong \upepsilon.$ Then, by \Cref{thm:tdTproperties}, there is a $2$-morphism
\[ 
{\bf T}_{\upiota_\Hmr (\upalpha) \oplus \upepsilon }\upiota_\Hmr\Emr \longrightarrow {\bf T}_{\upiota_\Hmr (\upalpha \oplus \upepsilon_2) }{\bf T}_{\upiota_\Hmr (\upalpha \oplus \upepsilon_1) }\upiota_\Hmr\Emr
\]
which is pointwise equivalence. However, not all $1$-dimensional $\Hmr$ representations may be a restriction of $\Gmr$-representations, and therefore  for any subgroup $\Hmr \subset \Gmr$,  there is a natural map in $\Ecal_{\Hmr, {\bf 0}}$
\begin{equation} \label{map:ell}
\begin{tikzcd}
\ell_{\Hmr, \upalpha}: \upiota_\Hmr{\bf L}_{\upalpha} \Emr \rar &  {\bf L}_{\upiota_\Hmr (\upalpha)} \upiota_\Hmr \Emr 
\end{tikzcd}
 \end{equation}
for any $\Emr \in \Ecal_{\Gmr, {\bf 0}}$ which may not be a pointwise  equivalence in general. However,  it is an equivalence when  $\upiota_\Hmr(\Ucal_\Gmr^{(1:\upalpha)}) = \Ucal_{\Hmr}^{(1:\upiota_\Hmr(\upalpha))}$ (compare \Cref{lem:resderinf}). 

\begin{defn} \label{defn:vpoly}
Let $\vec{v}$ be a direction vector for $\Gmr$. We set 
\[ (\vec{v}: \upalpha) =\{\upalpha \otimes \upepsilon_1 \otimes \dots \otimes \upepsilon_k: \upepsilon_k \in \vec{v} \} \]
and say that $\Emr\in \Ecal_0$ is {\bf $(\upalpha, \vec{v})$-polynomial} if 
\[
\Emr \longrightarrow {\bf T}_{(\upalpha, \vec{v})}\Emr
\]
is a pointwise equivalence restricted to the universe $\Ucal^{(\vec{v}: \upalpha)}_\Gmr$. We say that $\Emr$ is {\bf fully $\upalpha$-polynomial}, if it is $(\upalpha, \vec{v})$-polynomial for the full direction vector $\vec{v}$. \end{defn}
\begin{rmk} We notice  that $\Emr$ is  strongly $\upalpha$-polynomial, if it is $(\upalpha, \vec{v})$-polynomial for the empty direction vector $\vec{v}$.
\end{rmk}

\begin{defn} \label{defn:vhomo} We say a $1$-morphism 
\[ \Emr: \Jfrak_\Gmr \longrightarrow \Tfrak_{\Gmr} \] is  {\bf $(\upalpha, \vec{v})$-homogeneous} if 
\begin{enumerate}
\item $ {\td {\bf T}}_{\upalpha}\Emr$ is pointwise contractible, and 
\item   $\Emr$ is a fully $\upalpha$-polynomial.
\end{enumerate}
We simply say $\Emr$ is {\bf $\upalpha$-homogeneous} when  $\vec{v}$ is the full vector. 
\end{defn}
\begin{rmk}
Note that restriction of an $\upalpha$-homogeneous functor to $\Hmr$ may not be  an $\upiota_{\Hmr}(\upalpha)$-homogeneous  unless  $\upiota_\Hmr(\Ucal_\Gmr^{(1:\upalpha)}) = \Ucal_{\Hmr}^{(1:\upiota_\Hmr(\upalpha))}$. 
\end{rmk}
It is straightforward to check that ${\bf L}_{\upalpha}\Emr$ is  $\upalpha$-homogenous  for all $\Gmr$-representation $\upalpha$ and for all $\Emr \in \Ecal_{\Gmr, {\bf 0}}$. In \Cref{thm:keyex} we will show that  an $\upalpha$-system of $\Gmr$-spectra (as in \Cref{defn:alphasystem}) leads $\upalpha$-homogeneous functors which is crucial to \Cref{main3}.

The next set of results in this paper is sensitive to our choice of universe, and therefore we establish the following notations. 
\begin{notn} Given a set of $\Gmr$-representations $\sfR$, we may define a functor  
\[ 
\begin{tikzcd}
\Jfrak_{\sfR}: \Ocal^{\mr{op}}_{\Gmr} \rar & \Cat 
\end{tikzcd}
\]
which sends $\Gmr/\Hmr$ to the full sub category of $\Rep_\Hmr$ generated by sub-representations of $\upiota_{\Hmr}(\Ucal_{\Gmr}^{(\Rmr)})$. Given any $\Gmr$-representation $\upalpha$ within $\Ucal^{(\sfR)}_\Gmr$,   one may construct  analogous presheaf $\Jfrak_{\Rmr}^{\upalpha}$ by declaring the space of morphisms as Thom spaces of the $\Gmr$-equivariant bundles defined in  \eqref{bundlealpha}. We let $\Ecal_{\sfR, \upalpha}$ denote the category of $1$-morphisms from $\Jfrak_{\Rmr}^{\upalpha} \longrightarrow \Tfrak_\Gmr$, and define 
\begin{itemize}
\item $(- )^{(\upalpha)}$ -- the derivative functor,
\item ${\bf T}_{\upalpha}(-)$ -- the $\upalpha$-th Taylor approximation functor, 
\item ${\bf L}_{\upalpha}(-)$ -- the $\upalpha$-th layer  functor, 
\item ${\bf R}^{\upalpha}_{\bf 0}(-)$ -- the restriction functor 
\item ${\bf I}^{\upalpha}_{\bf 0}(-)$ -- the induction functor 
\end{itemize}
internal to the universe $\Ucal^{(\sfR)}_{\Gmr}$.  The above constructions behave well under restrictions of universes because the inclusion of  $\Jfrak_{\Rmr}^{\upalpha} \hookrightarrow \Jfrak_{\Rmr'}^{\upalpha}$ is full and faithful  for all $\sfR \subset \sfR'$.  Therefore, we use the same notations regardless of the underlying universe. 
\end{notn}

\subsection{Equivariant stable symmetric Weiss functors}

\begin{defn} Let $\upalpha, \upepsilon \in \Rep_\Gmr$ such that $|\upepsilon| = 1$. We call $\Fmr \in \Ecal_{\Gmr, \upalpha}$  a {\bf stable object} if the map 
\begin{equation} \label{mapadjoint}
\begin{tikzcd}
 \upnu_n:\Fmr_{\Gmr/\Hmr}(\Umr) \rar & \Omega^{\upiota_\Hmr(\upalpha \otimes \upepsilon)}\Fmr_{\Gmr/\Hmr}(\Umr \oplus \upiota_\Hmr(\upepsilon))
\end{tikzcd}
\end{equation}
adjoint to  \eqref{map:UV} is a weak equivalence for all $ \Umr  \underset{\mr{finite}}\subset \Ucal_{\Hmr}$. 
\end{defn}
\begin{defn}We  will call $\Fmr \in \Ecal_{\sfR, \upalpha}$ a {\bf symmetric object} if $\Omr(\upalpha)$ acts on $\Fmr$. 
\end{defn}
\begin{notn} \label{notn:symE} Let $\Ecal_{\Gmr, \upalpha}^{\Sigma}$ denote the full sub-category of $\Ecal_{\Gmr, \upalpha}$ consisting of stable symmetric objects. 
\end{notn}
In \Cref{rmk:DerStabSym}, we show that the $\upalpha$-th derivative $\Emr^{(\upalpha)}$ for any $\Emr \in \Ecal_{\Gmr, {\bf 0}}$ is an example of a symmetric object. By \Cref{cor:sc1cor}, $\Emr^{(\upalpha)}$ is stable precisely when  $(\upalpha \oplus \upepsilon)$-th derivative of $\Emr$ vanishes for all $1$-dimensional $\Gmr$-representation to $\upepsilon$. A straightforward consequence of these observations is the following lemma. 
\begin{lem} \label{lem:stable} For any $\Emr \in \Ecal_{\Gmr, {\bf 0}}$, the $\upalpha$-th derivative of its $\upalpha$-th Taylor approximation $({\bf T}_{\upalpha}\Emr)^{(\upalpha)}$ is stable symmetric in $\Ecal_{\Gmr, \upalpha}$. 
\end{lem}

\begin{notn} \label{notn:pound} Associated to $\Xmr \in \Sp^{\Gmr}_{(\sfR)}$, we construct a spectra $\Xmr^{\#}$ whose $\Vmr$-th space is 
\[ \Xmr^{\#}_\Vmr  := \Omega^{\infty}_\sfR(\Smr^{\Vmr} \sma \Xmr)  \]
for any $\Vmr \in \Ucal_{\Gmr}^{(\sfR)}.$ Note that there is a natural inclusion map 
\begin{equation} \label{incVsp}
\begin{tikzcd}
\Xmr_{\Vmr} \rar[hook] & \Omega^{\Vmr}(\Smr^{\Vmr} \sma \Xmr) \rar[hook] &  \Omega^{\infty}_\sfR(\Smr^{\Vmr} \sma \Xmr) = \Xmr^{\#}_\Vmr
\end{tikzcd}
\end{equation}
is a weak equivalence. Therefore, $\Xmr$ and $\Xmr^{\#} $ are equivalent spectrum.
\end{notn}
\begin{rmk} \label{rmk:Oequiv}
 When we apply the construction of \Cref{notn:pound} to ${\bf \Theta}(\Emr^{(\upalpha)})(\Umr)$ (c.f. \Cref{defn:der-at-infty}), the $\Wmr$-th term of the resultant spectrum 
\[
 \Omega^{\infty}_\sfR(\Smr^{\Wmr} \sma {\bf \Theta}(\Emr^{(\upalpha)})(\Umr)) := \underset{\Vmr' \in \Ucal_{\Gmr}^{(1)}}\hocolim \ \Omega^{\upalpha \otimes \Vmr'}_{\sfR}(\Smr^{\Wmr} \sma {\bf \Theta}(\Emr^{(\upalpha)})(\Umr))_{\Vmr'})
\]
also admits an action of ${\sf O}(\upalpha)$ (see  \Cref{rmk:DerStabSym}) and the inclusion map \eqref{incVsp} preserves the ${\sf O}(\upalpha)$-action. 
\end{rmk}

\begin{thm}[{\bf Classification of stable symmetric objects}] \label{thm:stablesymmetric} Let ${\bf Sys}^{\upalpha}$ denote the category of $\upalpha$-systems in $\Ucal_{\Gmr}$. Then there is  an equivalence of  categories between  $\Ecal_{\Gmr, \upalpha}^\Sigma$ and systems in ${\bf Sys}^{\upalpha}$   with  n\"aive  ${\sf O}(\upalpha)$-actions.
\end{thm}

\begin{proof} Let $\wh{{\bf Sys}}^{\upalpha}$ denote the full sub-category of ${\bf Sys}^{\upalpha}$ generated by the ones  with n\"aive action of $\Omr(\upalpha)$. Let 
\begin{equation} \label{thetapound}
\begin{tikzcd}
{\bf \Theta}^\#: \Ecal^{\Sigma}_{\Gmr, \upalpha} \rar & \wh{{\bf Sys}}^{\upalpha}, 
\end{tikzcd}
\end{equation}
be the functor that assigns  $\Fmr  \in \Ecal_{\Gmr, \upalpha}^\Sigma$,  to its system $\{ {\bf \Theta} \Fmr(\Umr)
^\#: \Umr \in \Ucal_\Gmr \}$ (see \eqref{eqn:system}). Now we define the inverse 
 \begin{equation} \label{eqn:phi}
 \begin{tikzcd}
 {\bf \Phi} : \wh{{\bf Sys}}^{\upalpha} \rar &  \Ecal^{\Sigma}_{\Gmr, \upalpha}  
 \end{tikzcd}
 \end{equation}
by setting 
\[ 
{\bf \Phi}(\upchi)_{\Gmr/\Hmr}(\Umr) =  \Omega^{\infty}_{\upiota_\Hmr(1:\upalpha)} \left(\Smr^{\upiota_\Hmr( \upalpha) \otimes \Umr_{\upalpha}}  \sma \upchi(\Umr_{\upalpha}^{\perp})^\# \right), 
 \]
 where $\Umr_\upalpha = \Umr \cap \Ucal_{\Hmr}^{\upiota_{\Hmr}(1:\upalpha)}$, for all $ \Umr \subset \Ucal_{\Hmr}$. We let $\Omr(\upalpha)$ act diagonally utilizing its  action on both $\upalpha$ and ${\bf \Theta}$. 

Note that 
\[ ( {\bf \Theta}^\# \circ {\bf \Phi}) (\upchi)(\Umr)  = {\upchi}(\Umr)^{\#\#},  \] 
and therefore $\upchi \simeq ( {\bf \Theta}^\# \circ {\bf \Phi}) (\upchi)$ via an $\Omr(\upalpha)$-equivariant map (see \Cref{rmk:Oequiv}).  On the other hand 
\begin{eqnarray*}
( {\bf \Theta}^\# \circ {\bf \Phi})(\Fmr)(\Umr) &=& \Omega^{\infty}_{(1:\upalpha)} \left(\Smr^{ \upalpha \otimes \Umr_\upalpha} \sma {\bf \Theta} \Fmr(\Umr_{\upalpha}^{\perp}) \right)  \\
 &=&  \underset{\substack{ \Vmr' \in \Ucal_{\Gmr}^{(1:\upalpha)} \\ \Vmr'' \in \Ucal_\Gmr^{(1)} }}\hocolim \text{ } \Omega^{ \Vmr'}(\Smr^{\upalpha\otimes \Umr_\upalpha} \sma \Omega^{\upalpha\otimes \Vmr''}(\Smr^{\Vmr'} \sma \Fmr(\Vmr'' \oplus \Umr_{\upalpha}^\perp)))
 \end{eqnarray*}
 for all $\Vmr \in \Ucal_\Gmr^{(1:\upalpha)}$. 
Now define 
\[ \widetilde{\Fmr}(\Umr) =  \underset{\substack{ \Vmr' \in \Ucal_{\Gmr}^{(1:\upalpha)} \\ \Vmr'' \in \Ucal_\Gmr^{(1)} }}\hocolim \text{ } \Omega^{ \Vmr' \oplus \upalpha\otimes \Vmr''}(\Smr^{ \Vmr'} \sma  \Fmr( \Vmr'' \oplus \Umr))\]
and observe that we have a zigzag of maps 
\[ 
\begin{tikzcd}
\Fmr \rar[ "\simeq"] & \widetilde{\Fmr} & \lar["\simeq"']  ( {\bf \Theta}^\# \circ {\bf \Phi})(\Fmr)
\end{tikzcd}
 \]
 which are pointwise equivalences. 
\end{proof}

\subsection{Classification of homogeneous Weiss functors.} \  

We first establish a few supporting results (essentially equivariant versions of the some of the results needed in \cite{WeissCalc}) before indulging with \Cref{thm:classification},  which is the main result of this subsection, and required in the proof of \Cref{main3}. 

\begin{defn} \label{defn:conn}
We say that an object $\Emr \in \Ecal_{\Gmr, {\bf 0}}$ is {\bf connected at infinity} if 
\[
\underset{\Vmr \in \Ucal_{\Gmr}^{(1)}}{\bf hocolim} \ \Emr_{\Gmr/\Hmr}( \Umr \oplus \upiota_\Hmr(\Vmr))
\]
is connected as a $\Hmr$-space for all $\Umr \in \Ucal_{\Hmr}$. 
\end{defn}

\begin{lem} \label{lem:Weiss5.10}
Let $f:  \Emr\longrightarrow \Fmr$ in $\Ecal_{\Gmr,{\bf  0}}$ be such that the homotopy fiber of $f$ 
is pointwise contractible.  If $\Emr$ and $\Fmr$ are $(\upalpha, \vec{v})$-polynomial and  $\Fmr$ is connected at infinity then $f$ is an equivalence.
\end{lem}

\begin{proof} This is primarily a basepoint issue. Let $(\Fmr_{{\sf b}})_{\Gmr/\Hmr}(\Vmr)$ denote the base point component of  $\Fmr_{\Gmr/\Hmr}(\Vmr)$ for all finite  $\Hmr$-representation $\Vmr$ and $\Hmr$ subgroup of $\Gmr$.  If $\Fmr$ is connected at infinity then
\[ 
\begin{tikzcd}
{\bf T}_{(\upalpha, \upepsilon)}\Fmr_{\sf b}  = \underset{k} {\bf hocolim} \ (\uptau_{\upalpha\oplus \upepsilon}^{(k)}\Fmr_{\sfb}) \rar & {\bf T}_{(\upalpha , \upepsilon)}\Fmr  = \underset{k} {\bf hocolim} \ (\uptau_{\upalpha\oplus \upepsilon}^{(k)}\Fmr)
\end{tikzcd}
 \] 
 is a pointwise equivalence for all subspace  $\upepsilon \subset \Ucal_{\Gmr}$ such that $|\upepsilon| =1$. Then the result follows from our assumption.
\end{proof}
\begin{lem} \label{lem:fiberpoly} Suppose $f: \Emr \longrightarrow \Fmr \in \Ecal_{\Gmr,  {\bf 0}}$ such that 
\begin{itemize}
\item  $\Emr  \overset{\simeq}{\longrightarrow} \td{\bf T}_{\upalpha}\Emr $ is an equivalence, 
\item  $\Fmr^{(\upalpha)}  \simeq \ast$,
\end{itemize}
then for the homotopy fiber $\Dmr:= {\bf hofib}(f)$ is pointwise equivalent to ${\td {\bf T}}_{\upalpha}\Dmr$. 

Consequently, if $\Emr$ is $(\upalpha, \vec{v})$-polynomial and if $\Fmr^{(\upalpha \oplus \upepsilon)}\simeq *$ for every $\upepsilon$ in $\vec{v}$, then $\Dmr$ is $(\upalpha, \vec{v})$-polynomial.
\end{lem}

\begin{proof} By equivariant Steifel combinatorics II, i.e. \Cref{thm:SC2}, we have a homotopy commutative diagram
\[ 
\begin{tikzcd}
{\bf R}^{\upalpha}_{\bf 0}\Dmr^{(\upalpha)} \rar["\simeq"] \dar &{\bf R}^{\upalpha}_{\bf 0} \Emr^{(\upalpha)} \rar \dar & {\bf R}^{\upalpha}_{\bf 0}\Fmr^{(\upalpha)} \simeq \ast  \dar \\
\Dmr \dar \rar &  \Emr \dar \rar   & \Fmr \dar  \\
\uptau_{\upalpha}\Dmr \rar & \uptau_{\upalpha}\Emr \rar& \uptau_{\upalpha}\Fmr
\end{tikzcd}
\]
in which rows and columns are fiber sequences (pointwise). Our assumptions imply that the connectivity of 
the map $\Dmr \to \uptau_{\upalpha}\Dmr$ equals  the connectivity of  $\Emr \to \uptau_{\upalpha}\Emr$, and  more generally, the connectivity of 
 $\Dmr \to \uptau_{\upalpha}^{(k)}\Dmr$ equals   that of $\Emr \to \uptau_{\upalpha}^{(k)}\Emr$ for all $k \geq 1$. Hence,  the result.

\end{proof}

\begin{lem}
Suppose $\Fmr \in \Ecal_{\Gmr, {\bf 0}}$  such that $\Fmr^{(\upalpha \oplus \upepsilon)} \simeq *$ for all $\upepsilon \in \vec{v}$ then
\[
\begin{tikzcd}
\Umr \rar &  \Omega\Fmr_{\Gmr/\Hmr}(\Umr)
\end{tikzcd}
\]
is $(\upalpha, \vec{v})$-polynomial. 
\end{lem}

\begin{proof}
Set $\Emr\simeq *$ in \Cref{lem:fiberpoly}. 
\end{proof}

\begin{lem} \label{lem:Weiss5.12}
If $\Ecal \in \Ecal_{\Gmr, {\bf 0}}$ is $(\upalpha, \vec{v})$-polynomial then 
\[
\Emr^{(\alpha)}_{\Gmr/\Hmr}(\Umr) \longrightarrow \Omega^{\upiota_\Hmr(\upalpha \otimes \upepsilon)}\Emr_{\Gmr/\Hmr}(\Umr \oplus \upiota_{\Hmr}(\upepsilon))
\]
is an equivalence for every $\upepsilon \in \vec{v}$ and $\Umr \underset{\mr{finite}}\subset \Ucal_{\Hmr}$.
\end{lem}
\begin{proof} This is a consequence of \Cref{thm:SC1} (equivariant Steifel combinatorics I), in particular \Cref{cor:sc1cor}.
\end{proof}

\begin{lem} \label{lem:Weiss5.13} Let $\upalpha$ be a finite sum of one dimensional $\Gmr$-representations. 
Suppose $f: \Emr \longrightarrow \Fmr$ in $\Ecal_{\Gmr, {\bf 0}}$ be a map between fully $\upalpha$-polynomial functors such that 
\begin{enumerate}
\item the induced map 
\[
\begin{tikzcd}
\underset{\Vmr \in \Ucal_{\Gmr}^{(1)}}{\bf hocolim} \ \Emr_{\Gmr/\Hmr}(\Umr \oplus \Vmr) \rar & \underset{\Vmr \in \Ucal_{\Gmr}^{(1)}}{\bf hocolim} \ \Fmr_{\Gmr/\Hmr}(\Umr \oplus \Vmr)
\end{tikzcd}
 \]
is an equivalence for all $\Umr \underset{\mr{finite}}\subset \Ucal_\Hmr$, 
\item the induced map on the systems
\[
\begin{tikzcd}
{\bf \Theta}(f): {\bf \Theta}\Emr^{(\upalpha')}(\Umr) \rar &  {\bf \Theta}\Fmr^{(\upalpha')}(\Umr)
\end{tikzcd}
\]
is an equivalence for every sub $\Gmr$-representation $\upalpha'$, and 
\item $\Fmr$ is connected at infinity. 
\end{enumerate}
 Then $f$ is an equivalence.
\end{lem}
\begin{proof} By \Cref{lem:Weiss5.10}, it suffices to show  that ${\bf hofib}(f)$ is pointwise contractible. Since $\Emr$ and $\Fmr$ are fully $\upalpha$-polynomial, so is  its homotopy fiber. Now assume that ${\bf hofib}(f)$ is strongly $\upalpha'$-polynomial for some \[{\bf 0} \neq \upalpha' \subset \upalpha,\] then its $\upalpha'$-th derivative must be stable (by \Cref{lem:stable}) and equivalent to 
\[
{\bf hofib}(f)_{\Gmr/\Hmr}(\Umr) \simeq  \Omega^{\infty}_{\upiota_\Hmr(1:\upalpha)} \left(\Smr^{\upiota_\Hmr( \upalpha) \otimes \Umr_{\upalpha}}  \sma \upchi(\Umr_{\upalpha}^{\perp})^\# \right)
 \] 
for all $\Umr \subset \Ucal_\Hmr$ where $\upchi$ is a $\upalpha'$-system equivalent to $\upalpha'$-th derivative  of ${\bf hofib}(f)$ at $\infty$ (by \Cref{thm:stablesymmetric}). Since sequential homotopy limit commutes with homotopy fibers 
\[ {\bf \Theta}({\bf hofib}(f))^{(\upalpha')} \simeq {\bf hofib} \left( {\bf \Theta}(f) \right)  \simeq \ast ,\]
which is a contradiction. 
\end{proof}

\begin{thm} \label{thm:htpyorbit}
Let $\upalpha$ be a finite sum of $1$-dimensional representations and  $\uptheta$ be a $\Gmr$-spectrum in $\Sp^{\Gmr}_{(\sfR)}$ with a n\"aive $\Omr(\upalpha)$-action.
Then $\Emr$ and $\Fmr$ in $\Ecal_{\sfR, {\bf 0}}$ given by 
\[
\Emr_{\Gmr/\Hmr}(\Umr):= [\Omega_{\upiota_\Hmr(\sfR)}^{\infty}(\Smr^{\upiota_\Hmr(\upalpha) \otimes \Umr} \sma \upiota_\Hmr{\uptheta})]_{\mr{h} {\sf O}(\upiota_{\Hmr}(\upalpha))}
\]
\[
\Fmr_{\Gmr/\Hmr}(\Umr):= [\Omega_{\upiota_\Hmr(\sfR)}^{\infty}(\Smr^{\upiota_\Hmr(\upalpha) \otimes \Umr} \sma \upiota_\Hmr{\uptheta})_{\mr{h} {\sf O}(\upiota_{\Hmr}(\upalpha))}]
\]
are ${\bf T}_{(\upalpha, \vec{v})}$-equivalent for every direction vector $\vec{v}$ of $\Gmr$.
\end{thm}

\begin{proof} 
Fix  $\Hmr \subset \Gmr$. We essentially follow the strategy in \cite[Example 6.4]{WeissCalc} to show that for a $\sfc$-connected $\Hmr$-spectrum  ${\bf X}$ in $\Sp^\Gmr_{\upiota_\Hmr(\sfR)}$ with 
${\sf O}(\upiota_\Hmr(\upalpha))$-action, the canonical map
\[
(\Omega^{\infty}_{\upiota_\Hmr(\sfR)} {\bf X})_{\mr{h} {\sf O}(\upiota_\Hmr(\upalpha))} \longrightarrow \Omega^{\infty}_{\upiota_\Hmr(\sfR)}({\bf X}_{\mr{h} {\sf O}(\upiota_\Hmr(\upalpha))})
\]
is $\nu$-connected for every $\nu$ so that
\begin{itemize}
 \item[-] $\nu(\Kmr)\leq 2\sfc(\Kmr)+1$ for all subgroup $\Kmr\leq \Hmr$, and 
 \item[-] $\nu(\Kmr)\leq 2\sfc(\Kmr')$ for all subgroups $\Kmr' < \Kmr \leq \Hmr$.
\end{itemize}
A similar idea as in \cite[Example 6.4]{WeissCalc} allows us to reduce it to the case where ${\bf X}$ is a suspension spectrum $\Sigma^{\infty}_{\upiota_\Hmr(\sfR)} \Ymr$. Let ${\bf Q}_{\upiota_\Hmr(\sfR)}$ denote the composite functor $\Omega^{\infty}_{\upiota_\Hmr(\sfR)} \Sigma^{\infty}_{\upiota_\Hmr(\sfR)}$. By equivariant Freudenthal theorem  \cite[Theorem 9.1]{Alaska}   the composite
\[
\begin{tikzcd}
\Ymr_{\mr{h} {\sf O}(\upiota_\Hmr(\upalpha))} \rar["f"] &  ({\bf Q}_{\upiota_\Hmr(\sfR)}\Ymr)_{\mr{h} {\sf O}(\upiota_\Hmr(\upalpha))} \rar["\upsilon"] &  {\bf Q}_{\upiota_\Hmr(\sfR)}(\Ymr_{\mr{h} {\sf O}(\upiota_\Hmr(\upalpha))})
\end{tikzcd}
\]
and the map $f$  are $\nu$-connected. Therefore, $\upsilon$ is $\nu$-connected. 

Now consider ${\bf X} = \Smr^{\upiota_\Hmr(\upalpha)\otimes \Umr} \sma \upiota_\Hmr\uptheta$ which is $\sfc$-connected,  where
\[
\sfc(\Kmr) = | \upalpha^{\Kmr}| | \Umr^{\Kmr}|  + |(\upalpha^{\Kmr})^{\perp}|  |(\Umr^{\Kmr})^{\perp}|  + {\sf b}_{\Kmr}
\]
and $\sfb_{\Kmr}$ is the connectivity of ${\uptheta}^{\Kmr}$ (independent of  $\Umr$) for any $\Kmr \subset \Hmr$.
Our previous discussion now implies that 
\[
\begin{tikzcd}
\Emr_{\Gmr/\Hmr}(\Umr) \rar & \Fmr_{\Gmr/\Hmr}(\Umr)
\end{tikzcd}
\] 
is $\nu$-connected, where $\nu(\Hmr)$ is the minimum of 
\begin{itemize}
 \item[-] $2|\upalpha^{\Hmr} ||\Umr^{\Hmr}| + 2|(\upalpha^{\Hmr})^{\perp}| | (\Umr^{\Hmr})^{\perp}|  + 2{\sf b}_{\Hmr} + 1$
 \item[-] $\underset{\Kmr \subset \Hmr}{\sf min}\left\lbrace |\upalpha^{\Kmr}| |\Umr^{\Kmr}| + |(\upalpha^{\perp})^{\Kmr}| |(\Umr^{\perp})^{\Kmr} | + {\sf b}_{\Kmr}\right\rbrace$.
\end{itemize}
Thus, by \Cref{lem:tauconnect} the map 
$ 
\uptau_{(\alpha, \upepsilon)}\Emr_{\Gmr/\Hmr}(\Umr) \longrightarrow \uptau_{(\alpha, \upepsilon)}\Fmr_{\Gmr/\Hmr}(\Umr)
$
is at least $(\nu + 1)$-connected for any  $\Gmr$-representation $\upepsilon$ such that $|\upepsilon|= 1$. Hence
\[
\begin{tikzcd}
{\bf T}_{(\alpha, \upepsilon)}\Emr \rar &  {\bf T}_{(\alpha, \epsilon)}\Fmr
\end{tikzcd}
\]
is a pointwise equivalence, and therefore \[
\begin{tikzcd}
{\bf T}_{(\alpha, \vec{v})}\Emr \rar &  {\bf T}_{(\alpha, \vec{v})}\Fmr
\end{tikzcd}
\]
is an equivalence for any direction vector $\vec{v}$.
\end{proof}

\bigskip
The following result   restricts  $\upalpha$ to a finite sum of $1$-dimensional $\Gmr$-representations because the proof is an inductive argument  on dimensions which breaks down equivariantly in the  presence  $\Gmr$-representations of a dimension two or higher  in the summand. 
\begin{defn} \label{defn:bbelow}We call an  $\upalpha$-system $\upchi$  {\bf  uniformly bounded below} by the function 
\[ 
\begin{tikzcd}
\nu: \{ \Hmr: \Hmr \text{ subgroup of } \Gmr \} \rar & \ZZ
\end{tikzcd}
\]
such that  $\upchi(\Umr)^\Hmr$ is at least $\nu(\Hmr)$-connected for all $\Umr \in \Ucal_{\Hmr}$ and $\Hmr \subset \Gmr$. 
\end{defn}
   
\begin{thm} \label{thm:keyex}
Let $\upalpha$ be a finite  sum of $1$-dimensional $\Gmr$-representations, i.e., $\upalpha \subset \Ucal_{\Gmr}^{(1)}$,
Suppose  $\upchi$ is an uniformly bounded below $\upalpha$-system  with n\"aive $\Omr(\upalpha)$-action. Then $\Fmr \in\Ecal_{\Gmr, {\bf 0}}$ given by 
\[
\Fmr_{\Gmr/\Hmr}(\Umr):= 
\Omega^{\infty}_{\upiota_\Hmr(1)}\left[ (\Smr^{ \upiota_\Hmr(\upalpha) \otimes \Umr_{\upalpha}}  \sma \upchi(\Umr_{\upalpha}^{\perp}))_{\mr{h} {\sf O}(\upiota_\Hmr(\upalpha))}\right]
\]
where  $\Umr_{\upalpha}  = \Umr \cap \Ucal_\Gmr^{\upiota_\Hmr(1)}$, is $\upalpha$-homogeneous in the sense of \Cref{defn:vhomo}. 
\end{thm}
\begin{proof} Since $\upalpha$ is a finite sum of $1$-dimensional representations, we may write \[ \upalpha \cong \upbeta \oplus \upepsilon\] for some $1$-dimensional representation.
\end{proof}

\begin{proof}
We first show that $\Fmr$ is fully $\upalpha$-polynomial. 
Since $\Fmr$ admits a delooping, by \Cref{lem:fiberpoly}, it suffices to show that $\Fmr^{(\upalpha \oplus \upepsilon)}\simeq *$ for any one-dimensional $\Gmr$-representation $\upepsilon$.

For any sub $\Gmr$-representation $\upbeta \subset \upalpha$,  we write $\upalpha = \upbeta \oplus \upbeta'$ and define a new functor
\[
\Fmr[\upbeta]_{\Gmr/\Hmr}(\Umr) := [{\bf \Phi}(\upchi)_{\mr{h} \Omr(\upbeta')}]_{\Gmr/\Hmr}(\Umr) = \Omega^{\infty}_{\upiota_\Hmr(1)}\left[\Smr^{\upiota_\Hmr(\upalpha) \otimes \Umr_{\upalpha}} \sma \upchi(\Umr_{\upalpha}^{\perp})_{\mr{h} {\sf O}(\upiota_{\Hmr}(\upbeta'))}\right]
\]
for any finite  $\Umr \subset \Ucal_{\Hmr}$, where ${\bf \Phi}$ is the functor in \eqref{eqn:phi}. It is routine to check that $\Fmr[\upbeta]\in \Ecal_{\upbeta}$. 
We then follow the proof of \cite[Theorem 7.3]{WeissCalc} equivariantly to identify the $\upbeta$-th derivative $\Fmr^{(\upbeta)}$  with $\Fmr[\upbeta]$ using an inductive argument. Thus, when  $\upbeta = \upalpha$, we will get 
\begin{equation} \label{eqn:step1}
\Fmr^{(\upalpha)} \simeq \Fmr[\alpha] = {\bf \Phi}(\upchi)
\end{equation}
is a stable symmetric object in $\Ecal_{\Gmr, \upalpha}$ by \Cref{thm:stablesymmetric}, and consequently (by \Cref{cor:sc1cor})
\[
\Fmr^{(\upalpha \oplus \upepsilon)} = (\Fmr^{(\upalpha)})^{(\upepsilon)} \simeq \Fmr[\alpha]^{(\upepsilon)} \simeq *
,
\] 
as desired.

We now execute the inductive argument claimed above. Suppose $\upepsilon \subset \upbeta'$ is a  one-dimensional sub $\Gmr$-representation. Our goal is show that the adjoint of
\[
{\bf R}_{\upbeta}^{\upbeta \oplus \upepsilon} \Fmr[\upbeta \oplus \upepsilon] \longrightarrow \Fmr[\upbeta]
,
\]
namely
\[
\Fmr[\upbeta \oplus \upepsilon] \longrightarrow {\bf I}_{\upbeta}^{\upbeta \oplus \upepsilon}\Fmr[\upbeta]  = \Fmr[\upbeta]^{(\upepsilon)}
,
\]
is an equivalence. By \Cref{cor:sc1cor}, one may identify $\Fmr[\upbeta]_{\Gmr/\Hmr}^{(\upepsilon)}(\Umr)$ as the homotopy fiber of   
\[ 
\begin{tikzcd}
\Fmr[\upbeta]_{\Gmr/\Hmr}(\Umr)  \rar &\Omega^{\upiota_{\Hmr}(\upbeta \otimes \upepsilon)}\Fmr[\upbeta]_{\Gmr/\Hmr}(\Umr \oplus \upepsilon).
\end{tikzcd}
\] 
Explicitly, $\Fmr[\upbeta]^{(\upepsilon)}_{\Gmr/\Hmr}(\Umr)$ equals $\Omega^{\infty}_{\upiota_\Hmr(1)}(-) $ applied to to the homotopy fiber of the map 
\[ 
(\Smr^{\upiota_\Hmr(\upalpha) \otimes \Umr_{\upalpha}} \sma \upchi(\Umr_{\upalpha}^{\perp}))_{\mr{h} {\sf O}(\upiota_\Hmr(\upbeta'))} \to
\Omega^{\upiota_\Hmr(\upbeta \otimes \upepsilon)}(\Smr^{\upiota_\Hmr(\upalpha) \otimes (\Umr_{\upalpha} \oplus\upiota_\Hmr( \epsilon))} \sma \upchi(\Umr_{\upalpha}^{\perp}))_{\mr{h} {\sf O}(\upiota_\Hmr(\upbeta'))}
\]
as $\upepsilon \in \Ucal_{\Gmr}^{(1)}$. Since 
\begin{eqnarray*}
\upiota_\Hmr(\upalpha) \otimes (\Umr_{\upalpha} \oplus \upiota_\Hmr(\upepsilon)) &= &  \upiota_\Hmr(\upalpha)\otimes \Umr_{\upalpha} \oplus \upiota_\Hmr(\upalpha \otimes \upepsilon) \\
& = & \upiota_\Hmr(\upalpha)\otimes \Umr_{\upalpha} \oplus \upiota_\Hmr((\upbeta \oplus \upbeta') \otimes \upepsilon) \\
&=& \upiota_\Hmr(\upalpha)\otimes \Umr_{\upalpha} \oplus \upiota_\Hmr(\upbeta\otimes\upepsilon) \oplus \upiota_\Hmr(\upbeta' \otimes\upepsilon),
\end{eqnarray*} 
 $\Fmr[\upbeta]^{(\upepsilon)}_{\Gmr/\Hmr}(\Umr)$ is equivalent to $\Omega^{\infty}_{\upiota_\Hmr(1)}$ to the homotopy fiber of 
\[ 
(\Smr^{\upiota_\Hmr(\upalpha )\otimes (\Umr_{\upalpha} \oplus \epsilon)} \sma \upchi(\Umr_{\upalpha}^{\perp}))_{\mr{h} {\sf O}(\upiota_\Hmr(\upbeta'))} \to
(\Smr^{\upiota_\Hmr(\upalpha) \otimes \Umr_{\upalpha} \oplus \upiota_\Hmr(\upbeta' \otimes \upepsilon))} \sma \upchi(\Umr_{\upalpha}^\perp))_{\mr{h} {\sf O}(\upiota_\Hmr(\upbeta'))}. \]
Since the homotopy fiber of
$
\Sigma^{\infty}_{(1)} \Smr^0 \longrightarrow  \Sigma^{\infty}_{(1)} \Smr^\Vmr 
$
 is equivalent to the sphere spectrum $\Sigma^{\infty}_{(1)} \Smr(\Vmr)_+$ in $\Sp^{\Gmr}_{(1)}$  (assuming $\Vmr$ in $\Ucal_{\Gmr}^{(1)}$ and finite)\footnote{Here $\Smr(\Vmr)$ denotes the unbased unit sphere in $\Vmr$. }, it follows that
\[ 
\Fmr[\upbeta]^{(\upepsilon)}_{\Gmr/\Hmr}(\Umr) \simeq \Omega_{\upiota_\Hmr(1)}^{\infty} 
[
(\Smr(\upiota_{\Hmr}(\upbeta' \otimes \upepsilon))_+ \sma \Smr^{\upiota_\Hmr(\upalpha) \otimes \Umr_{\upalpha}} \sma \upchi(\Umr_{\upalpha}^\perp))_{\mr{h} {\sf O}(\upiota_\Hmr(\upbeta'))}.
]
\]
Now we  set ${\bf X} = \Smr^{\upiota_\Hmr(\upalpha) \otimes \Umr} \sma \upchi(\Umr_{\upalpha}^{\perp})$ and observe that  
\begin{eqnarray*}
 (\Smr(\upiota_\Hmr(\upbeta' \otimes \upepsilon))_{+} \sma {\bf X})_{\mr{h} \mr{O}(\upiota_\Hmr(\upbeta'))} &\simeq& \mr{ EO}(\upiota_\Hmr(\upbeta'))_{+}\sma_{\mr{h} \mr{O}(\upiota_\Hmr(\upbeta'))} (\Smr(\upiota_\Hmr(\upbeta' \otimes \upepsilon))_{+} \sma {\bf X}) \\
 &\simeq& \mr{EO}(\upiota_\Hmr(\upbeta'))_{+}\sma_{\mr{h} \mr{O}(\upiota_\Hmr(\upbeta'))} (\mr{O}(\upiota_\Hmr(\upbeta'))/\mr{O}(\upiota_\Hmr(\upbeta' - \upepsilon))_{+} \sma {\bf X}) \\
 &\simeq& \mr{EO}(\upiota_\Hmr(\upbeta'))_{+}\sma_{\mr{O}(\upiota_\Hmr(\upbeta' - \upepsilon))} {\bf X} \\
 &\simeq& \mr{EO}(\upiota_\Hmr(\upbeta' - \upepsilon))_{+}\sma_{\mr{O}(\upiota_\Hmr(\upbeta' - \upepsilon))} {\bf X} \\
 &\simeq& {\bf X}_{\mr{h} \mr{O}(\upiota_\Hmr(\upbeta' - \upepsilon))}
\end{eqnarray*}
which completes the inductive step, and therefore, the proof of \eqref{eqn:step1}.

To complete the argument we must show that  $\td{\bf T}_{\upalpha} \Fmr$,  the $\upalpha$-th reduced tower of $\Fmr$,  is trivial. Next, we observe that 
$\Fmr_{\Gmr/\Hmr}(\Umr)$ is  at least $\sfc$-connected, where for $\Kmr \subset \Hmr$
\[
\sfc(\Kmr) \geq | \upalpha^{\Kmr} | | \Umr_{\upalpha}^{\Kmr}|  + |(\upalpha^{\Kmr})^{\perp}| | (\Umr_{\upalpha}^{\Kmr})^{\perp} | + \nu(\Kmr),
\]
and $\nu(\Kmr)$ is the uniform lower bound on the connectivity of $\upchi(\Umr^\perp_\upalpha)$. Then, by setting $\Dmr$ as the trivial functor in \Cref{lem:tauconnect}, we conclude that $\uptau_{\upalpha}\Fmr_{\Gmr/\Hmr}(\Umr)$ is $(\sfc +1)$-connected. Consequently, $\td{\bf T}_{\upalpha} \Fmr$ is pointwise contractible as desired. 
\end{proof}

\begin{notn} Let $\Jfrak_{\Gmr}^{(1)}: \Ocal_{\Gmr}^{\mr{op}} \longrightarrow \Cat$ denote the presheaf that sends $\Gmr/\Hmr$ to the full-subcategory of $\Rep_\Hmr$ generated by restrictions of $1$-dimensional $\Gmr$-representations, and let $\Ecal_{\Gmr, {\bf 0}}^{(1)}$ denote the category of $1$-morphisms $\Emr: \Jfrak_{\Gmr}^{(1)} \longrightarrow \Tfrak_{\Gmr}$.  
\end{notn}
Observe that our definition of $\upalpha$-homogeneous, i.e. \Cref{defn:vhomo} (also see \Cref{defn:vpoly}), makes sense for objects in $\Ecal_{\Gmr, {\bf 0}}^{(1)}$. We utilize this to establish a classification theorem which can be thought of as an equivariant analog of the classification theorem of homogeneous functors (see \cite[Theorem 7.3]{WeissCalc}). 
\begin{thm}[{\bf Classification of homogeneous functors}] \label{thm:classification}
Let $\upalpha$ be a direct sum of one-dimensional $\Gmr$-representations.
Suppose $\Fmr\in \Ecal_{\Gmr, {\bf 0}}$ is $\upalpha$-homogeneous and ${\bf \Theta} \Fmr^{(\upalpha)}$,  its $\upalpha$-th derivative at $\infty$,  is uniformly bounded below. Then 
\[
\Fmr_{\Gmr/\Hmr}(\Umr) = 
  \Omega^{\infty}_{\upiota_\Hmr(1)}\left[ (\Smr^{ \upiota_\Hmr(\upalpha) \otimes \Umr_{\upalpha}}  \sma \upchi(\Umr_{\upalpha}^{\perp}))_{\mr{h} {\sf O}(\upiota_\Hmr(\upalpha))}\right],
\]
where $\upchi = {\bf \Theta}^\# \Fmr^{(\upalpha)} \simeq {\bf \Theta} \Fmr^{(\upalpha)}$. Conversely, any functor of the above form is $\upalpha$-homogeneous.
\end{thm}

\begin{proof}
The converse is already established as \Cref{thm:keyex}. 

 Now consider  an $\upalpha$-homogeneous  Weiss functor $\Fmr\in \Ecal_{\Gmr, {\bf 0}}$. Define $\Emr \in \Ecal_{\Gmr, {\bf 0}}$ 
\[
\Emr_{\Gmr/\Hmr}(\Umr) := \Fmr_{\Gmr/\Hmr}^{(\alpha)}(\Umr)_{\mr{h} {\sf O}(\upiota_{\Hmr}(\upalpha))}
\]
as the homotopy orbits of the $\upalpha$-th derivative of $\Fmr$. Since $\Fmr^{(\alpha)}$ is a stable symmetric object in $\Ecal_{\Gmr, \upalpha}$, \Cref{thm:stablesymmetric} implies that  it is of the form
\[
\Fmr^{(\alpha)}_{\Gmr/\Hmr}(\Umr) = (\Smr^{\upiota_\Hmr( \upalpha) \otimes \Umr_{\upalpha}}  \sma \upchi(\Umr_{\upalpha}^{\perp}))
\]
where $\upchi = {\bf \Theta}^\# \Fmr^{(\upalpha)} \simeq {\bf \Theta} \Fmr^{(\upalpha)}$ (see \Cref{notn:pound} and \eqref{thetapound}). Therefore, by \Cref{thm:htpyorbit}, 
\begin{equation} \label{TEtoD}
 {\bf T}_{(\alpha, \vec{v})} \Emr \simeq \Dmr,
 \end{equation}
where $\Dmr$ is the functor  given by  
\[
\Dmr_{\Gmr/\Hmr}(\Umr) :=  \Omega^{\infty}_{\upiota_\Hmr(1)} \left[ (\Smr^{ \upiota_\Hmr(\upalpha) \otimes \Umr_{\upalpha}}  \sma \upchi(\Umr_{\upalpha}^{\perp}))_{\mr{h} {\sf O}(\upiota_\Hmr(\upalpha))} \right],
\]
 where $\vec{v}$ is the full direction vector. Since $\Fmr$ is $\upalpha$-homogeneous, it is fully $\upalpha$-polynomial, and therefore, if we show that 
\[ 
{\bf T}_{(\upalpha, \vec{v})}\Fmr^{(\upalpha)} \simeq {\bf T}_{(\upalpha, \vec{v})}\Emr^{(\upalpha)}
\] 
then $\Fmr$ will be equivalent to $\Dmr$ as desired.

Note that ${\bf R}_{\bf 0}^{\alpha}\Fmr^{\upalpha} \longrightarrow \Fmr$ can be viewed as a $\Omr(\upalpha)$-equivariant map where the codomain is given the trivial action, it follows that  $u$  factors through a map $u:\Emr \longrightarrow \Fmr$. Consequently,  we have a commutative diagram (also see the proof of \cite[Theorem 7.3]{WeissCalc})
\begin{equation} \label{keyad}
\begin{tikzcd}
  {\bf R}_{\bf 0}^{\alpha}\Fmr^{\upalpha} \ar[r] \dar["\mr{id}"'] & \Emr \ar[r] \ar[d, "u"'] & {\bf T}_{(\alpha, \vec{v})}\Emr \ar[d, "{\bf T}_{(\alpha, \vec{v})}(u)"] \\
 {\bf R}_{\bf 0}^{\alpha}\Fmr^{\upalpha} \ar[r] & \Fmr \ar[r] & {\bf T}_{(\upalpha, \vec{v})}\Fmr
\end{tikzcd}
\end{equation}
in which we will show that the $\upalpha$-th derivative of ${\bf T}_{(\alpha, \vec{v})}(u)$ is an equivalence. To this end, we consider the zig-zag 
\[
\begin{tikzcd}
{\bf T}_{(\upalpha, \vec{v})}\Emr^{(\upalpha)} & \lar["(\Amr)"'] \Fmr^{(\upalpha)} \rar["(\Bmr)"] & {\bf T}_{(\upalpha, \vec{v})}\Fmr^{(\upalpha)},
\end{tikzcd}
 \]
where $(\Amr)$  and $(\Bmr)$ are  the adjoint of the top and bottom horizontal compositions in \eqref{keyad}. The map $(\Amr)$ is clearly an equivalence because  $\Fmr$ is $\upalpha$-homogeneous, and $(\Bmr)$ is an equivalence because 
\[ 
{\bf T}_{(\upalpha, \vec{v})}\Emr^{(\upalpha)} \simeq {\bf T}_{(\upalpha, \vec{v})}\Dmr^{(\upalpha)} \simeq \Dmr[\upalpha] = \Fmr^{(\upalpha)}
\]
due to \eqref{TEtoD} and \Cref{thm:keyex}. Hence, the result.
\end{proof}

\bigskip
\begin{proof}[\bf Proof of \Cref{main3}. ] Since ${\bf L}_{\upalpha}\Emr$ is $\upalpha$-homogeneous,  the result follows immediately from \Cref{thm:classification}. 
\end{proof}
\begin{notn} Let $\Ecal_{\Gmr, {\bf 0}}^{\Sigma, {\sf conn}}$ denote the full subcategory of $\Ecal_{\Gmr, {\bf 0}}^{\Sigma, {\sf conn}}$ (as in \Cref{notn:symE})  generated by objects that are connected at infinity (see \Cref{defn:conn}). Let $\wh{{\bf Sys}}^{\upalpha}_{\sf unif}$ denote the category of uniformly bounded below $\upalpha$-systems with a n\"aive action of $\Omr(\upalpha)$. 
\end{notn}
In the classical Weiss calculus, the $n$-th derivative of the $n$-layer $\Lmr_n\Emr$ of a Weiss functor
$
 \Emr: \Jcal^\RR \longrightarrow \Top_*$
at $\infty$, is equivalent to the $n$-th derivative of $\Emr$ at $\infty$, i.e., 
\[ 
\uptheta \Lmr_n\Emr \simeq \uptheta \Emr^{(n)}
\] 
as spectra. The above equivalence was made precise through model theoretic approach by Barnes and Oman in \cite{BO}. Following their ideas, one may consider a model structure on $\Ecal_{\Gmr, \upalpha}^{\Sigma, {\sf conn}}$  in which a map $\Emr \longrightarrow \Fmr$ is a weak equivalence if the induced map \[ ({\bf T}_{\upalpha, \vec{v}} \Emr)^{(\upalpha)} \longrightarrow ({\bf T}_{\upalpha, \vec{v}} \Emr)^{(\upalpha)}  \]
is a pointwise equivalence. Then it will follow  that 
\[ \hat{\bf \Theta}(\Emr^{(\upalpha)}):= {\bf \Theta}({\bf L}_{\upalpha}\Emr^{(\upalpha)}) \simeq {\bf \Theta}(\Emr^{(\upalpha)})  \]  
provided the following conjecture is true:
\begin{conj} \label{conj:symsys} The homotopy category of $\Ecal_{\Gmr, {\bf 0}}^{\Sigma, {\sf conn}}$ and $\wh{{\bf Sys}}^{\upalpha}_{\sf unif}$ are isomorphic:
\[ {\bf Ho}(\Ecal_{\Gmr, {\bf 0}}^{\Sigma, {\sf conn}}) \cong {\bf Ho}(\wh{{\bf Sys}}^{\upalpha}_{\sf unif}). \]
\end{conj}

\bibliographystyle{amsalpha}
\bibliography{EWeiSS.bib}
\end{document}